\documentclass[11pt,twoside]{article}

\usepackage[english]{babel}

\usepackage{amsmath}
\usepackage{amssymb}

\usepackage{mathrsfs}
\usepackage{amsthm}

\usepackage{indentfirst}
\usepackage{color}
\usepackage{txfonts}

\usepackage{anysize}

\usepackage[colorlinks=true,
  linkcolor=blue,
  citecolor=red,
  urlcolor=magenta,
  ]{hyperref}

\allowdisplaybreaks

\newtheorem{theorem}{Theorem}[section]
\newtheorem{lemma}[theorem]{Lemma}

\newtheorem{proposition}[theorem]{Proposition}

\theoremstyle{definition}
\newtheorem{remark}[theorem]{Remark}
\newtheorem{definition}[theorem]{Definition}

\numberwithin{equation}{section}

\begin{document}

\title{\bf\Large Atomic Characterization and Its Applications of
Matrix-Weighted Variable Hardy Spaces
\footnotetext{\hspace{-0.35cm} 2020 {\it
Mathematics Subject Classification}. Primary 42B30; Secondary 42B25, 42B20, 47A56,
46E40, 46E35, 42B35. \endgraf
{\it Key words and phrases}. variable Hardy space, variable matrix weight,
convex-body maximal function, atom,
dual space, Calder\'on--Zygmund operator.
}}
\date{}
\author{}
\author{Yiqun Chen, Dachun Yang\footnote{Corresponding author,
E-mail: \texttt{dcyang@bnu.edu.cn}/{\color{red}\today}/Final version.},\ \ Wen Yuan and
Zongze Zeng}
\maketitle

\vspace{-0.8cm}

\begin{center}
\begin{minipage}{13cm}
{\small {\bf Abstract:}\quad
In this article, by means of the matrix-weighted grand maximal function
we first introduce the variable Hardy space
$H^{p(\cdot)}_W$ on $\mathbb{R}^n$ with the $\mathscr{A}_{p(\cdot),\infty}$
matrix weight $W$ and with the variable exponent $p(\cdot)$
having globally log-H\"older continuity, and then via using several
different convex body valued maximal
functions we establish its various maximal function equivalent
characterizations. Combining a refined Whitney decomposition with
both the convex body valued maximal function and its corresponding
convex-body reducing operator, we obtain the atomic characterization
of $H^{p(\cdot)}_W$. As applications, we give its dual space
and establish the boundedness of Calder\'on--Zygmund operators
from $H^{p(\cdot)}_W$ to the matrix-weighted variable Lebesgue space
$L^{p(\cdot)}_W$ and to itself.
This approach to establishing atomic characterization
differs from all previous ones.
}
\end{minipage}
\end{center}

\vspace{0.15cm}

\tableofcontents

\vspace{0.15cm}

\section{Introduction}

The study of matrix weights can be traced back to
the work of Wiener and Masani \cite{wm58} to develop
the prediction theory of multivariate stochastic processes. In the 1990s,
motivated by problems concerning the angle between the past and the future of
multivariate random stationary processes and the boundedness of inverses of Toeplitz
operators, Treil and Volberg \cite{tv97} introduced the appropriate matrix-valued
analogue of the Muckenhoupt $A_2$ condition on ${\mathbb{R}^n}$. Later, Nazarov and Treil
\cite{nt96}, and independently Volberg \cite{v97} by a different approach,
extended this theory to matrix $A_p$ weights with $p\in(1,\infty)$. Since then,
the theory of matrix weights has attracted considerable attention. In particular,
Christ and Goldberg \cite{cg01,g03} obtained the boundedness of certain maximal
operators and Calder\'on--Zygmund operators adapted to matrix $A_p$ weights.
Nazarov et al. \cite{nptv17} obtained the boundedness of Calder\'on--Zygmund
operators on matrix-weighted $L^2_W$ with operator norm controlled by
$[W]_{A_2}^{\frac 32}$, and Domelevo et al. \cite{dptv24} later proved that the
exponent $\frac 32$ is indeed sharp.
In addition, Bownik and Cruz-Uribe
\cite{bc22} extended the Jones factorization theorem and the Rubio de Francia
extrapolation theorem to matrix Muckenhoupt weights.
Moreover, matrix weights arise naturally in partial differential equations
\cite{cmr,IM} and in the theory of multivariate stationary processes
\cite{tv97,tv99}. We refer to
\cite{dhl20,dkps24,llor24,llor24 2,v24} for recent studies on the boundedness
of operators on matrix-weighted Lebesgue spaces and to
\cite{kn26,n13,nr18,ns21,ns25,ns26} for matrix weights on more general bases.
Besides matrix-weighted Lebesgue spaces, matrix-weighted theory
has also been developed for more refined function spaces.
A systematic study of matrix-weighted
Besov and Triebel--Lizorkin spaces was initiated by Frazier and Roudenko
\cite{fr04,fr21,r03,r04}. More recently, Bu et al.
\cite{bhyy23,bhyy23 2,bhyy23 3,bhyy24} developed the theory of
matrix-weighted Besov--Triebel--Lizorkin-type spaces.
For more results on other matrix-weighted function spaces,
we refer to \cite{byyz25,cyy25,cp23,kg25,n25,yyz25 2}.
It is then natural to develop a corresponding Hardy
space theory in the matrix-weighted setting,
which is the main target of the present article.

Recall that the Hardy space was introduced by Stein and Weiss
\cite{sw60} and was further developed by Fefferman and Stein in their seminal
work \cite{fs72}. In the scalar weighted setting,
Garc\'ia-Cuerva \cite{g79}, Bui \cite{b81}, and Str\"omberg and Torchinsky
\cite{st89} established a systematic theory of weighted Hardy spaces. Related
weighted and generalized Hardy-type spaces, together with their applications,
were further studied by Bonami et al. \cite{bnv10,bnv11}, Bui \cite{b14},
Ky \cite{ky14}, Yang et al. \cite{ylk17}, Ho \cite{h13,h17,h19}, Cruz-Uribe and Wang
\cite{cw14}, Nakai and Sawano \cite{ns12}, Sawano et al. \cite{shyy17}, and
Izuki et al. \cite{inns23}.
Hardy spaces and their variants have also played an important role in harmonic
analysis and partial differential equations; see, for example,
\cite{bdt21,clms93,cw77,gnns19,gnns21,yyz21,zyyw21} and
\cite{bdl18,bdl20,m94}, as well as the recent developments in
\cite{bd25,bd25 2,bhh23,bhh24,bl24}. For comprehensive introductions
of Hardy spaces and real-variable methods, we refer to the monographs
\cite{d01,gcrdf85,g14,g14 2,mc97,s93,t86}.
In particular, motivated by the aforementioned classical 
real-variable theory of Hardy spaces, Bu et al.
\cite{bcyy24} recently initiated the study of matrix-weighted Hardy spaces
with the matrix $A_p$ weight which, when $p\in (0,1]$ and $m=1$,
coincides with the scalar Muckenhoupt
$A_1$ weight.

On the other hand, the theory of variable Muckenhoupt $\mathcal{A}_{p(\cdot)}$ weights was
initiated by Cruz-Uribe et al. \cite{cdh11}. Subsequently,
Cruz-Uribe and Penrod \cite{cp23} introduced variable
matrix $\mathscr{A}_{p(\cdot)}$ weights,
while Yang et al. \cite{yyz25}
further extended them to variable matrix
$\mathscr{A}_{p(\cdot),\infty}$ weights.
We refer to \cite{cf13,cfn12,dhr17,np25} for more studies
on variable Lebesgue spaces and their corresponding (matrix) weights
and to \cite{ah10,dhr09,wgx24,x08} for more studies on variable exponent
function spaces. Variable Hardy spaces were introduced by Nakai and
Sawano \cite{ns12} and Cruz-Uribe and Wang \cite{cw14}, and weighted variable
Hardy spaces were later studied by Ho \cite{h17}. We refer to
\cite{cmn19,cmn20,cn21,h19,inns23} for further developments on variable Hardy
spaces and related operator estimates.

In this article, by means of the matrix-weighted grand maximal function
we first introduce the variable Hardy space
$H^{p(\cdot)}_W$ on $\mathbb{R}^n$ with the $\mathscr{A}_{p(\cdot),\infty}$
matrix weight $W$ and with the variable exponent $p(\cdot)$
having the globally log-H\"older continuity, and then via using several
different convex body valued maximal
functions we establish its various maximal function equivalent
characterizations. Combining a refined Whitney decomposition with
both the convex body valued maximal function and its corresponding
convex-body reducing operator, we obtain the atomic characterization
of $H^{p(\cdot)}_W$. As applications, we give its dual space
and establish the boundedness of Calder\'on--Zygmund operators
from $H^{p(\cdot)}_W$ to the matrix-weighted variable Lebesgue space
$L^{p(\cdot)}_W$ and to itself.
It is worth mentioning that, even when $p(\cdot)\equiv p$ with $p\in(0,1]$
is a constant exponent, this article also extends the atomic characterization
and its applications in \cite{bcyy24,cyy25} from matrix $A_p$ weights
to matrix $A_{p,\infty}$ weights,
which in the scalar case is exactly the extension
from the Muckenhoupt $A_1$ weight to the Muckenhoupt $A_\infty$ weight.

This generalization needs an improved approach for the atomic
decomposition, which is essentially different from the one used in \cite{bcyy24}.
Indeed, the key of the approach used in \cite{bcyy24} is the level-set construction.
A level set is determined by single or finite scalar quantities related
to vector-valued functions under consideration and therefore records
only their sizes, but not directions. However, in the matrix-weighted setting,
the direction of vector-valued functions under consideration matters under the action
of the matrix weight. Thus, the level-set construction has to involve the whole
matrix-weighted object such as $W\vec f$, which in turn leads to a stronger
assumption of the matrix weight under consideration,
namely $W\in A_p$ in \cite{bcyy24}. We escape this by introducing the convex-body
reducing operator, which represents the local average of convex body valued
functions under consideration as a positive definite matrix, i.e, a discretization
to a certain extent.
This average serves as a vector-valued
substitute for the scalar level in the classical level-set method. Thus,
instead of comparing a pointwise quantity with a number such as $2^j$,
in this article we introduce the level set $E_Q$ in \eqref{eq}
with $Q$ being a dyadic cube via comparing the pointwise convex-valued function with
its local average on a given cube. This, combined with
the stopping-time construction appearing in the proof
of \cite[(2-12)]{ccdo17} (namely a refined Whitney decomposition,
which is formulated as Lemma \ref{stop coll}), yields the atomic
characterization of $H^{p(\cdot)}_W$, which allows us to weaken the
assumption on matrix weights from $A_p$ in \cite{bcyy24} to
$\mathscr{A}_{p(\cdot),\infty}$ in the present article.
This approach to establishing atomic characterization
differs from all previous ones.

In addition, as mentioned above we introduce several convex body valued maximal functions,
which provide a convenient way to separate the matrix action from
the construction of the matrix-weighted maximal functions. This also plays a crucial
role in establishing the atomic characterization and the mutual
equivalences among various maximal function characterizations of $H^{p(\cdot)}_W$.
Moreover, Lemmas \ref{f dec eq le} and \ref{a atom eq le} serve as substitutes
of the Fefferman--Stein type vector-valued inequality in the scalar case, which is commonly used in
the variable Hardy space theory. These two estimates are frequently used in the proofs of the
atomic characterization and the boundedness of Calder\'on--Zygmund operators.

The organization of the remainder of this article is as follows.

Section \ref{sec weight} consists of three subsections.
Subsection \ref{vlspace} recalls some basic facts on variable
Lebesgue spaces, Subsection \ref{sec weight 2}
collects some necessary properties of variable matrix ${\mathscr A}_{p(\cdot),\infty}$
weights, and Subsection \ref{sec convex} recalls the elementary notions on
convex body valued functions.

In Section \ref{sec Hardy}, we recall the definition of
several matrix-weighted maximal functions (see Definition \ref{def all M}),
and, moreover, we introduce the corresponding convex body valued maximal functions
(see Definition \ref{def convex max}) and show their measurability (see Lemma \ref{meas M}).
After that, we introduce the matrix-weighted variable Hardy space $H^{p(\cdot)}_W$
in terms of the matrix-weighted grand maximal function (see Definition \ref{def Hardy})
and establish its various equivalent characterizations (see Theorem \ref{H equal}).
Furthermore, as applications, we give several fundamental properties of $H^{p(\cdot)}_W$,
such as embedding into the space of tempered distributions and its completeness
(see Propositions \ref{H embed} and \ref{H comp}).
Finally, we prove that $H^{p(\cdot)}_W$
coincides with the matrix-weighted variable Lebesgue space $L^{p(\cdot)}_W$
when $p(\cdot) \in {\mathcal P}\cap LH$ and $W\in {\mathscr A}_{p(\cdot)}$ (see Theorem \ref{H p W}).

In Section \ref{sec atom}, we first introduce a natural variant of classical atoms in the
variable matrix-weighted setting (see Definition \ref{def atom})
and then establish the atomic characterization of $H^{p(\cdot)}_W$
(see Theorem \ref{atom con}).
To achieve this goal, we first
give two vector-valued inequalities for $H^{p(\cdot)}_W$
(see Lemmas \ref{f dec eq le} and \ref{a atom eq le}),
which are substitutes for the commonly used Fefferman--Stein inequalities
on variable Hardy spaces.
Then we prove a density result (namely Proposition \ref{H dense}),
which is used to obtain the other density result
(namely Proposition \ref{H dense O}).
Finally, we introduce the convex-body reducing operator
(see Lemma \ref{M u exist}).
With the aid of the convex body valued maximal function
and its corresponding convex-body reducing operator,
together with a refined Whitney decomposition Lemma \ref{stop coll},
Proposition \ref{H dense O}, and Lemmas \ref{f dec eq le} and \ref{a atom eq le},
we establish the atomic characterization of $H^{p(\cdot)}_W$.

In Section \ref{sec dual}, using the atomic characterization,
we prove that the reducing matrix-weighted variable Campanato space
(see Definition \ref{def Com}) is the dual space of $H^{p(\cdot)}_W$
(see Theorem \ref{dual Com}).

Finally, in Section \ref{sec CZ}, applying the atomic characterization,
we show that, for any $p(\cdot)\in {\mathcal P}_0\cap LH$
and $W\in {\mathscr A}_{p(\cdot),\infty}$, Calder\'on--Zygmund operators
are bounded from $H^{p(\cdot)}_W$ to $L^{p(\cdot)}_W$ and from $H^{p(\cdot)}_W$ to itself
(see Theorem \ref{bound CZ}).

We end this introduction by making some notational conventions.
Throughout this article, we work in $\mathbb{R}^n$
and, unless otherwise specified, we always take $\mathbb{R}^n$
as the default underlying space.
Let $\mathbb{Z}$ be the collection of all integers,
$\mathbb{Z}_+:=\{0,1,\dots\}$, $\mathbb{N} := \{1,2,\dots\}$,
and $\mathbb{Q}$ be the set of all rational numbers.
For any $\gamma := (\gamma_1,\dots,\gamma_n)\in \mathbb{Z}^n_+$,
let $|\gamma| := \gamma_1 + \cdots + \gamma_n$ and, for any
$x := (x_1,\dots,x_n) \in \mathbb{R}^n$, let $x^\gamma := x_1^{\gamma_1}\cdots x_n^{\gamma_n}$
and $D^\gamma := (\frac{\partial}{\partial x_1})^{\gamma_1}\cdots(\frac{\partial}{\partial x_n})^{\gamma_n}$.
For any measurable set $E$ in $\mathbb{R}^n$,
denote by the \emph{symbol $\mathscr{M}(E)$} the set
of all measurable functions on $E$
and, when $E = \mathbb{R}^n$, simply write $\mathscr{M}(\mathbb{R}^n)$ as $\mathscr{M}$.
In addition, we use the symbol $L^p_{\rm loc}$ with $p\in (0,\infty)$ to denote
the set of all locally $p$-integrable functions on $\mathbb{R}^n$
and use the symbol $C^{\infty}_{\rm c}$ to denote the set of
all infinitely differentiable functions on $\mathbb{R}^n$ with compact support.
For any $x\in\mathbb{R}^n$ and $r\in (0,\infty)$,
the \emph{open ball $B(x,r)$} is defined to be the set $\{y\in\mathbb{R}^n:\ |x-y|< r\}$ and let
$\mathbb{B} := \{B(x,r):\ x\in\mathbb{R}^n\ \text{and}\ r\in(0,\infty)\}.$
A \emph{cube} $Q$ in $\mathbb{R}^n$ always has finite edge length
and edges of cubes are always assumed to be parallel to the coordinate axes,
but $Q$ is not necessary to be open or closed.
For any cube $Q$ in $\mathbb{R}^n$,
we use $l(Q)$ to denote its edge length and $c_Q$ its center.
If $E$ is a measurable set in $\mathbb{R}^n$,
then we denote by $\mathbf{1}_E$ its \emph{characteristic function}
and, for any bounded measurable set $E\subset \mathbb{R}^n$ with $|E| \neq 0$
and for any $f\in L^1_{\rm loc}$,
let $$\fint_E f(x)\,dx := \frac{1}{|E|} \int_E f(x)\,dx.$$
For any $p\in [1,\infty]$, let $p'$ be its conjugate number,
that is, $\frac1p+\frac{1}{p'} = 1$.
We always use $C$ to denote a positive constant
independent of the main parameters involved,
but it may vary from line to line.
The notation $f\lesssim g$ means $f\leq  Cg$
and, if $f\lesssim g\lesssim f$, we then write $f\sim g$.
The \emph{notation $s\to0^+$} means that there exists
$c_0\in (0,\infty)$ such that $s\in(0,c_0)$ and $s\to0$.
Finally, in all proofs we
consistently retain the notation introduced
in the original theorem (or related statement).

\section{Preliminaries}\label{sec weight}

This section is devoted to some necessary preliminaries.
We first recall some basic properties of variable Lebesgue spaces (Subsection \ref{vlspace}),
then recall some results of matrix $\mathscr{A}_{p(\cdot),\infty}$ weights (Subsection \ref{sec weight 2}),
and finally present several elementary properties of convex body valued functions (Subsection \ref{sec convex}).
\subsection{Variable Lebesgue Spaces}\label{vlspace}
We begin with the definition of exponent functions.
A measurable function $p:\ \mathbb{R}^n\to(0,\infty]$ is called an \emph{exponent function}.
We use the \emph{symbol $\mathcal{P}$} to denote the set of all exponent functions $p:\ \mathbb{R}^n \rightarrow [1,\infty]$,
and we use the \emph{symbol $\mathcal{P}_0$} to denote the set of all exponent functions $p:\ \mathbb{R}^n \rightarrow (0,\infty]$
satisfying $\mathop{\rm{ess}\inf}_{x\in \mathbb{R}^n} p(x) > 0$.
For any $p(\cdot)\in \mathcal{P}_0$ and any measurable set $E$ in $\mathbb{R}^n$,
let
$$ p_-(E) := \mathop{\rm{ess}\inf}_{x\in E} p(x)\ \  \text{and}\ \  p_+(E) := \mathop{\rm{ess}\sup}_{x\in E} p(x); $$
moreover, write $p_- := p_-(\mathbb{R}^n)$ and $p_+ := p_+(\mathbb{R}^n)$.

Then we recall the definition of variable Lebesgue spaces
(see, for instance, \cite[Definition 2.16]{cf13}).
\begin{definition}\label{def Leb}
The \emph{variable Lebesgue space $L^{p(\cdot)}$} associated
with $p(\cdot)\in \mathcal{P}_0$
is defined to be the set of all $f\in\mathscr{M}$ such that
$$ \|f\|_{L^{p(\cdot)}} := \inf\left\{ \lambda\in (0,\infty):\ \rho_{L^{p(\cdot)}}\left(\frac{f}{\lambda}\right) \leq 1 \right\} < \infty, $$
where $\rho_{L^{p(\cdot)}}$ is the \emph{variable exponent modular}
defined by setting
$$\rho_{L^{p(\cdot)}} (f) :=  \int_{\mathbb{R}^n\setminus \Omega_\infty} \left|f(x)\right|^{p(x)}\,dx
+ \mathop{\rm{ess}\sup}_{x\in \Omega_\infty} |f(x)| $$
with $\Omega_\infty := \{x\in\mathbb{R}^n:\ p(x) = \infty\}$.
\end{definition}

The following log-H\"older continuous condition of variable exponents
(see, for instance, \cite[Definition 2.2]{cf13}) is frequently used in the theory of variable function spaces.

\begin{definition}\label{def LH}
A measurable real-valued function $r$ on $\mathbb{R}^n$ is said to be
\emph{locally log-H\"older continuous},
denoted by $r(\cdot) \in LH_0$,
if there exists a positive constant $C_0$
such that, for any $x,y\in\mathbb{R}^n$ with $|x-y| < \frac12$,
\begin{align*}
|r(x)-r(y)| \leq -\frac{C_0}{\log|x-y|}.
\end{align*}
A measurable real-valued function $r$ on $\mathbb{R}^n$
is said to be \emph{log-H\"older continuous at infinity},
denoted by $r(\cdot) \in LH_\infty$,
if there exist positive constants $r_\infty$ and $C_{\infty}$
such that, for any $x\in \mathbb{R}^n$,
\begin{align*}
|r(x)-r_\infty| \leq \frac{C_\infty}{\log(e+|x|)}.
\end{align*}
Furthermore, a measurable real-valued function $r$  on $\mathbb{R}^n$ is said to be
\emph{globally log-H\"older continuous},
denoted by $r(\cdot) \in LH$,
if $r(\cdot)$ is both locally log-H\"older continuous
and log-H\"older continuous at infinity.
\end{definition}

We now recall some basic properties of $L^{p(\cdot)}$ which are needed later.
In what follows, for any $p(\cdot) \in \mathcal{P}_0$ and any cube $Q$ in $\mathbb{R}^n$,
let $$p_Q:= \left[\fint_{Q} \frac{1}{p(x)}\,dx\right]^{-1}$$
and, for any $p(\cdot)\in\mathcal P$, the conjugate $p'(\cdot)$ of $p(\cdot)$ is defined to be the exponent function such that
$\frac{1}{p(x)} + \frac{1}{p'(x)} = 1$ for almost every $x\in{\mathbb{R}^n}$.
Also, if a positive constant $C$ depends on some indices associated with $p(\cdot)$ or,
more precisely, $C$ depends on some of $\{p_-,p_+,p_\infty,C_0,C_\infty\}$,
then we simply say that \emph{$C$ depends on $p(\cdot)$}.
The following lemma is precisely \cite[Theorem 4.5.7]{dhr17}.

\begin{lemma}\label{est Q}
Let $p(\cdot) \in \mathcal{P}\cap LH$.
Then, for any cube $Q$ in $\mathbb{R}^n$,
$$ \left\| \mathbf{1}_Q \right\|_{L^{p(\cdot)}}\sim\left| Q \right|^{\frac{1}{p_Q}},
\ \
\left\| \mathbf{1}_Q \right\|_{L^{p'(\cdot)}}\sim \left| Q \right|^{\frac{1}{p'_Q}},
\ \ \text{and}\ \
\left\| \mathbf{1}_Q \right\|_{L^{p(\cdot)}}\left\| \mathbf{1}_Q \right\|_{L^{p'(\cdot)}} \sim |Q|, $$
where the positive equivalence constants depend only on $p(\cdot)$ and $n$.
\end{lemma}

We recall the following H\"older's inequality in variable Lebesgue space,
which is exactly \cite[Theorem 2.26]{cf13}.

\begin{lemma}\label{Holder}
Let $p(\cdot) \in \mathcal{P}$.
If $f\in L^{p(\cdot)}$ and $g\in L^{p'(\cdot)}$,
then $fg\in L^1$ and, moreover,
$$ \int_{\mathbb{R}^n} \left| f(x)g(x) \right|\,dx \lesssim \|f\|_{L^{p(\cdot)}} \|g\|_{L^{p'(\cdot)}}, $$
where the implicit positive constant depends only on $p(\cdot)$.
\end{lemma}

As a consequence of Lemmas \ref{est Q} and \ref{Holder}, we immediately obtain
the following conclusion, which is precisely, for instance, \cite[Lemma 2.8]{yyz25}.

\begin{lemma}\label{est fQ}
Let $p(\cdot) \in \mathcal{P}\cap LH$.
Then, for any $f\in \mathscr{M}$ and any cube $Q$ in $\mathbb{R}^n$,
\begin{align*}
\fint_Q \left|f(x)\right|\,dx \lesssim \frac{1}{\|\mathbf{1}_Q\|_{L^{p(\cdot)}}} \left\| f \mathbf{1}_Q \right\|_{L^{p(\cdot)}},
\end{align*}
where the implicit positive constant depends only on $p(\cdot)$ and $n$.
\end{lemma}

The following one is exactly \cite[Theorem 2.34]{cf13}.

\begin{lemma}\label{fg Lp}
Let $p(\cdot) \in \mathcal{P}$.
Then, for any $f\in \mathscr{M}$,
$f\in L^{p(\cdot)}$ if and only if
$$ \left\|f\right\|'_{L^{p(\cdot)}} := \sup_{\|g\|_{L^{p'(\cdot)}} \leq 1}\int_{\mathbb{R}^n}  |f(x)g(x)| \,dx < \infty $$
and, moreover, $ \|f\|_{L^{p(\cdot)}} \sim \|f\|'_{L^{p(\cdot)}}$,
where the positive equivalence constants depend only on $p(\cdot)$.
\end{lemma}

On the following convexification of variable Lebesgue spaces,
we refer to,  for instance, \cite[Proposition 2.18]{cf13} and
\cite[Lemma 3.2.6]{dhr17}.

\begin{lemma}\label{con f}
Let $p(\cdot) \in\mathcal{P}_0$ with $p_+ < \infty$.
Then, for any $r \in (0,\infty)$ and $f\in\mathscr{M}$,
$ \|f\|_{L^{rp(\cdot)}} = \|\, |f|^r \|^\frac1r_{L^{p(\cdot)}} $.
\end{lemma}
\subsection{Variable Matrix ${\mathscr A}_{p(\cdot),\infty}$ Weights}\label{sec weight 2}
In this subsection, we recall some basic properties of
matrix $\mathscr{A}_{p(\cdot),\infty}$ weights obtained in our previous work \cite{yyz25}.
For any $m,n\in\mathbb{N}$, the set of all $m\times n$ complex-valued matrices
is denoted by the \emph{symbol $M_{m, n}$},
and $M_{m, m}$ is simply denoted by $M_m$.
For any $A\in M_m$, let
\begin{align*}
\|A\| := \sup_{\vec{z}\in\mathbb{C}^m, |\vec{z}| = 1} \left| A\vec{z} \right|.
\end{align*}
Then $(M_m, \|\cdot\|)$ is a Banach space.
Moreover, we have the following well-known result
(see, for instance, \cite[Lemma 2.3]{bhyy23}).

\begin{lemma}\label{norm matrix}
Let $A,B\in M_m$ be two nonnegative definite matrices.
Then $\|AB\| = \|BA\|$.
\end{lemma}

Now, we recall the concept of matrix weights
(see, for instance, \cite[Definition 2.7]{bhyy23}).
\begin{definition}\label{def matrix}
A matrix-valued function $W:\ \mathbb{R}^n \rightarrow M_m$ is called a \emph{matrix weight}
if $W$ satisfies that
\begin{itemize}
\item[{\rm (i)}] for almost every $x\in\mathbb{R}^n$, $W(x)$ is nonnegative definite,
\item[{\rm (ii)}] for almost every $x\in\mathbb{R}^n$, $W(x)$ is invertible,
\item[{\rm (iii)}] the entries of $W$ are all locally integrable.
\end{itemize}
\end{definition}

Next, we recall the definition of matrix-weighted variable Lebesgue spaces
(see, for instance, \cite[p.\,1135]{cp23}).
\begin{definition}
Let $p(\cdot)\in{\mathcal P}_0$ and $W$ be a matrix weight.
Then the \emph{matrix-weighted variable Lebesgue space $L^{p(\cdot)}_W$} associated
with $p(\cdot)$ is defined to be the set of all $\vec f\in(\mathscr{M})^m$ such that
$ \|f\|_{L^{p(\cdot)}_W} := \|\,|W(\cdot)\vec{f} |\,\|_{L^{p(\cdot)}} < \infty. $
\end{definition}

The following is the definition of matrix ${\mathscr A}_{p(\cdot)}$ weights,
which was first introduced in \cite[(1.2)]{cp23}.
\begin{definition}
Let $p(\cdot)\in{\mathcal P}$. A matrix weight $W$ on $\mathbb{R}^n$ is
called a \emph{matrix $\mathscr{A}_{p(\cdot)}$ weight} if
$$ \left[W\right]_{\mathscr{A}_{p(\cdot)}} := \sup_{Q}\frac1{|Q|} \left\|\, \left\|\, \left\| W(x)W^{-1}(\cdot) \right\| \mathbf{1}_Q \right\|_{L^{p'(\cdot)}}\mathbf{1}_{Q} \right\|_{L^{p(\cdot)}_x} < \infty, $$
where the supremum is taken over all cubes $Q$ in $\mathbb{R}^n$ and $L^{p(\cdot)}_x$ indicates to take the norm
with respect to the variable $x$.
\end{definition}
\begin{remark}\label{rem Ap}
Let $p(\cdot)\in{\mathcal P}$ with $p_+ < \infty$. Then, by the definition of ${\mathscr A}_{p(\cdot)}$,
we find that, for any $W\in {\mathscr A}_{p(\cdot)}$, $W^{-1} \in {\mathscr A}_{p'(\cdot)}$.
\end{remark}
Now, we recall matrix $\mathscr{A}_{p(\cdot),\infty}$ weights
introduced in \cite[Definition 1.1(ii)]{yyz25}.
\begin{definition}\label{def infty}
Let $p(\cdot)\in \mathcal{P}_0$.
A matrix weight $W$ on $\mathbb{R}^n$ is called a
\emph{matrix $\mathscr{A}_{p(\cdot),\infty}$ weight} if
$$ \left[W\right]_{\mathscr{A}_{p(\cdot),\infty}} := \sup_{Q} \exp\left( \fint_{Q} \log\left( \frac{1}{\|\mathbf{1}_Q\|_{L^{p(\cdot)}}} \left\|\, \left\| W(\cdot)W^{-1}(y) \right\| \mathbf{1}_Q \right\|_{L^{p(\cdot)}}
 \right)\,dy \right) <\infty , $$
where the supremum is taken over all cubes $Q$ in $\mathbb{R}^n$.
\end{definition}
In what follows, all positive constants related to matrix weights $W\in\mathscr A_{p(\cdot),\infty}$
mean that they depend only on $[W]_{\mathscr A_{p(\cdot),\infty}}$, rather that $W$ themselves.
\begin{remark}\label{rem Apinfty}
\begin{itemize}
\item[{\rm (i)}] If $p(\cdot) \equiv p$ is a constant exponent,
then, for any $W\in \mathscr{A}_{p,\infty}$,
the $p$-th power of $W$ is a matrix $A_{p,\infty}$ weight
(see, for example, \cite{bhyy24,bhyy23 1,v97} for the definition of $A_{p,\infty}$
weights and its various equivalent characterizations).

\item [{\rm (ii)}] From \cite[Theorem 3.1]{yyz25},
it follows that, for any scalar-valued weight $w$,
if $p(\cdot)\in \mathcal{P}_0\cap LH$, then
$w \in \mathscr{A}_{p(\cdot),\infty}$ if and only if $w^{p(\cdot)} \in A_\infty$.
There exists no analogue for matrix weights.
\end{itemize}
\end{remark}

Next, we recall the concept of reducing operators
for matrix $\mathscr{A}_{p(\cdot),\infty}$ weights,
which is precisely \cite[Definition 3.8]{yyz25}.

\begin{definition}\label{def reducing operator}
Let $p(\cdot) \in \mathcal{P}_0$ and $W$ be a matrix weight
and let $Q$ be any cube in $\mathbb{R}^n$.
The matrix $A_Q\in M_m$ is called a
\emph{reducing operator of order $p(\cdot)$ for $W$}
if $A_Q$ is positive definite and self-adjoint such that,
for any $\vec{z} \in \mathbb{C}^m$,
\begin{align}\label{eq redu}
\left| A_Q \vec{z} \right|
\sim \frac{1}{\|\mathbf{1}_Q\|_{L^{p(\cdot)}}} \left\|\, \left| W(\cdot) \vec{z} \right|  \mathbf{1}_{Q} \right\|_{L^{p(\cdot)}},
\end{align}
where the positive equivalence constants depend only on $m$ and $p(\cdot)$.
\end{definition}

The following lemma guarantees the existence of reducing operators of order $p(\cdot)$
for matrix weights, which is exactly \cite[Proposition 3.9]{yyz25}.

\begin{lemma}\label{ext redu}
Let $p(\cdot) \in \mathcal{P}_0$.
Then, for any matrix weight $W$ and any cube $Q$ in $\mathbb{R}^n$,
the reducing operator $A_Q$ of order $p(\cdot)$ for $W$ always exists.
\end{lemma}

The next lemma  extends \eqref{eq redu} from any vector $\vec{z}$ to any matrix $M\in M_m$,
which is precisely \cite[Lemma 3.10]{yyz25}.

\begin{lemma}\label{eq reduc M}
Let $p(\cdot) \in \mathcal{P}_0$ and $W$ be a matrix weight and
let $Q$ be any cube in $\mathbb{R}^n$.
If $A_Q$ is a reducing operator of order $p(\cdot)$ for $W$,
then, for any matrix $M\in M_m$,
$$ \left\| A_Q M \right\| \sim \frac{1}{\|\mathbf{1}_Q\|_{L^{p(\cdot)}}} \left\| \, \left\| W(\cdot) M  \right\|  \mathbf{1}_{Q} \right\|_{L^{p(\cdot)}}, $$
where the positive equivalence constants depend only on $m$ and $p(\cdot)$.
\end{lemma}

We also
recall the following concepts of  the lower and the upper $\mathscr{A}_{p(\cdot),\infty}$
weight dimensions introduced in \cite[Definition 3.21]{yyz25}.

\begin{definition}
Let $p(\cdot) \in \mathcal{P}_0$ and $d\in\mathbb{R}$.
A matrix weight $W$ is said to have \emph{$\mathscr{A}_{p(\cdot),\infty}$-lower dimension} $d$,
denoted by $W \in \mathbb{D}^{\rm lower}_{p(\cdot),\infty,d}$,
if there exists a positive constant $C$ such that,
for any $\lambda \in [1,\infty)$ and any cube $Q$ in $\mathbb{R}^n$,
\begin{align*}
\exp\left( \fint_{\lambda Q} \log\left( \frac{1}{\|\mathbf{1}_Q\|_{L^{p(\cdot)}}} \left\|\, \left\| W(\cdot)W^{-1}(y) \right\| \mathbf{1}_Q \right\|_{L^{p(\cdot)}}
 \right)\,dy \right) \leq C \lambda^d.
\end{align*}
A matrix weight $W$ is said to have \emph{$\mathscr{A}_{p(\cdot),\infty}$-upper dimension} $d$,
denoted by $W \in \mathbb{D}^{\rm upper}_{p(\cdot),\infty,d}$,
if there exists a positive constant $C$ such that,
for any $\lambda \in [1,\infty)$ and any cube $Q$ in $\mathbb{R}^n$,
\begin{align*}
\exp\left( \fint_{Q} \log\left( \frac{1}{\|\mathbf{1}_{\lambda Q}\|_{L^{p(\cdot)}}} \left\|\, \left\| W(\cdot)W^{-1}(y) \right\| \mathbf{1}_{\lambda Q} \right\|_{L^{p(\cdot)}}
 \right)\,dy \right) \leq C \lambda^d.
\end{align*}
\end{definition}

On $\mathscr A_{p(\cdot),\infty}$ weight dimensions, we
have the following basic properties,
which is exactly \cite[Proposition 3.22]{yyz25}.

\begin{proposition}\label{dim ext}
Let $p(\cdot) \in \mathcal{P}_0\cap LH$.
Then the following statements hold.
\begin{itemize}
\item[{\rm (i)}] For any $d \in (-\infty,0)$,
$\mathbb{D}^{\rm lower}_{p(\cdot),\infty,d} = \emptyset$
and $\mathbb{D}^{\rm upper}_{p(\cdot),\infty,d} = \emptyset$.
\item[{\rm (ii)}] For any $W \in \mathscr{A}_{p(\cdot),\infty}$,
there exists $d_1 \in [0,\frac{n}{p_-})$
such that $W \in \mathbb{D}^{\rm lower}_{p(\cdot),\infty,d_1}$.
\item[{\rm (iii)}] For any $W \in \mathscr{A}_{p(\cdot),\infty}$,
there exists $d_2 \in [0,\infty)$
such that $W \in \mathbb{D}^{\rm upper}_{p(\cdot),\infty,d_2}$.
\end{itemize}
\end{proposition}

Let $p(\cdot)\in \mathcal{P}_0\cap LH$.
Then, for any matrix weight $W \in \mathscr{A}_{p(\cdot),\infty}$,
let
\begin{align*}
d^{\rm lower}_{p(\cdot),\infty}(W) := \inf\left\{ d\in \left(0,\frac{n}{p_-}\right):\ W\ \text{has}\ \mathscr{A}_{p(\cdot),\infty}\text{-lower dimension}\ d \right\}
\end{align*}
and
\begin{align*}
d^{\rm upper}_{p(\cdot),\infty}(W) := \inf\left\{ d\in (0,\infty):\ W\ \text{has}\ \mathscr{A}_{p(\cdot),\infty}\text{-upper dimension}\ d \right\}.
\end{align*}

The upper and the lower $\mathscr{A}_{p(\cdot),\infty}$ weight dimensions play an important role in the
following estimate,
which is sharp and is precisely \cite[Lemma 3.27]{yyz25}
(see, for instance, \cite[Proposition 6.5]{bhyy23 1} for a similar result for matrix $A_{p,\infty}$ weights) and is frequently used below.

\begin{lemma}\label{QP5}
Let $p(\cdot) \in \mathcal{P}_0\cap LH$, $W \in \mathscr{A}_{p(\cdot),\infty}$,
$d_1 \in ( d^{\rm lower}_{p(\cdot),\infty}(W) ,\frac{n}{p_-})$,
$d_2 \in ( d^{\rm upper}_{p(\cdot),\infty}(W) ,\infty)$,
and $\{A_Q\}_{\mathrm{cube }Q}$ be a family of reducing operators of order $p(\cdot)$ for $W$.
Then, for any cubes $Q$ and $R$ in $\mathbb{R}^n$,
\begin{align*}
\left\| A_Q A_R^{-1} \right\| \lesssim \max\left\{ \left[ \frac{l(R)}{l(Q)} \right]^{d_1},
\left[ \frac{l(Q)}{l(R)} \right]^{d_2} \right\}\left[ 1+ \frac{|x_Q - x_R|}{\max\{l(Q), l(R)\}} \right]^{\Delta},
\end{align*}
where $x_Q$ and $x_R$ are, respectively, any points in $Q$ and $R$, $\Delta := d_1 + d_2$,
and the implicit positive constant is independent of $Q$ and $R$.
\end{lemma}

The following is the reverse H\"older inequality for $\mathscr{A}_{p(\cdot),\infty}$ weights
in variable Lebesgue spaces,
which is exactly \cite[Theorem 3.15]{yyz25}.
\begin{lemma}\label{reverse Holder}
Let $p(\cdot)\in\mathcal{P}_0\cap LH$.
Then, for any $W \in \mathscr{A}_{p(\cdot),\infty}$,
there exist positive constants $C_1$, $C_2$, $A$, and $A_1$,
depending only on $p(\cdot)$ and $n$,
such that, for any $r \in (1,r_W]$ with
\begin{align}\label{rW}
r_W := 1+ \frac{1}{C_1 [W]_{\mathscr{A}_{p(\cdot),\infty}}^{A_1} 2^{C_2[W]_{\mathscr{A}_{p(\cdot),\infty}}}},
\end{align}
any cube $Q$ in $\mathbb{R}^n$, and any matrix $M \in M_{m}$,
\begin{align*}
\frac{1}{\|\mathbf{1}_{Q}\|_{L^{rp(\cdot)}}} \left\|\left\| W(\cdot)M \right\|\mathbf{1}_Q\right\|_{L^{rp(\cdot)}}
\lesssim [W]_{\mathscr{A}_{p(\cdot),\infty}}^{A}
\frac{1}{\|\mathbf{1}_{Q}\|_{L^{p(\cdot)}}} \left\| \left\| W(\cdot)M \right\|\mathbf{1}_Q\right\|_{L^{p(\cdot)}},
\end{align*}
where the implicit positive constant depends only on $p(\cdot)$, $n$, and $m$.
\end{lemma}

\begin{lemma}\label{Apinfty u}
Let $p(\cdot) \in \mathcal{P}_0\cap LH$.
Then there exist positive constants $C_1$ and $C_2$,
depending only on $p(\cdot)$ and $n$, such that,
for any $W\in \mathscr{A}_{p(\cdot),\infty}$ and $\alpha\in (0,\frac{\log 2}{C_1 + C_2 \log([W]_{\mathscr{A}_{p(\cdot),\infty}})})$,
\begin{align*}
\sup_Q \fint_Q \left\|W^{-1}(x) A_Q\right\|^\alpha\,dx < \infty,
\end{align*}
where the supremum is taken over all cubes $Q$ in $\mathbb{R}^n$
and $A_Q$ is the reducing operator of order $p(\cdot)$ for $W$.
\end{lemma}

\subsection{Convex Body Valued Functions}\label{sec convex}
Finally, in this subsection, we recall some elementary concepts of convex sets.
For any set $E\subset \mathbb{C}^m$, let $\overline{E}$ denote its
\emph{closure}. For any sets $E,F\subset \mathbb{C}^m$, their \emph{Minkowski
sum $E+F$} is defined by setting
$$
 E+F:=\{x+y:\ x\in E,\ y\in F\}.
$$
For any $\lambda\in \mathbb{C}$, let
$\lambda E:=\{\lambda x:\ x\in E\}.$
A set $E\subset \mathbb{C}^m$ is said to be \emph{symmetric} if $\lambda E=E$
for any $\lambda\in {\mathbb{C}}$ with $|\lambda| = 1$,
and \emph{absorbing} if, for every $v\in \mathbb{C}^m$,
there exists $t\in(0,\infty)$ such that $v\in tE$.
A set $K\subset \mathbb{C}^m$ is said to be \emph{convex} if,
for any $x,y\in K$ and any $\lambda\in (0,1)$, $\lambda x+(1-\lambda)y\in K.$
For any set $E\subset \mathbb{C}^m$, we denote by $\operatorname{conv}(E)$
its \emph{convex hull}, namely the smallest convex set containing $E$. Equivalently,
$$
 \operatorname{conv}(E)
 =
 \left\{
 \sum_{i=1}^{k}\alpha_i x_i:\
 x_i\in E,\ \alpha_i\in [0,1],\ \sum_{i=1}^{k}\alpha_i=1,\ k\in\mathbb{N}
 \right\}.
$$
Moreover, we write $\overline{\operatorname{conv}}(E)$ for the
\emph{closure of the convex hull} of $E$.
Let the \emph{symbol $\mathcal{K}$} denote the collection of all nonempty closed subsets of $\mathbb{C}^m$.
For any closed set $K\in {\mathcal K}$, $K$ is called a \emph{convex body}
if it is both convex and symmetric
and, moreover, we use the \emph{symbol ${\mathcal K}_{\mathrm{cs}}$} to denote the set of all convex bodies.
The \emph{symbol $ {\mathcal K}_{\mathrm{bcs}} $} denotes the set of all bounded convex bodies
and the \emph{symbol $ {\mathcal K}_{\mathrm{abcs}} $} denotes the set of all absorbing and bounded convex bodies.
For any set $K\subset {\mathbb{C}}^m$, we define its \emph{norm $|K|$}
by setting
$$ |K| := \sup\left\{ |\vec{v}|:\ \vec{v}\in K \right\}. $$
For any positive definite matrix $A \in M_m$ and any convex body $K$, the \emph{product $AK$}
is defined by setting
$$ AK := \left\{A\vec{v}:\ \vec{v}\in K\right\} $$
and it is obvious that $AK$ also is a convex body.
The following two lemmas can be easily deduced from the definition of the convex set; we omit the details.
\begin{lemma}\label{A F}
Let $\{K_\alpha\}_{\alpha\in \Lambda}\subset {\mathcal K}_{\mathrm{cs}}$ be a collection of convex bodies,
where $\Lambda$ may be uncountable.
Then, for any positive definite matrix $A\in M_m$,
$$ \left| A \overline{{\mathop\mathrm{\,conv\,}}}\left( \bigcup_{\alpha\in \Lambda} K_\alpha\right) \right| =
\sup_{\alpha\in \Lambda} \left| A K_\alpha \right|. $$
\end{lemma}
\begin{lemma}\label{Kconvex}
Let $\{K_{\alpha,\gamma}\}_{\alpha\in \Lambda, \gamma\in\Gamma}\subset {\mathcal K}_{\mathrm{cs}}$,
where $\Lambda$ and $\Gamma$ may be uncountable.
Then
$$ \overline{{\mathop\mathrm{\,conv\,}}}\left( \bigcup_{\gamma\in \Gamma} \bigcup_{\alpha\in \Lambda} K_{\alpha,\gamma} \right)
= \overline{{\mathop\mathrm{\,conv\,}}}\left[ \bigcup_{\gamma\in \Gamma} \overline{{\mathop\mathrm{\,conv\,}}}\left(\bigcup_{\alpha\in \Lambda} K_{\alpha,\gamma}\right) \right]. $$
\end{lemma}

Now, we recall some basic properties of convex body valued functions.
A function $F:\ {\mathbb{R}^n} \to {\mathcal K}_{\mathrm{cs}}$ is said to be \emph{measurable} if,
for every open set $U\subset {\mathbb{C}}^m$, it holds that the set
$$ F^{-1}(U) := \left\{ x\in{\mathbb{R}^n}:\ F(x)\cap U \neq \emptyset\right\} $$
is measurable in the sense of Lebesgue.

The following is the measurability of the convex hull union of a sequence of measurable convex body valued functions
(see, for instance, \cite[Theorem 3.3]{bc22} and \cite[Theorem 8.24]{af09}).
\begin{lemma}\label{meas 2}
For any sequence of measurable convex body valued functions $\{F_k\}_{k\in{\mathbb{N}}}$,
the convex hull union map $G:\ {\mathbb{R}^n} \to {\mathcal K}_{\mathrm{cs}}$,
defined by setting
$$ G := \overline{{\mathop\mathrm{\,conv\,}}}\left( \bigcup_{k\in{\mathbb{N}}} F_k \right), $$
is measurable.
\end{lemma}

For any $p(\cdot)\in\mathcal P_0$, let the \emph{symbol $L^{p(\cdot)}({\mathcal K})$}
denote the set of all measurable
convex body valued functions $F$ such that
$\|F\|_{L^{p(\cdot)}({\mathcal K})} := \|\,|F|\,\|_{L^{p(\cdot)}} < \infty. $
Moreover, for any $p\in (0,\infty)$, let the \emph{symbol $L^{p}_{\rm loc}({\mathcal K})$}
denote the set of all measurable convex body valued functions $F$
such that $ |F| \in L^p_{\rm loc} $.

Finally, we introduce matrix-weighted maximal operators for convex body valued functions.
Let $\alpha\in(0,\infty)$ and $W$ be any given matrix weight.
Then the \emph{$\alpha$-convexification Christ--Goldberg convex body maximal operator
${\mathcal M}^{(\alpha)}_W$} is defined by setting,
for any $F\in L^{\alpha}_{\rm loc}({\mathcal K})$ and $x\in {\mathbb{R}^n}$,
$${\mathcal M}^{(\alpha)}_W(F) (x) := \sup_{x\in B} \left[\fint_B |W(x)W^{-1}(y)F(y)|^\alpha\,dy\right]^{\frac{1}{\alpha}}, $$
where the supremum is taken over all balls in ${\mathbb{R}^n}$ containing $x$.
When $\alpha = 1$, we simply use the \emph{symbol ${\mathcal M}_W$} to denote ${\mathcal M}_W^{(1)}$.

The following lemma is precisely \cite[Theorem 2.22]{yyz25}
with the vector-valued function $\vec{f}$ replaced by the convex body valued function $F$,
which can be proved by the same argument as that used in the proof of \cite[Theorem 6.9]{bc22};
we omit the details here.
\begin{lemma}\label{bound max}
Let $p(\cdot) \in \mathcal{P}_0\cap LH$.
Then, for any $W \in \mathscr{A}_{p(\cdot),\infty}$,
there exists $\alpha \in (0,1]$ such that,
for any $F \in L^{1}_{\rm loc}({\mathcal K})$,
\begin{align*}
\left\| \mathcal{M}_{W}^{(\alpha)} \left( F \right) \right\|_{L^{p(\cdot)}}
\lesssim \left\|F\right\|_{L^{p(\cdot)}({\mathcal K})},
\end{align*}
where the implicit positive constant is independent of $F$.
\end{lemma}

\section{Matrix-Weighted Variable Hardy Spaces}\label{sec Hardy}
In this section, we introduce the matrix-weighted variable Hardy space
and obtain some properties of matrix-weighted variable Hardy spaces.
We first introduce the concepts of several matrix
${\mathscr A}_{p(\cdot),\infty}$ weighted maximal operators
(see \cite[Definition 2.5]{bcyy24} for the case $W\in A_{p,\infty}$).
\begin{definition}\label{def all M}
Let $p(\cdot) \in {\mathcal P}_0$, $W\in {\mathscr A}_{p(\cdot),\infty}$,
$\psi\in{\mathcal S}$, $N\in{\mathbb{N}}$, and $a,l\in (0,\infty)$.
Let $\vec f\in(\mathcal S')^m.$
Then the \emph{matrix-weighted radial maximal function $M_W(\vec{f},\psi)$},
the \emph{matrix-weighted grand radial maximal function $(M_N)_W(\vec{f})$},
the \emph{matrix-weighted non-tangential maximal function $(M_a^{\ast})_W(\vec{f},\psi)$},
the \emph{matrix-weighted maximal function $(M_l^{\ast\ast})_W(\vec{f},\psi)$ of Peetre type},
and the \emph{matrix-weighted grand maximal function $(M_{l,N}^{\ast\ast})_W(\vec{f})$ of Peetre type}
are defined, respectively, by setting, for any $x\in {\mathbb{R}^n}$,
\begin{align*}
M_W \left(\vec{f},\psi\right)(x) := \sup_{t\in(0,\infty)} \left| W(x)\psi_t\ast \vec{f}(x) \right|,
\ \ \left(M_N\right)_W\left(\vec{f}\right)(x) := \sup_{\phi\in {\mathcal S}_N}\sup_{t\in(0,\infty)} \left| W(x)\phi_t\ast \vec{f}(x) \right|,
\end{align*}
\begin{align*}
\left(M_a^{\ast}\right)_W\left(\vec{f},\psi\right)(x) := \sup_{t\in(0,\infty)}\sup_{y\in B(x,at)} \left| W(x)\psi_t\ast \vec{f}(y) \right|,
\end{align*}
\begin{align*}
\left(M_l^{\ast\ast}\right)_W\left(\vec{f},\psi\right)(x) := \sup_{t\in(0,\infty)}
\sup_{y\in{\mathbb{R}^n}} \left| W(x) \psi_t \ast \vec{f}(x-y) \right|\left( 1 + \frac{|y|}{t} \right)^{-l},
\end{align*}
and
\begin{align*}
\left(M_{l,N}^{\ast\ast}\right)_W\left(\vec{f}\right)(x) := \sup_{\phi\in {\mathcal S}_N} \sup_{t\in(0,\infty)}
\sup_{y\in{\mathbb{R}^n}} \left| W(x) \phi_t \ast \vec{f}(x-y) \right|\left( 1 + \frac{|y|}{t} \right)^{-l}.
\end{align*}
\end{definition}

Now, we introduce the concept of matrix-weighted variable Hardy spaces
(see \cite[Definition 2.4]{bcyy24} for matrix-weighted Hardy spaces).
\begin{definition}\label{def Hardy}
Let $p(\cdot) \in {\mathcal P}_0$, $N\in\mathbb Z_+$, and $W\in {\mathscr A}_{p(\cdot),\infty}$.
The \emph{matrix-weighted variable Hardy space $H^{p(\cdot)}_{W,N}$}
is defined to be the set of all $\vec{f} \in ({\mathcal S}')^m$
such that
$$\left\|\vec{f}\right\|_{H^{p(\cdot)}_{W,N}} := \left\|(M_{N})_W (\vec{f})\right\|_{L^{p(\cdot)}} < \infty. $$
\end{definition}

In order to establish the equivalent characterizations of matrix-weighted variable Hardy spaces
in terms of the above introduced various matrix-weighted maximal functions,
we first introduce the following index.
Let $p(\cdot) \in \mathcal{P}_0\cap LH$.
For any $W\in {\mathscr A}_{p(\cdot),\infty}$, let
\begin{align}\label{def alphaW}
\alpha_W := \sup\left\{\alpha \in (0,1]:\ \mathcal{M}_{W}^{(\alpha)} \ \text{is bounded from}\ L^{p(\cdot)}({\mathcal K})\ \text{to}\ L^{p(\cdot)}\right\}.
\end{align}
Then, by Lemma \ref{bound max}, we find that,
for any $W\in \mathscr{A}_{p(\cdot),\infty}$, there exists $\alpha\in(0,1]$
such that the operator ${\mathcal M}^{(\alpha)}_W$ is bounded on $L^{p(\cdot)}({\mathcal K})$,
i.e., $\alpha_W\in(0,1]$ exists.

The following is the aforementioned equivalent characterizations
of matrix-weighted variable Hardy spaces
(see \cite[Theorem 2.28]{bcyy24} for the corresponding ones of matrix-weighted Hardy spaces).
\begin{theorem}\label{H equal}
Let $p(\cdot) \in {\mathcal P}_0\cap LH$ and $W\in {\mathscr A}_{p(\cdot),\infty}$.
Assume that $\psi \in {\mathcal S}$ with $\int_{{\mathbb{R}^n}} \psi(x)\,dx \neq 0$.
Let  $a\in (0,\infty)$,
$l\in (\frac{n}{\alpha_W},\infty)$,
and $N\in(l,\infty)\cap{\mathbb{N}}$, where $\alpha_W$ is as in \eqref{def alphaW}.
Then, for any $\vec{f} \in ({\mathcal S}')^m$,
\begin{align*}
\left\| \vec{f} \right\|_{H^{p(\cdot)}_{W,N}} \sim \left\| M_W \left( \vec{f},\psi \right) \right\|_{L^{p(\cdot)}}
\sim \left\| \left(M_a^\ast\right)_W \left( \vec{f},\psi \right) \right\|_{L^{p(\cdot)}}\sim \left\| \left(M^{\ast\ast}_l\right)_W \left( \vec{f},\psi \right) \right\|_{L^{p(\cdot)}}
\sim \left\| \left(M^{\ast\ast}_{l,N}\right)_W \left( \vec{f} \right) \right\|_{L^{p(\cdot)}},
\end{align*}
where the positive equivalence constants are independent of $\vec{f}$.
\end{theorem}
\begin{remark}
\begin{itemize}
\item[{\rm (i)}] Based on Theorem \ref{H equal}, in what follows, for any $p(\cdot)\in {\mathcal P}_0\cap LH$
and $W\in {\mathscr A}_{p(\cdot),\infty}$,
we denote $H^{p(\cdot)}_{W,N}$ simply by $H^{p(\cdot)}_W$
if $N \in (\frac{n}{\alpha_W},\infty)\cap\mathbb N$.
\item[{\rm (ii)}] Let $m = 1$, $p(\cdot)\in \mathcal{P}_0\cap LH$, and $W \equiv 1$;
then, in this case, $H^{p(\cdot)}_W$ coincides with the classical variable Hardy space,
and Theorem \ref{H equal} coincides with \cite[Theorem 3.3]{ns12} and \cite[Theorem 3.1]{cw14}.
On the other hand, for more general case $m\neq1$ and
$W\in\mathscr A_{p(\cdot),\infty}$, Theorem \ref{H equal} is new.
\end{itemize}
\end{remark}

To prove Theorem \ref{H equal},
we need to introduce the concepts of their corresponding convex body valued maximal functions.
In what follows, for any vector $\vec{v}\in {\mathbb{C}}^m$,
let
$${\mathcal K}(\vec{v}) := \{\lambda \vec{v}:\ \lambda\in
{\mathbb{C}} \ \text{with}\ |\lambda|\leq 1\}.$$
\begin{definition}\label{def convex max}
Let $p(\cdot) \in {\mathcal P}_0$, $W\in {\mathscr A}_{p(\cdot),\infty}$,
$\psi\in{\mathcal S}$, $N\in{\mathbb{N}}$, and $a,l\in (0,\infty)$.
Let $\vec f\in(\mathcal S')^m$.
Then the \emph{convex body valued radial maximal function $M^{{\mathcal K}}(\vec{f},\psi)$},
the \emph{convex body valued grand radial maximal function $M_N^{{\mathcal K}}(\vec{f})$},
the \emph{convex body valued non-tangential maximal function $M_a^{\ast, {\mathcal K}}(\vec{f},\psi)$},
and the \emph{convex body valued maximal function $M_l^{\ast\ast,{\mathcal K}}(\vec{f},\psi)$ of Peetre type}
are defined, respectively, by setting, for any $x\in {\mathbb{R}^n}$,
\begin{align*}
M^{{\mathcal K}} \left(\vec{f},\psi\right)(x) := \overline{{\mathop\mathrm{\,conv\,}}} \left( \bigcup_{t\in(0,\infty)} {\mathcal K}\left(\psi_t\ast \vec{f}\right)(x) \right),
\end{align*}
\begin{align*}
M_N^{{\mathcal K}} \left(\vec{f}\right)(x) := \overline{{\mathop\mathrm{\,conv\,}}} \left( \bigcup_{\phi\in {\mathcal S}_N}\bigcup_{t\in(0,\infty)} {\mathcal K}\left(\phi_t\ast \vec{f}\right)(x) \right),
\end{align*}
\begin{align*}
M_a^{\ast,{\mathcal K}}\left(\vec{f},\psi\right)(x) := \overline{{\mathop\mathrm{\,conv\,}}} \left(  \bigcup_{t\in(0,\infty)}\bigcup_{y\in B(x,at)} {\mathcal K}\left( \psi_t\ast \vec{f}\right)(y) \right),
\end{align*}
and
\begin{align*}
M_l^{\ast\ast,{\mathcal K}}\left(\vec{f},\psi\right)(x) := \overline{{\mathop\mathrm{\,conv\,}}} \left(
\bigcup_{t\in(0,\infty)} \bigcup_{y\in{\mathbb{R}^n}}
{\mathcal K}\left( \psi_t \ast \vec{f} \right) (x-y) \left( 1 + \frac{|y|}{t} \right)^{-l}\right).
\end{align*}
\end{definition}
\begin{remark}
Let $ p(\cdot) $, $W$, $\psi$, $N$, $a$, and $l$ be the same as in Definition \ref{def convex max}.
Then, from Lemma \ref{A F}, we infer that,
for any vector-valued function $\vec{f}\in ({\mathcal S}')^m$ and any $x\in{\mathbb{R}^n}$,
$$M_W \left(\vec{f},\psi\right)(x) = \left|W(x) M^{{\mathcal K}} \left(\vec{f},\psi\right)(x)\right|,\ \
(M_N)_W\left(\vec{f}\right)(x) = \left|W(x) M_N^{{\mathcal K}} \left(\vec{f}\right)(x)\right|,$$
$$(M_a^{\ast})_W\left(\vec{f},\psi\right)(x) = \left|W(x) M_a^{\ast,{\mathcal K}}\left(\vec{f},\psi\right)(x)\right|,
\text{ and } (M_l^{\ast\ast})_W\left(\vec{f},\psi\right)(x) = \left|W(x) M_l^{\ast\ast,{\mathcal K}}\left(\vec{f},\psi\right)(x)\right|.$$
\end{remark}
The following lemma shows the measurability of these convex body valued maximal functions.
\begin{lemma}\label{meas M}
Let $\psi\in{\mathcal S}$,
$N\in{\mathbb{N}}$, and $a,l\in (0,\infty)$.
Then, for any $\vec{f} \in ({\mathcal S}')^m$,
$M^{{\mathcal K}} (\vec{f},\psi)$, $M_N^{{\mathcal K}} (\vec{f})$, $M_a^{\ast,{\mathcal K}}(\vec{f},\psi)$,
and $M_l^{\ast\ast,{\mathcal K}}(\vec{f},\psi)$ are measurable.
\end{lemma}

\begin{proof}
We only give the proof for $M^{\mathcal K}(\vec f,\psi)$
because the proofs for the other three convex body valued maximal functions can be similarly proved
via replacing the corresponding parameter sets by
countable dense subsets.

For each $t\in (0,\infty)$, since $\psi_t*\vec f\in (\mathscr M)^m$, it follows that
the corresponding convex body valued function
$\mathcal K(\psi_t*\vec f(x))$ is measurable as well.
Moreover, by the continuity of $t\mapsto \psi_t$ in $\mathcal S$
and the assumption that $\vec f\in(\mathcal S')^m$, for every $x\in\mathbb R^n$ the map
$t\mapsto \psi_t*\vec f(x)$ is continuous on $(0,\infty)$. Hence,
by the density of $\mathbb{Q}\cap(0,\infty)$ in $(0,\infty)$,
$$
M^{\mathcal K}\left(\vec f,\psi\right)(x)
=
\overline{\operatorname{conv}}
\left(
\bigcup_{t\in\mathbb{Q}\cap(0,\infty)}\mathcal K\left(\psi_t*\vec f(x)\right)
\right),
$$
which, combined with Lemma \ref{meas 2}, further implies that $M^{\mathcal K}(\vec f,\psi)$ is measurable.
This finishes the proof of Lemma \ref{meas M}.
\end{proof}

We also need several truncated matrix-weighted
(and their corresponding convex body valued) maximal functions.
Let $p(\cdot)\in {\mathcal P}_0$ and $W\in {\mathscr A}_{p(\cdot),\infty}$.
Assume that $a,l\in (0,\infty)$, $\epsilon,K \in [0,\infty)$, and $\psi\in{\mathcal S}$.
For any $\vec{f} \in ({\mathcal S}')^m$ and $x\in {\mathbb{R}^n}$,
let
\begin{align*}
\left(M_a^{\ast}\right)_W^{\epsilon,K}\left(\vec{f},\psi\right)(x) :=
\sup_{t\in (0,\epsilon^{-1})}\sup_{y\in B(x,at)} \left| W(x) \psi_t\ast \vec{f}(y) \right|
\left( \frac{t}{t+\epsilon} \right)^K \left( \frac{1}{1+\epsilon |y|} \right)^K
\end{align*}
and
\begin{align*}
\left(M_l^{\ast\ast}\right)_W^{\epsilon,K}\left(\vec{f},\psi\right)(x) :=
\sup_{t\in (0,\epsilon^{-1})}\sup_{y\in{\mathbb{R}^n}} \left| W(x) \psi_t\ast \vec{f}(x-y) \right|
\left( 1 + \frac{|y|}{t} \right)^{-l} \left( \frac{t}{t+\epsilon} \right)^K \left( \frac{1}{1+\epsilon |x-y|} \right)^K,
\end{align*}
where $0^{-1} := \infty$.
Furthermore, we use the \emph{symbol $(M_a^{\ast})^{\epsilon,K,{\mathcal K}}$} to denote their corresponding
convex body valued maximal functions, which are respectively defined by setting,
for any $\vec{f}\in ({\mathcal S}')^m$ and $x\in\mathbb R^n$,
\begin{align*}
\left(M_a^{\ast}\right)^{\epsilon,K,{\mathcal K}}\left(\vec{f},\psi\right)(x) := \overline{{\mathop\mathrm{\,conv\,}}}
\left(\bigcup_{t\in(0,\epsilon^{-1})} \bigcup_{y\in B(x,at)} {\mathcal K}\left(\psi_t\ast \vec{f}\right)(y)
\left( \frac{t}{t+\epsilon} \right)^K \left( \frac{1}{1+\epsilon |y|} \right)^K \right)
\end{align*}
and
\begin{align*}
\left(M_l^{\ast\ast}\right)^{\epsilon,K,{\mathcal K}}\left(\vec{f},\psi\right)(x)
 := \overline{{\mathop\mathrm{\,conv\,}}}
\left(\bigcup_{t\in(0,\epsilon^{-1})} \bigcup_{y\in {\mathbb{R}^n}} {\mathcal K}\left(\psi_t\ast \vec{f}\right)(y)
\left( 1 + \frac{|x-y|}{t} \right)^{-l} \left( \frac{t}{t+\epsilon} \right)^K \left( \frac{1}{1+\epsilon |y|} \right)^K\right).
\end{align*}
Similarly to the proof of Lemma \ref{meas M}, we obtain the measurability of both
$(M_a^{\ast})^{\epsilon,K,{\mathcal K}}(\vec{f},\psi)$ and $(M_l^{\ast\ast})^{\epsilon,K,{\mathcal K}}(\vec{f},\psi)$
as follows.
\begin{lemma}\label{M finite}
Let $p(\cdot) \in {\mathcal P}_0\cap LH$ and $W\in {\mathscr A}_{p(\cdot),\infty}$.
Assume that $a\in (0,\infty)$ and $\psi\in{\mathcal S}$ with $\int_{{\mathbb{R}^n}} \psi(x)\,dx\neq 0$.
For any $\vec{f} \in ({\mathcal S}')^m$,
there exists $\widetilde{K} \in (0,\infty)$ such that, for any $\epsilon \in (0,\infty)$
and $K\in(\widetilde{K},\infty)$,
$\left(M_a^{\ast}\right)_W^{\epsilon,K}\left(\vec{f},\psi\right) \in L^{p(\cdot)}. $
\end{lemma}
\begin{proof}
Let $\epsilon \in (0,\infty)$ and $r := \min\{p_-,1\}$. Observe that, for any $x\in{\mathbb{R}^n}$,
\begin{align}\label{W p eq 1}
\left\|W(x)\right\| \leq \left\|W(x) A_{Q(\mathbf{0},\epsilon^{-1})}^{-1}\right\| \left\|A_{Q(\mathbf{0},\epsilon^{-1})}\right\|
\lesssim \left\|W(x) A_{Q(\mathbf{0},\epsilon^{-1})}^{-1}\right\|,
\end{align}
where the implicit positive constant may depend on $\epsilon$.
Using the same argument as in \cite[pp. 64--65]{g14 2} via replacing $f$ by $\vec{f}$,
we find that there exists $L\in{\mathbb{Z}}_+$, depending only on $\psi$ and $\vec{f}$,
such that, for any $t\in (0,\infty)$ and $y\in{\mathbb{R}^n}$,
\begin{align*}
\left| \psi_t\ast \vec{f}(y) \right| \lesssim \left( 1 +\epsilon |y| \right)^L \left(1+t^L\right)\left( t^{-n} + t^{-n-L} \right).
\end{align*}
Combining this with \eqref{W p eq 1},
we conclude that, for any $t\in (0,\epsilon^{-1})$ and $K \in (L + d_2 + \frac{n}{r}, \infty)$,
where $d_2$ is as in Lemma \ref{QP5},
and for $x\in{\mathbb{R}^n}$ and $y\in B(x,at)$,
\begin{align*}
&\left| W(x) \psi_t\ast \vec{f}(y) \right| \left( \frac{t}{t+\epsilon} \right)^K \left( \frac{1}{1+\epsilon |y|} \right)^K \\
&\quad \lesssim \left\|W(x) A_{Q(\mathbf{0},\epsilon^{-1})}^{-1}\right\|
\left(1+t^L\right)\left( t^{K-n} + t^{K-n-L} \right) \left( \frac{1}{t+\epsilon} \right)^K \left( \frac{1}{1+\epsilon |y|} \right)^{K-L}\\
&\quad \lesssim  \left\|W(x) A_{Q(\mathbf{0},\epsilon^{-1})}^{-1}\right\| \left( \frac{1}{1+\epsilon |y|} \right)^{K-L},
\end{align*}
where the implicit positive constant may depend on $\epsilon$.
From this and the obvious inequalities that $1 + \epsilon|x| \leq 1 + \epsilon |y| + \epsilon |x-y| < 1+ \epsilon|y| + \epsilon at < (1+a)(1 + \epsilon |y|)$,
we deduce that
\begin{align}\label{W p eq 2}
\left(M_a^{\ast}\right)_W^{\epsilon,K}\left(\vec{f},\psi\right)(x) \lesssim \left\|W(x) A_{Q(\mathbf{0},\epsilon^{-1})}^{-1} \right\| \left( \frac{1}{1+\epsilon |x|} \right)^{K-L}
= \epsilon^{K-L} \left\|W(x) A_{Q(\mathbf{0},\epsilon^{-1})}^{-1} \right\| \left( \frac{\epsilon^{-1}}{\epsilon^{-1}+|x|} \right)^{K-L}.
\end{align}
Let $Q_k := Q(\mathbf{0},2^{k}\epsilon^{-1})$ for any $k\in{\mathbb{Z}}_+$.
Using Lemma \ref{con f}, we find that
\begin{align}\label{W p eq 3}
&\left\|\,\left\|W(\cdot) A_{Q(\mathbf{0},\epsilon^{-1})}^{-1} \right\| \left( \frac{\epsilon^{-1}}{\epsilon^{-1}+|\cdot|} \right)^{K-L} \right\|_{L^{p(\cdot)}}^r {\nonumber}\\
&\quad  = \left\|\,\left\|W(\cdot) A_{Q(\mathbf{0},\epsilon^{-1})}^{-1} \right\|^r \left( \frac{\epsilon^{-1}}{\epsilon^{-1}+|\cdot|} \right)^{r(K-L)} \right\|_{L^{\frac{p(\cdot)}{r}}} {\nonumber}\\
&\quad  \lesssim \left\|\,\left\|W(\cdot) A_{Q(\mathbf{0},\epsilon^{-1})}^{-1} \right\|^r {\mathbf{1}}_{Q_0} \right\|_{L^{\frac{p(\cdot)}{r}}}
 + \sum_{k = 0}^\infty 2^{-rk(K-L)} \left\|\,\left\|W(\cdot) A_{Q(\mathbf{0},\epsilon^{-1})}^{-1} \right\|^r {\mathbf{1}}_{Q_{k+1} \setminus Q_k} \right\|_{L^{\frac{p(\cdot)}{r}}}{\nonumber}\\
&\quad \leq \sum_{k = 0}^\infty 2^{-rk(K-L)} \left\|\,\left\|W(\cdot) A_{Q(\mathbf{0},\epsilon^{-1})}^{-1} \right\| {\mathbf{1}}_{Q_k} \right\|_{L^{p(\cdot)}}^r.
\end{align}
Note that, for any $k\in{\mathbb{Z}}_+$,
$\frac{\|{\mathbf{1}}_{Q_k}\|_{L^{p(\cdot)}}}{\|{\mathbf{1}}_{Q_0}\|_{L^{p(\cdot)}}} \lesssim 2^{k\frac{n}{r}}$
(see, for instance, \cite[Lemma 3.2]{cfn12} with $w := 1$).
This, together with \eqref{W p eq 3}, Lemmas \ref{est fQ}, \ref{eq reduc M}, and \ref{QP5},
and the assumption $K \in (L + d_2 + \frac{n}{r}, \infty)$,
further implies that
\begin{align*}
&\left\|\,\left\|W(\cdot) A_{Q(\mathbf{0},\epsilon^{-1})}^{-1} \right\| \left( \frac{\epsilon^{-1}}{\epsilon^{-1}+|\cdot|} \right)^{K-L} \right\|_{L^{p(\cdot)}}^r \\
&\quad \lesssim \sum_{k = 0}^\infty 2^{-rk(K-L)} \left[\frac{\|{\mathbf{1}}_{Q_k}\|_{L^{p(\cdot)}}}{\|{\mathbf{1}}_{Q_0}\|_{L^{p(\cdot)}}}\right]^r
\frac{1}{\|{\mathbf{1}}_{Q_k}\|^r_{L^{p(\cdot)}}} \left\|\, \left\|W(\cdot) A_{Q(\mathbf{0},\epsilon^{-1})}^{-1} \right\| {\mathbf{1}}_{Q_k} \right\|_{L^{p(\cdot)}}^r\\
&\quad \lesssim \sum_{k = 0}^\infty 2^{-rk(K-L - \frac{n}{r})} \left\|A_{Q_k} A_{Q(\mathbf{0},\epsilon^{-1})}^{-1} \right\|^r
\lesssim \sum_{k = 0}^\infty 2^{- r k(K-L - \frac{n}{r} - d_2)} < \infty.
\end{align*}
Using this and \eqref{W p eq 2}, we conclude that
\begin{align*}
\left\| \left(M_a^{\ast}\right)_W^{\epsilon,K}\left(\vec{f},\psi\right) \right\|_{L^{p(\cdot)}}
&\lesssim  \left\| \left\|W(\cdot) A_{Q(\mathbf{0},\epsilon^{-1})}^{-1}\right\| \left( \frac{\epsilon^{-1}}{\epsilon^{-1}+|\cdot|} \right)^{K-L} \right\|_{L^{p(\cdot)}}
<\infty
\end{align*}
and hence $ (M_a^\ast)^{\epsilon,K}_W(\vec f,\psi) \in L^{p(\cdot)} $.
This finishes the proof of Lemma \ref{M finite}.
\end{proof}
\begin{lemma}\label{a ast l}
Let $p(\cdot)\in {\mathcal P}_0\cap LH$ and $W\in {\mathscr A}_{p(\cdot),\infty}$.
Assume that $a\in (0,\infty)$, $\epsilon,K \in [0,\infty)$, and $\psi\in{\mathcal S}$ with $\int_{{\mathbb{R}^n}} \psi(x)\,dx\neq 0$.
If $l\in (\frac{n}{\alpha_W},\infty)$, then, for any $\vec{f}\in ({\mathcal S}')^m$,
\begin{align*}
\left\| \left(M_l^{\ast\ast}\right)_W^{\epsilon,K}\left(\vec{f},\psi\right) \right\|_{L^{p(\cdot)}}
\lesssim \left\|\left(M_a^{\ast}\right)_W^{\epsilon,K}\left(\vec{f},\psi\right)\right\|_{L^{p(\cdot)}} ,
\end{align*}
where the implicit positive constant is independent of $\vec{f}$, $\epsilon$, and $K$.
\end{lemma}
\begin{proof}
Let $\alpha \in (\frac nl, \alpha_W)$.
Using the definition of $(M_a^{\ast})^{\epsilon,K,{\mathcal K}}(\vec{f},\psi)$,
we find that, for any $t\in (0,\epsilon^{-1})$, $x,y\in{\mathbb{R}^n}$, and $z\in B(x-y,at)$,
\begin{align*}
I(t,x,y) := \left| W(x) \psi_t\ast \vec{f}(x-y) \right| \left( \frac{t}{t+\epsilon} \right)^K \left( \frac{1}{1+\epsilon |x-y|} \right)^K
\leq \left|W(x) \left(M_a^{\ast}\right)^{\epsilon,K,{\mathcal K}}\left(\vec{f},\psi\right)(z)\right|.
\end{align*}
Then, taking the average with respect to $z$ over $B(x-y,at)$ on the both sides of the above inequality,
we obtain, for any $t\in (0,\epsilon^{-1})$ and $x,y\in{\mathbb{R}^n}$,
\begin{align*}
\left[ I(t,x,y) \right]^{\alpha} & \leq \fint_{B(x-y,at)} \left|W(x) \left(M_a^{\ast}\right)^{\epsilon,K,{\mathcal K}}\left(\vec{f},\psi\right)(z)\right|^\alpha\,dz \\
& \leq \left(\frac{|y|+at}{at} \right)^n  \fint_{B(x,|y|+at)} \left|W(x) \left(M_a^{\ast}\right)^{\epsilon,K,{\mathcal K}}\left(\vec{f},\psi\right)(z)\right|^\alpha\,dz \\
& \lesssim \left( 1 + \frac{|y|}{at} \right)^{\alpha l} \left[{\mathcal M}^{(\alpha)}_W\left(W(\cdot)\left(M_a^{\ast}\right)^{\epsilon,K,{\mathcal K}}\left(\vec{f},\psi\right)\right)(x)\right]^\alpha,
\end{align*}
which, combined with the definition of $(M_l^{\ast\ast})_W^{\epsilon,K}$,
further implies that, for any $x\in {\mathbb{R}^n}$,
\begin{align*}
\left(M_l^{\ast\ast}\right)_W^{\epsilon,K}\left(\vec{f},\psi\right)(x) \lesssim {\mathcal M}^{(\alpha)}_W\left(W(\cdot)\left(M_a^{\ast}\right)^{\epsilon,K,{\mathcal K}}\left(\vec{f},\psi\right)\right)(x).
\end{align*}
From this and Lemmas \ref{bound max} and \ref{A F},
we infer that
\begin{align*}
\left\| \left(M_l^{\ast\ast}\right)_W^{\epsilon,K}\left(\vec{f},\psi\right) \right\|_{L^{p(\cdot)}}
&\lesssim \left\| {\mathcal M}^{(\alpha)}_W\left(W(\cdot)\left(M_a^{\ast}\right)^{\epsilon,K,{\mathcal K}}\left(\vec{f},\psi\right)\right) \right\|_{L^{p(\cdot)}}
\lesssim \left\|\,\left|W(\cdot)\left(M_a^{\ast}\right)^{\epsilon,K,{\mathcal K}}\left(\vec{f},\psi\right)\right|\, \right\|_{L^{p(\cdot)}} \\
&= \left\|\left(M_a^{\ast}\right)^{\epsilon,K}_W \left(\vec{f},\psi\right) \right\|_{L^{p(\cdot)}}.
\end{align*}
This finishes the proof of Lemma \ref{a ast l}.
\end{proof}

For any $\vec{f} \in (C^{\infty})^m$, let
\begin{align*}
\nabla \vec f:=\left[\begin{matrix}
\partial_1 f_1&\cdots&\partial_nf_1\\
\vdots&\ddots&\vdots\\
\partial_1 f_m&\cdots&\partial_nf_m
\end{matrix}\right].
\end{align*}
Now, let $p(\cdot) \in {\mathcal P}_0$, $W\in {\mathscr A}_{p(\cdot),\infty}$, and
$\psi \in {\mathcal S}$ with $\int_{{\mathbb{R}^n}} \psi(x)\,dx \neq 0$
and let $a\in (0,\infty)$ and $\epsilon,K\in [0,\infty)$.
Then, for any $\vec{f} \in ({\mathcal S}')^m$ and $x\in\mathbb R^n$,
we define
\begin{align*}
\left(U_a^{\ast}\right)_W^{\epsilon,K}\left(\vec{f},\psi\right)(x) :=
\sup_{t\in (0,\epsilon^{-1})}\sup_{y\in B(x,at)} t\left\| W(x) \nabla\left(\psi_t\ast \vec{f}\right)(y) \right\|
\left( \frac{t}{t+\epsilon} \right)^K \left( \frac{1}{1+\epsilon |y|} \right)^K.
\end{align*}
\begin{lemma}\label{UM 1}
Let $p(\cdot) \in {\mathcal P}_0\cap LH$, $W\in {\mathscr A}_{p(\cdot),\infty}$,
$\psi \in {\mathcal S}$ with $\int_{{\mathbb{R}^n}} \psi(x)\,dx \neq 0$, $a,l\in (0,\infty)$,
and $\epsilon,K\in [0,\infty)$.
Then, for any $\vec{f} \in ({\mathcal S}')^m$ and $x \in {\mathbb{R}^n}$,
\begin{align*}
\left(U_a^{\ast}\right)_W^{\epsilon,K}\left(\vec{f},\psi\right)(x)
\lesssim \left(M_l^{\ast\ast}\right)_W^{\epsilon,K}\left(\vec{f},\psi\right)(x),
\end{align*}
where the implicit positive constant is independent of $\vec{f}$ and $\epsilon$.
\end{lemma}
\begin{proof}
Observe that, by the linearity of $\nabla$, for any $x,y\in{\mathbb{R}^n}$,
$$W(x) \nabla\left(\psi_t\ast \vec{f}\right)(y) =  \nabla\left[W(x)\psi_t\ast \vec{f}\right](y).$$
Using this and the well-known fact that all norms on a given finite dimensional vector space are equivalent,
we conclude that, for any $t\in (0,\infty)$ and $x,y\in {\mathbb{R}^n}$
\begin{align*}
t\left\| W(x) \nabla\left(\psi_t\ast \vec{f}\right)(y) \right\|
&\sim \sum_{i = 1}^m \sum_{j = 1}^n t \left| \partial_j \left[W(x)\psi_t\ast \vec{f}\right]_i(y) \right|\\
& = \sum_{i = 1}^m \sum_{j = 1}^n \left|  \left[W(x)\left(\partial_j\psi\right)_t \ast \vec{f}(y)\right]_i \right|
\sim \sum_{j = 1}^n \left| W(x)\left(\partial_j\psi\right)_t \ast \vec{f}(y) \right|.
\end{align*}
By \cite[Lemma 2.1.5]{g14 2} with $\Phi$, $\Psi$, and $m$
replaced, respectively,  by $\psi$, $\partial_j \psi$, and $L_0 := \lfloor l+1 \rfloor + K$,
we find that there exist $\{ \Theta^{(s)} \}_{s\in (0,1]} \subset {\mathcal S}$
such that, for any $j\in \{1,\dots,n\}$ and $x\in {\mathbb{R}^n}$
\begin{align*}
\partial_j \psi(x) = \int_{0}^1 \left[ \Theta^{(s)} \ast \psi_s \right](x) \,ds
\end{align*}
and, for any $s \in (0,1]$,
\begin{align}\label{UM eq 2}
\int_{{\mathbb{R}^n}} \left( 1 + |z| \right)^{L_0} \left| \Theta^{(s)}(z) \right|\,dz \lesssim s^{L_0}.
\end{align}
From these and the definition of $(M_l^{\ast\ast})_W^{\epsilon,K}$,
we deduce that, for any $j\in \{1,\dots,n\}$, $t\in (0,\epsilon^{-1})$, $x\in{\mathbb{R}^n}$,
and $y\in B(x,at)$,
\begin{align}\label{UM eq 1}
I(t,j,x,y) :&= \left| W(x) \left( \partial_j \psi\right)_t \ast \vec{f}(y) \right| \left( \frac{t}{t+\epsilon} \right)^K \left(\frac{1}{1 + \epsilon |y|}\right)^K{\nonumber} \\
& = \left| W(x) \int_0^1 \left[\left(\Theta^{(s)}\right)_t \ast \psi_{st} \ast \vec{f}\right](y)\,ds \right| \left( \frac{t}{t+\epsilon} \right)^K \left(\frac{1}{1 + \epsilon |y|}\right)^K{\nonumber}\\
& \leq \left( \frac{t}{t+\epsilon} \right)^K \ \int_0^1 \int_{{\mathbb{R}^n}} \left| t^{-n} \Theta^{(s)}(t^{-1}z) \right| \frac{| W(x) \psi_{st} \ast \vec{f}(y-z)|}{(1 + \epsilon |y|)^K} \,dz\,ds.
\end{align}
Note that, for any $z\in {\mathbb{R}^n}$,
$$\frac{1}{(1 + \epsilon |y|)^K} \leq \frac{(1 + \epsilon |z|)^K}{( 1 +\epsilon |y-z|)^K}.$$
Using this, \eqref{UM eq 1}, and \eqref{UM eq 2} with inserting the factor $1$ written as
\begin{align*}
\left( \frac{ts}{ts + |x - y + z|} \right)^{\lfloor l + 1\rfloor} \left(\frac{ts}{ts+ \epsilon}\right)^K
\left( \frac{ts + |x - y + z|}{ts} \right)^{\lfloor l + 1\rfloor} \left(\frac{ts+ \epsilon}{ts}\right)^K,
\end{align*}
we conclude that
\begin{align*}
I(t,j,x,y) & \lesssim \left(M_{\lfloor l + 1\rfloor}^{\ast\ast}\right)_W^{\epsilon,K}\left(\vec{f},\psi\right)(x)
\int_0^1 s^{-K} \\
&\quad \times \int_{{\mathbb{R}^n}} \left| t^{-n} \Theta^{(s)}(t^{-1}z) \right| (1 + \epsilon |z|)^K \left( \frac{ts + |x - y + z|}{ts} \right)^{\lfloor l + 1\rfloor}
 \,dz \,ds\\
& \leq \left(M_{\lfloor l + 1\rfloor}^{\ast\ast}\right)_W^{\epsilon,K}\left(\vec{f},\psi\right)(x)
\int_0^1 s^{-\lfloor l + 1\rfloor - K} \\
&\quad \times \int_{{\mathbb{R}^n}} \left| t^{-n} \Theta^{(s)}(t^{-1}z) \right| \left(1 + \epsilon t \frac{|z|}{t}\right)^K  \left( 1 + \frac{|x - y|}{t} + \frac{|z|}{t} \right)^{\lfloor l + 1\rfloor}
 \,dz \,ds\\
& \lesssim \left(M_{\lfloor l + 1\rfloor}^{\ast\ast}\right)_W^{\epsilon,K}\left(\vec{f},\psi\right)(x)
\int_0^1 s^{-\lfloor l + 1\rfloor - K} \int_{{\mathbb{R}^n}} \left| \Theta^{(s)}(z) \right| \left(1 + |z|\right)^K  \left( 1 + |z|\right)^{\lfloor l + 1\rfloor}
 \,dz \,ds \\
& \lesssim \left(M_{\lfloor l + 1\rfloor}^{\ast\ast}\right)_W^{\epsilon,K}\left(\vec{f},\psi\right)(x).
\end{align*}
Applying this and the definitions of $(U_a^{\ast})_W^{\epsilon,K}$ and $(M_l^{\ast\ast})_W^{\epsilon,K}$
yields, for any $x\in {\mathbb{R}^n}$,
\begin{align*}
\left(U_a^{\ast}\right)_W^{\epsilon,K}\left( \vec{f},\psi \right)(x)
\lesssim \left(M_{\lfloor l + 1\rfloor}^{\ast\ast}\right)_W^{\epsilon,K}\left( \vec{f},\psi \right)(x)
\leq \left(M_l^{\ast\ast}\right)_W^{\epsilon,K}\left( \vec{f},\psi \right)(x),
\end{align*}
which completes the proof of Lemma \ref{UM 1}.
\end{proof}
The following lemma is exactly \cite[Lemma 2.26]{bcyy24}.
\begin{lemma}\label{r delta}
Let $r,\delta \in (0,\infty)$ and $x,y\in{\mathbb{R}^n}$ with $|x-y| < (1+\delta)r$.
Then the following assertions hold.
\begin{itemize}
\item[{\rm (i)}] There exists $z\in{\mathbb{R}^n}$ such that $B(z,r^\ast) \subset [B(x,r)\cap B(y,\delta r)]$,
where
\begin{align*}
r^\ast := \frac{(1+\delta)r - \max\{|x-y|,|1-\delta|r\}}{2}.
\end{align*}
\item[{\rm (ii)}] If we further assume $y\in B(x,r)$,
then
\begin{align*}
B(z,C_\delta r)\subset \left[B(x,r)\cap B(y,\delta r)\right]
\ \ \text{and}\ \ \left|B(x,r)\cap B(y,\delta r)\right| \geq \left(C_\delta\right)^n \left|B(x,r)\right|,
\end{align*}
where $C_\delta := \min\{\frac{\delta}{2},1\}$.
\end{itemize}
\end{lemma}
The equivalences of $L^{p(\cdot)}$ norms of various matrix-weighted maximal functions are stated as follows.
\begin{theorem}\label{M equal}
Let $p(\cdot) \in {\mathcal P}_0\cap LH$, $W\in {\mathscr A}_{p(\cdot),\infty}$,
and $\psi \in {\mathcal S}$ with $\int_{{\mathbb{R}^n}} \psi(x)\,dx \neq 0$.
Assume that $N\in{\mathbb{N}}$, and $a,l\in (0,\infty)$.
Then the following statements hold.
\begin{itemize}
\item[{\rm (i)}] For any $\vec{f} \in ({\mathcal S}')^m$ and $x\in{\mathbb{R}^n}$,
\begin{align*}
M_W\left(\vec{f},\psi\right)(x) \leq \left(M_a^{\ast}\right)_W\left(\vec{f},\psi\right)(x) \leq (1+a)^l \left( M_l^{\ast\ast}\right)_W \left(\vec{f},\psi\right)(x),
\end{align*}
\begin{align*}
\left(M_N\right)_W\left(\vec{f}\right)(x) \leq \left(M_{a,N}^{\ast}\right)_W\left(\vec{f}\right)(x) \leq (1+a)^l \left( M_{l,N}^{\ast\ast}\right)_W \left(\vec{f}\right)(x),
\end{align*}
\begin{align*}
M_W\left(\vec{f},\psi\right)(x) \leq \left\|\psi\right\|_{{\mathcal S}_N} \left(M_N\right)_W\left(\vec{f}\right)(x),\ \
\left( M_{a}^{\ast} \right)_W \left(\vec{f},\psi\right)(x) \leq \left\|\psi\right\|_{{\mathcal S}_N} \left( M_{a,N}^{\ast}\right)_W \left(\vec{f}\right)(x),
\end{align*}
and
\begin{align*}
\left(M_l^{\ast\ast}\right)_W\left(\vec{f},\psi\right)(x) \leq \left\|\psi\right\|_{{\mathcal S}_N} \left( M_{l,N}^{\ast\ast}\right)_W \left(\vec{f}\right)(x).
\end{align*}
Moreover, for any $\widetilde{N} \in {\mathbb{N}}$ with $N \leq \widetilde{N}$,
any $\vec{f} \in ({\mathcal S}')^m$, and any $x\in{\mathbb{R}^n}$,
$$ \left( M_{l,N}^{\ast\ast}\right)_W \left(\vec{f}\right)(x) \leq \left( M_{l,\widetilde{N}}^{\ast\ast}\right)_W \left(\vec{f}\right)(x) . $$
\item[{\rm (ii)}] If $l \in (\frac{n}{\alpha_W},\infty)$, then, for any
$\vec{f} \in ({\mathcal S}')^m$,
\begin{align*}
 \left\| \left(M_l^{\ast\ast}\right)_W\left(\vec{f},\psi\right) \right\|_{L^{p(\cdot)}}
\lesssim \left\|\left(M_a^{\ast}\right)_W\left(\vec{f},\psi\right)\right\|_{L^{p(\cdot)}},
\end{align*}
where the implicit positive constant is independent of $\vec{f}$.
\item[{\rm (iii)}] If $N>l$, then, for any $\vec{f}\in ({\mathcal S}')^m$ and $x\in{\mathbb{R}^n}$,
\begin{align*}
\left( M_{l,N}^{\ast\ast}\right)_W \left(\vec{f}\right)(x)\lesssim \left(M_l^{\ast\ast}\right)_W\left(\vec{f},\psi\right)(x),
\end{align*}
where the implicit positive constant is independent of $\vec{f}$.
\item[{\rm (iv)}] For any $\vec{f} \in ({\mathcal S}')^m$,
\begin{align*}
\left\| \left(M_a^{\ast}\right)_W\left(\vec{f},\psi\right) \right\|_{L^{p(\cdot)}}
\lesssim \left\|M_W\left(\vec{f},\psi\right)\right\|_{L^{p(\cdot)}},
\end{align*}
where the implicit positive constant is independent of $\vec{f}$.
\end{itemize}
\end{theorem}
\begin{proof}
The assertion {\rm (i)} follows immediately from Definition \ref{def all M}
and, observing that the implicit positive constant in Lemma \ref{a ast l} is independent of $\epsilon$ and $K$,
we conclude the assertion {\rm (ii)} by taking $\epsilon$ and $K$ equal to $0$.

Next, we show {\rm (iii)}. From \cite[Lemma 2.1.5]{g14 2} with $\Psi$, $\Phi$, and $m$ replaced, respectively,
by $\phi$, $\psi$, and $N$,
it follows that, for any $\phi \in {\mathcal S}_N$,
there exist $\{\Theta^{(s)}\}_{s\in (0,1]} \subset {\mathcal S}$ such that, for any $x\in {\mathbb{R}^n}$,
\begin{align*}
\phi(x) = \int_0^1 \left[ \Theta^{(s)}\ast \psi_s \right](x)\,ds
\ \ \text{and}\ \
\int_{{\mathbb{R}^n}} \left( 1 + |z|\right)^N\left| \Theta^{(s)}(z) \right|\,dz \lesssim s^N.
\end{align*}
Combining these with the definition of $( M_{l}^{\ast\ast})_W$,
we conclude that, for any $t\in (0,\infty)$ and $x,y\in{\mathbb{R}^n}$,
\begin{align*}
\left|W(x) \phi_t\ast \vec{f}(x-y) \right|
& \leq   \int_0^1 \int_{{\mathbb{R}^n}} \left|\left[\Theta^{(s)}\right]_t(z)\right| \left| W(x)\psi_{st} \ast \vec{f}(x-y-z)\right|\,dz\,ds \\
& \leq \left( M_{l}^{\ast\ast}\right)_W\left(\vec{f},\psi\right)(x) \int_0^1 \int_{{\mathbb{R}^n}} \left|  \left[\Theta^{(s)}\right]_t(z)\right| \left( 1 + \frac{|y+z|}{st} \right)^l \,dz\,ds \\
& \leq \left( M_{l}^{\ast\ast}\right)_W\left(\vec{f},\psi\right)(x) \left( 1 +\frac{|y|}{t} \right)^l
\int_0^1 s^{-l} \int_{{\mathbb{R}^n}} \left| \Theta^{(s)}(z)\right| \left( 1 + |z| \right)^N \,dz\,ds \\
&\lesssim \left( M_{l}^{\ast\ast}\right)_W\left(\vec{f},\psi\right)(x) \left( 1 +\frac{|y|}{t} \right)^l,
\end{align*}
which, together with the definition of $( M_{l,N}^{\ast\ast})_W$,
further implies that, for any $x\in{\mathbb{R}^n}$,
\begin{align*}
\left( M_{l,N}^{\ast\ast}\right)_W\left(\vec{f}\right)(x) \lesssim \left( M_{l}^{\ast\ast}\right)_W\left(\vec{f},\psi\right)(x).
\end{align*}
This finishes the proof of {\rm (iii)}.

Finally, we give the proof of {\rm (iv)}.
Observe that, if $M_W(\vec{f},\psi) \notin L^{p(\cdot)}$, then {\rm (iv)} automatically holds.
Hence, we consider the case where $ M_W(\vec{f},\psi) \in L^{p(\cdot)} $.
Let $\epsilon, K \in[0,\infty)$ and $\vec{f} \in ({\mathcal S}')^m$.
By Lemmas \ref{a ast l} and \ref{UM 1},
we find that there exists a positive constant $C_1$,
independent of $\epsilon \in [0,\infty)$ and $\vec{f} \in ({\mathcal S}')^m$,
such that
\begin{align}\label{M equal eq 1}
\left\| \left(U_a^{\ast}\right)_W^{\epsilon,K}\left(\vec{f},\psi\right) \right\|_{L^{p(\cdot)}}
\leq C_1 \left\|\left(M_a^{\ast}\right)_W^{\epsilon,K}\left(\vec{f},\psi\right)\right\|_{L^{p(\cdot)}}.
\end{align}
Now, let
\begin{align*}
E_\epsilon := \left\{ x\in{\mathbb{R}^n}:\ \left(U_a^{\ast}\right)_W^{\epsilon,K}\left(\vec{f},\psi\right)(x)
\leq 2 C_1C_2 \left(M_a^{\ast}\right)_W^{\epsilon,K}\left(\vec{f},\psi\right)(x) \right\},
\end{align*}
where $C_2$ is the quasi-triangle constant of the quasi-norm of $L^{p(\cdot)}$.
Combining this with \eqref{M equal eq 1},
we conclude that
\begin{align}\label{M equal eq 6}
\left\|\left(M_a^{\ast}\right)_W^{\epsilon,K}\left(\vec{f},\psi\right) {\mathbf{1}}_{E_\epsilon^{\complement}}\right\|_{L^{p(\cdot)}}
\leq \frac1{2 C_1C_2} \left\| \left(U_a^{\ast}\right)_W^{\epsilon,K}\left(\vec{f},\psi\right) {\mathbf{1}}_{E_\epsilon^{\complement}} \right\|_{L^{p(\cdot)}}
\leq \frac1{2C_2}\left\|\left(M_a^{\ast}\right)_W^{\epsilon,K}
\left(\vec{f},\psi\right)\right\|_{L^{p(\cdot)}}.
\end{align}
Next, we prove that
\begin{align}\label{M equal eq 4}
\left\|\left(M_a^{\ast}\right)_W^{\epsilon,K}\left(\vec{f},\psi\right) {\mathbf{1}}_{E_\epsilon}\right\|_{L^{p(\cdot)}}
\lesssim \left\|M_W\left(\vec{f},\psi\right)\right\|_{L^{p(\cdot)}}.
\end{align}
To this end, let $x\in E_\epsilon$.
Then, by the definition of $(M_a^{\ast})_W^{\epsilon, K}$,
we find that there exist $t_x \in (0,\epsilon^{-1})$ and $y_x \in B(x,at_x)$
such that
\begin{align}\label{M equal eq 2}
\frac12\left(M_a^{\ast}\right)_W^{\epsilon, K}\left( \vec{f},\psi \right)(x)
\leq \left|W(x) \psi_{t_x}\ast \vec{f}(y_x)\right| \left( \frac{t_x}{t_x + \epsilon} \right)^K \left( \frac{1}{1 + \epsilon |y_x|} \right)^K.
\end{align}
Let $J(\xi) := t_x\|W(x) \nabla(\psi_{t_x}\ast \vec{f})(\xi)\|$.
Using the definitions of $E_\epsilon$ and $ (U_a^{\ast})_W^{\epsilon,K} $ yields,
for any $\xi \in B(x,at_x)$,
\begin{align*}
J(\xi) \left( \frac{t_x}{t_x + \epsilon} \right)^K \left( \frac{1}{1 + \epsilon |\xi|} \right)^K
\leq \left(U_a^{\ast}\right)_W^{\epsilon,K} \left(\vec{f},\psi\right)(x)
\leq 2 C_1 \left(M_a^{\ast}\right)_W^{\epsilon, K}\left( \vec{f},\psi \right)(x),
\end{align*}
which, together with \eqref{M equal eq 2}, further implies that
\begin{align*}
J(\xi) \leq 4 C_1 \left|W(x) \psi_{t_x}\ast \vec{f}(y_x)\right|  \left( \frac{1 + \epsilon|\xi|}{1 + \epsilon |y_x|} \right)^K.
\end{align*}
By this and the inequalities that $\frac{1 + \epsilon|\xi|}{1 + \epsilon |y_x|} \leq 1+\epsilon|\xi-y_x|<1+2\epsilon at_x < 1+2a$,
we obtain
\begin{align*}
J(\xi) \leq 4 C_1 (1+2a)^K \left|W(x) \psi_{t_x}\ast \vec{f}(y_x)\right|.
\end{align*}
From this, the mean value theorem, and the Cauchy--Schwartz inequality,
we infer that, for any $y\in B(x,at_x)$,
\begin{align}\label{M equal eq 3}
&\left|\, \left| W(x) \psi_{t_x}\ast \vec{f}(y) \right| - \left| W(x) \psi_{t_x}\ast \vec{f}(y_x) \right|\, \right|\nonumber\\
& \quad\leq \left| W(x) \psi_{t_x}\ast \vec{f}(y)  -  W(x) \psi_{t_x}\ast \vec{f}(y_x) \right| {\nonumber}\\
& \quad\sim \sum_{i = 1}^m \left| \left[W(x) \psi_{t_x}\ast \vec{f}(y)  -  W(x) \psi_{t_x}\ast \vec{f}(y_x)\right]_i \right|{\nonumber}\\
& \quad\leq \sum_{i = 1}^m \left| \nabla \left(\left[W(x) \psi_{t_x}\ast \vec{f}\right]_i\right)(\xi_i) \right|\left| y-y_x \right|{\nonumber}\\
& \quad\leq \sum_{i = 1}^m \left| W(x) \nabla \left(\psi_{t_x}\ast \vec{f}\right)(\xi_i) \right|\left| y-y_x \right|
\lesssim \left|W(x) \psi_{t_x}\ast \vec{f}(y_x)\right| \frac{|y-y_x|}{t_x},
\end{align}
where the implicit positive constant is independent of $\epsilon$ and $\vec{f}$
and, for any $i\in\{1,\dots,m\}$, $\xi_i = \theta_i y+ (1-\theta_i)y_x$ for some $\theta_i \in [0,1]$.
Let $C_3$ be the implicit positive constant in \eqref{M equal eq 3}.
Then, for any $y \in B(x,at_x) \cap B(y_x, (2C_3)^{-1}t_x)$,
\begin{align*}
\left|\, \left| W(x) \psi_{t_x}\ast \vec{f}(y) \right| - \left| W(x) \psi_{t_x}\ast \vec{f}(y_x) \right|\, \right|
\leq \frac12 \left|W(x) \psi_{t_x}\ast \vec{f}(y_x)\right|,
\end{align*}
which further implies that
$ | W(x) \psi_{t_x}\ast \vec{f}(y) | \geq \frac12 | W(x) \psi_{t_x}\ast \vec{f}(y_x) |$.
Applying this and \eqref{M equal eq 2} yields
$| W(x) \psi_{t_x}\ast \vec{f}(y)| \geq \frac14 (M_a^{\ast})_W^{\epsilon, K}( \vec{f},\psi )(x)$
and hence
\begin{align*}
\left| W(x) M^{{\mathcal K}}\left(\vec{f},\psi\right)(y) \right| \geq \frac14 \left(M_a^{\ast}\right)_W^{\epsilon, K}\left( \vec{f},\psi \right)(x).
\end{align*}
From this, Lemma \ref{r delta}{\rm (ii)},
and the definitions of $M_W$ and ${\mathcal M}^{(\alpha)}_{W}$,
it follows that, for any $\alpha \in (0, \alpha_W)$,
\begin{align*}
&{\mathcal M}^{(\alpha)}_W \left( W(\cdot) M^{{\mathcal K}}\left(\vec{f},\psi\right) \right)(x)\\
&\quad\geq \left\{\frac{1}{|B(x,at_x)|} \int_{B(x,at_x) \cap B(y_x, \frac{t_x}{2C_3})} \left| W(x) M^{{\mathcal K}}\left(\vec{f},\psi\right)(y) \right|^{\alpha}\,dy\right\}^{\frac1\alpha}\\
&\quad\gtrsim \left\{\frac{1}{|B(x,at_x)|} \int_{B(x,at_x) \cap B(y_x, \frac{t_x}{2C_3})} 
\left[ \left(M_a^{\ast}\right)_W^{\epsilon, K}\left( \vec{f},\psi \right)(x) \right]^{\alpha}\,dy\right\}^{\frac1\alpha}\\
& \quad\gtrsim \left(M_a^{\ast}\right)_W^{\epsilon, K}\left( \vec{f},\psi \right)(x).
\end{align*}
Using this and Lemma \ref{bound max},
we conclude that
\begin{align*}
\left\|\left(M_a^{\ast}\right)_W^{\epsilon, K}\left( \vec{f},\psi \right){\mathbf{1}}_{E_\epsilon}\right\|_{L^{p(\cdot)}}
\lesssim \left\|{\mathcal M}^{(\alpha)}_W \left( W(\cdot) M^{{\mathcal K}}\left(\vec{f},\psi\right) \right) \right\|_{L^{p(\cdot)}}
\lesssim \left\| M_W\left( \vec{f},\psi \right) \right\|_{L^{p(\cdot)}}.
\end{align*}
This finishes the proof of \eqref{M equal eq 4}.

Now, let $C_4$ be the implicit positive constant in \eqref{M equal eq 4},
which is independent of $\vec{f}$ and $\epsilon$.
Then, using \eqref{M equal eq 6} and \eqref{M equal eq 4},
we obtain
\begin{align}\label{M equal eq 8}
\left\| \left(M^{\ast}_a\right)_W^{\epsilon,K}\left( \vec{f},\psi \right) \right\|_{L^{p(\cdot)}}
\leq C_2C_4 \left\|M_W\left(\vec{f},\psi\right)\right\|_{L^{p(\cdot)}} + \frac{1}{2} \left\| \left(M^{\ast}_a\right)_W^{\epsilon,K}\left( \vec{f},\psi \right) \right\|_{L^{p(\cdot)}}.
\end{align}
From this and Lemma \ref{M finite}, we deduce that there exists $\widetilde{K} \in (0,\infty)$,
depending on $\vec{f}$, such that, for any $\epsilon\in(0,\infty)$ and $K\in (\widetilde{K},\infty)$,
$(M^{\ast}_a)_W^{\epsilon,K}( \vec{f},\psi ) \in L^{p(\cdot)}$ and hence
\begin{align}\label{M equal eq 7}
\left\| \left(M^{\ast}_a\right)_W^{\epsilon,K}\left( \vec{f},\psi \right) \right\|_{L^{p(\cdot)}} \leq 2C_4 \left\|M_W\left(\vec{f},\psi\right)\right\|_{L^{p(\cdot)}}.
\end{align}
Note that, for any $t\in (0,\epsilon^{-1})$ and $y\in B(x,at)$,
$$ \frac{1+\epsilon |x|}{1+\epsilon |y|} \geq \frac{1}{1 + \epsilon |x-y|} > \frac{1}{1+\epsilon at} > \frac{1}{1+a}. $$
By this, \eqref{M equal eq 7}, and the definition of $(M^{\ast}_a)_W^{\epsilon,K}$,
we find that, for any $x\in{\mathbb{R}^n}$,
\begin{align*}
\left(M^{\ast}_a\right)_W^{\epsilon,K}\left( \vec{f},\psi \right) (x)
\geq \sup_{t\in (0,\epsilon^{-1})} \sup_{y\in B(x,at)} \left| W(x)\psi_t\ast \vec{f}(y) \right| \left( \frac{t}{t+\epsilon} \right)^K \frac{(1+a)^{-K}}{(1+\epsilon|x|)^K},
\end{align*}
which further implies that, for any $x\in{\mathbb{R}^n}$,
\begin{align*}
\liminf_{\epsilon\to 0^+} \left(M^{\ast}_a\right)_W^{\epsilon,K}\left( \vec{f},\psi \right) (x)
\geq (1+a)^{-K} \left(M_a^{\ast}\right)_W \left( \vec{f},\psi \right) (x).
\end{align*}
Applying this and Fatou's lemma in the setting of variable Lebesgue spaces
(see, for instance, \cite[Theorem 2.59]{cf13}) yields, for any $K\in (\widetilde{K},\infty)$,
\begin{align*}
\left\| \left(M^{\ast}_a\right)_W\left( \vec{f},\psi \right) \right\|_{L^{p(\cdot)}} \leq (1+a)^K2C_4 \left\|M_W\left(\vec{f},\psi\right)\right\|_{L^{p(\cdot)}},
\end{align*}
which further implies that $(M^{\ast}_a)_W( \vec{f},\psi) \in L^{p(\cdot)}$.
Using this and \eqref{M equal eq 8} with $\epsilon = 0$ and $K=0$,
we conclude that
\begin{align*}
\left\| \left(M^{\ast}_a\right)_W\left( \vec{f},\psi \right) \right\|_{L^{p(\cdot)}}
\leq C_2C_4 \left\|M_W\left(\vec{f},\psi\right)\right\|_{L^{p(\cdot)}} + \frac{1}{2} \left\| \left(M^{\ast}_a\right)_W\left( \vec{f},\psi \right) \right\|_{L^{p(\cdot)}}
\end{align*}
and hence
\begin{align*}
\left\| \left(M^{\ast}_a\right)_W\left( \vec{f},\psi \right) \right\|_{L^{p(\cdot)}} \leq 2C_2C_4 \left\|M_W\left(\vec{f},\psi\right)\right\|_{L^{p(\cdot)}}.
\end{align*}
This finishes the proof of {\rm (iv)} and hence Theorem \ref{M equal}.
\end{proof}
Finally, we give the proof of Theorem \ref{H equal}.

\begin{proof}[Proof of Theorem \ref{H equal}]
It follows from Theorem \ref{M equal}{\rm (i)} that $ M_W(\vec{f},\psi) $
and $(M^{\ast\ast}_{l,N})_W ( \vec{f})$ are, respectively, the minimum and the maximum
of these maximal functions.
On the other hand, by {\rm (ii)} through {\rm (iv)} of Theorem \ref{M equal},
we find that
\begin{align*}
\left\| \left(M^{\ast\ast}_{l,N}\right)_W \left( \vec{f}\right) \right\|_{L^{p(\cdot)}}
&\lesssim \left\| \left(M^{\ast\ast}_l\right)_W \left( \vec{f},\psi \right) \right\|_{L^{p(\cdot)}}
\lesssim \left\| \left(M_a^\ast\right)_W \left( \vec{f},\psi \right) \right\|_{L^{p(\cdot)}}
\lesssim \left\| M_W \left( \vec{f},\psi \right) \right\|_{L^{p(\cdot)}},
\end{align*}
which, combined with the conclusion inferred above, implies the desired equivalences, and hence completes the proof of Theorem \ref{H equal}.
\end{proof}
Next, we present the embedding property about $H^{p(\cdot)}_W$
(see \cite[Proposition 2.33]{bcyy24} for the corresponding one of matrix-weighted Hardy spaces).
\begin{proposition}\label{H embed}
Let $p(\cdot)\in {\mathcal P}_0\cap LH$ and $W\in {\mathscr A}_{p(\cdot),\infty}$.
Then $H^{p(\cdot)}_W \subset ({\mathcal S}')^m$ and, moreover,
for any $\phi \in {\mathcal S}$ and $\vec{f} \in ({\mathcal S}')^m$,
\begin{align*}
\left|\left\langle \vec{f},\phi \right\rangle\right| \lesssim \|\phi\|_{{\mathcal S}_N} \left\| \vec{f} \right\|_{H^{p(\cdot)}_W},
\end{align*}
where $N \in (\frac{n}{\alpha_W},\infty)\cap\mathbb N$ and the implicit positive constant is independent of $\vec{f}$.
\end{proposition}
\begin{proof}
Observe that
\begin{align}\label{H embed eq 1}
\left|\left\langle \vec{f},\phi \right\rangle\right|
&= \left| \widetilde{\phi}\ast \vec{f}(\mathbf{0}) \right|
\leq \left\|A_{Q(\mathbf{0},1)}^{-1}\right\|  \left| A_{Q(\mathbf{0},1)} \left(\widetilde{\phi}\ast \vec{f}\right)(\mathbf{0}) \right|,
\end{align}
where $\widetilde{\phi} := \phi(-\cdot)$.
From the definition of $(M_{1,N}^\ast)_W$, we deduce that, for any $x\in Q(\mathbf{0},1)$,
$$ \left|W(x)\left(\widetilde{\phi}\ast\vec{f}\right)\left(\mathbf{0}\right)\right|\leq  \|\phi\|_{{\mathcal S}_N}\left(M_{\sqrt{2},N}^\ast\right)_W(\vec{f})(x). $$
Combining this with \eqref{eq redu}, \eqref{H embed eq 1}, and Theorem \ref{H equal},
we conclude that
\begin{align*}
\left|\left\langle \vec{f},\phi \right\rangle\right|
&\lesssim \left| A_{Q(\mathbf{0},1)} \widetilde{\phi}\ast \vec{f}(\mathbf{0}) \right|
\sim \frac{1}{\|{\mathbf{1}}_{Q(\mathbf{0},1)}\|_{L^{p(\cdot)}}}
\left\|\,\left| W(\cdot) \left(\widetilde{\phi}\ast \vec{f}\right)(\mathbf{0}) \right|{\mathbf{1}}_{Q(\mathbf{0},1)} \right\|_{L^{p(\cdot)}} \\
&\leq \frac{\|\phi\|_{{\mathcal S}_N}}{\|{\mathbf{1}}_{Q(\mathbf{0},1)}\|_{L^{p(\cdot)}}}
\left\| \left(M_{\sqrt{2},N}^\ast\right)_W\left(\vec{f}\right){\mathbf{1}}_{Q(\mathbf{0},1)} \right\|_{L^{p(\cdot)}}
\lesssim \left\|\phi\right\|_{{\mathcal S}_N} \left\|\vec{f}\right\|_{H^{p(\cdot)}_W}.
\end{align*}
This finishes the proof of Proposition \ref{H embed}.
\end{proof}
The completeness of $H^{p(\cdot)}_W$ is as follows
(see \cite[Proposition 2.33]{bcyy24} for the completeness of matrix-weighted Hardy spaces).
\begin{proposition}\label{H comp}
Let $p(\cdot)\in {\mathcal P}_0\cap LH$ and $W\in {\mathscr A}_{p(\cdot),\infty}$.
Then $H^{p(\cdot)}_W$ is complete.
\end{proposition}
\begin{proof}
Suppose that $\{\vec{f}_k\}_{k\in{\mathbb{N}}}$ is a Cauchy sequence in $H^{p(\cdot)}_W$.
By this and Proposition \ref{H embed},
we obtain $\{\vec{f}_k\}_{k\in{\mathbb{N}}}$ is also a Cauchy sequence in $({\mathcal S}')^m$.
From this and the completeness of $({\mathcal S}')^m$,
it follows that there exists $\vec{f} \in ({\mathcal S}')^m$ such that
$\vec{f} = \lim_{k\to \infty} \vec{f}_k$ in $({\mathcal S}')^m$.
Assume that $\psi\in {\mathcal S}$ with $\int_{{\mathbb{R}^n}} \psi(x)\,dx \neq 0$.
Using the definition of $M_W$,
we find that, for any $k,l\in{\mathbb{N}}$, $t\in (0,\infty)$, and $x\in{\mathbb{R}^n}$,
\begin{align*}
\left|W(x) \left\langle \vec{f}_{k+l} - \vec{f}_k, \psi_t(\cdot - x) \right\rangle\right|
\leq M_W\left( \vec{f}_{k+l} - \vec{f}_k, \psi \right)(x),
\end{align*}
which, together with letting $l \to \infty$,
further implies that, for any $k\in{\mathbb{N}}$, $t\in(0,\infty)$, and $x\in{\mathbb{R}^n}$,
\begin{align*}
\left|W(x) \left\langle \vec{f} - \vec{f}_k, \psi_t(\cdot - x) \right\rangle\right|
\leq \liminf_{l\to \infty} M_W\left( \vec{f}_{k+l} - \vec{f}_k, \psi \right)(x)
\end{align*}
and hence
\begin{align*}
M_W\left( \vec{f} - \vec{f}_k, \psi \right)(x) \leq \liminf_{l\to \infty} M_W\left( \vec{f}_{k+l} - \vec{f}_k, \psi \right)(x).
\end{align*}
From this and Fatou's lemma in the setting of variable Lebesgue spaces
(see, for instance, \cite[Theorem 2.59]{cf13}),
we infer that
\begin{align*}
\left\| M_W\left( \vec{f} - \vec{f}_k, \psi \right) \right\|_{L^{p(\cdot)}}
\leq \liminf_{l\to \infty} \left\| M_W\left( \vec{f}_{k+l} - \vec{f}_k, \psi \right)\right\|_{L^{p(\cdot)}} \to 0
\end{align*}
as $k\to \infty$,
which, combined with the sublinearity of the quasi-norm of $H^{p(\cdot)}_W$,
further implies that $\vec{f} \in H^{p(\cdot)}_W$ and
$\lim_{k\to \infty}\| \vec{f} - \vec{f}_k \|_{H^{p(\cdot)}_W} = 0$.
This finishes the proof of Proposition \ref{H comp}.
\end{proof}

The following lemma is precisely \cite[Lemma 3.12]{bcyy24}.
\begin{lemma}\label{f to}
Let $\psi\in{\mathcal S}$ satisfy ${\mathop\mathrm{\,supp\,}} \psi \subset B(\mathbf{0},1)$ and $\int_{{\mathbb{R}^n}} \psi(x)\,dx = 1$.
Assume that $\vec{f} \in (L^1_{\rm loc})^m$ and $H : {\mathbb{R}^n} \to M_m({\mathbb{C}})$ is a matrix-valued function.
Then, for almost every $x\in {\mathbb{R}^n}$,
$$ \lim\limits_{t\to 0^+} H(x)\psi_t \ast \vec{f}(x) = H(x) \vec{f}(x).  $$
\end{lemma}

We have the following coincidence of $H^{p(\cdot)}_W$ with $L^{p(\cdot)}_W$
if $p(\cdot) \in {\mathcal P}\cap LH$ with $p_- > 1$ and if $W\in \mathscr{A}_{p(\cdot)}$.
\begin{theorem}\label{H p W}
Let $p(\cdot)\in {\mathcal P}\cap LH$ with $p_- \in (1,\infty)$ and let $W\in {\mathscr A}_{p(\cdot)}$.
Then $H^{p(\cdot)}_W = L^{p(\cdot)}_W$ with equivalent norms.
\end{theorem}
\begin{proof}
Let $\psi\in {\mathcal S}$ with ${\mathop\mathrm{\,supp\,}} \psi \subset B(\mathbf{0},1)$ and $\int_{{\mathbb{R}^n}} \psi(x)\,dx = 1$.
We first show $L^{p(\cdot)}_W \subset H^{p(\cdot)}_W$.
Using Lemma \ref{Holder}, we find that, for any $\vec{f} \in L^{p(\cdot)}_W$ and $\phi \in{\mathcal S}$,
\begin{align*}
\int_{{\mathbb{R}^n}} \left| \vec{f}(x) \phi(x) \right|\,dx
\lesssim \left\|\, \left| W(\cdot) \vec{f} \right| \,\right\|_{L^{p(\cdot)}}
\left\| \,\left\| W(\cdot)^{-1} \phi \right\| \,\right\|_{L^{p'(\cdot)}}.
\end{align*}
Note that, for any $\phi \in {\mathcal S}$ and $x\in{\mathbb{R}^n}$,
$$|\phi(x)| \lesssim (1 + |x|)^{-n} \leq {\mathcal M}({\mathbf{1}}_{Q(\mathbf{0},1)})(x),$$
where the implicit positive constant is independent of $x$,
and, moreover, by Remark \ref{rem Ap} and \cite[Proposition 4.8]{cp23},
we have $\|W^{-1}\|$ is a scalar-valued $\mathscr{A}_{p'(\cdot)}$ weight.
Using these and the boundedness of the Hardy--Littlewood maximal operator
in $L^{p'(\cdot)}_{\|W^{-1}\|}$ (see, for instance, \cite[Theorem 1.3]{cdh11}), we conclude that
\begin{align}\label{H p W eq 1}
\left\|\, \left\|W^{-1}(\cdot)\right\| \phi \right\|_{L^{p'(\cdot)}}
\lesssim \left\|\, \left\|W^{-1}(\cdot)\right\|{\mathcal M}\left({\mathbf{1}}_{Q(\mathbf{0},1)}\right)\right\|_{L^{p'(\cdot)}}
\lesssim \left\|\, \left\|W^{-1}(\cdot)\right\|{\mathbf{1}}_{Q(\mathbf{0},1)}\right\|_{L^{p'(\cdot)}} < \infty
\end{align}
and hence $\vec{f} \in ({\mathcal S}')^m$.
Observe that, by \cite[Corollary 2.1.12]{g14} and $\psi\in\mathcal S$,
for any $\vec{f} \in L^{p(\cdot)}_W$ and $x\in {\mathbb{R}^n}$,
$$ M_{W} \left( \vec{f}, \psi \right)(x) \lesssim {\mathcal M}_W \left( W\vec{f} \right)(x). $$
Applying this and the boundedness of ${\mathcal M}_W$ in $L^{p(\cdot)}$,
we conclude that
\begin{align*}
\left\| \vec{f} \right\|_{H^{p(\cdot)}_W} \lesssim \left\| {\mathcal M}_W \left( W\vec{f} \right) \right\|_{L^{p(\cdot)}}
\lesssim \left\| W(\cdot) \vec{f} \right\|_{L^{p(\cdot)}} = \left\|\vec{f}\right\|_{L^{p(\cdot)}_W}
\end{align*}
and hence $L^{p(\cdot)}_W \subset H^{p(\cdot)}_W$.

Now, we prove $H^{p(\cdot)}_W\subset L^{p(\cdot)}_W$.
From the previously obtained result that $\|W^{-1}\|$ is a scalar-valued $\mathscr{A}_{p'(\cdot)}$ weight,
we deduce that, for any cube $Q\subset \mathbb R^n$,
$\|\,\|W^{-1}(\cdot)\|{\mathbf{1}}_Q \|_{L^{p'(\cdot)}}<\infty.$
By this, Lemma \ref{Holder}, and Theorem \ref{H equal}, we find that, for any $\vec{f} \in H^{p(\cdot)}_W$ and any cube $Q\subset{\mathbb{R}^n}$,
\begin{align*}
\int_{Q} \sup_{t\in (0,\infty)} \left| \psi_t \ast \vec{f} (x) \right|\,dx
\leq \left\|\,\left\|W^{-1}(\cdot)\right\|{\mathbf{1}}_Q \right\|_{L^{p'(\cdot)}} \left\|M_W \left( \vec{f},\psi \right) \right\|_{L^{p(\cdot)}}
\lesssim \left\|\vec{f}\right\|_{H^{p(\cdot)}_W} < \infty,
\end{align*}
which further implies that $\sup_{t\in (0,\infty)} |\psi_t \ast \vec{f}| \in L^{1}_{\rm loc}$.
Applying this and \cite[Lemma 7]{art05} in the case $\Omega := {\mathbb{R}^n}$ and $ B_x(\Omega):= \{\psi_t:\ t\in(0,\infty)\} $,
we conclude that there exists $\vec{f}_0 \in (L^1_{\rm loc})^m$
such that $\vec{f} = \vec{f}_0$ in $[(C^\infty_{\rm c})']^m$.
From this, Lemmas \ref{Holder} and \ref{f to}, \eqref{H p W eq 1}, and Theorem \ref{H equal},
we infer that, for any $\phi\in{\mathcal S}$,
\begin{align*}
\int_{{\mathbb{R}^n}} \left| \vec{f}_0(x) \phi(x) \right|\,dx
&\leq \left\|\, \left| W(\cdot) \vec{f}_0 \right| \,\right\|_{L^{p(\cdot)}}  \left\|\, \left| W(\cdot)^{-1} \phi \right|\, \right\|_{L^{p'(\cdot)}}\\
&\lesssim \left\|  \sup_{t\in (0,\infty)}\left| W(\cdot) \psi_t \ast \vec{f} (\cdot) \right| \right\|_{L^{p(\cdot)}}
\sim \left\|\vec{f}\right\|_{H^{p(\cdot)}_W} < \infty
\end{align*}
and hence $\vec{f}_0 \in ({\mathcal S}')^m$.
Using this with the well-known fact that $C^{\infty}_{\rm c}$ is dense in ${\mathcal S}$,
we conclude that $\vec{f} = \vec{f}_0$ in $({\mathcal S}')^m$.
This, together with Lemma \ref{f to} and Theorem \ref{H equal}, further implies that
\begin{align*}
\left\|\vec{f}_0\right\|_{L^{p(\cdot)}_W} \leq \left\| M_W\left(\vec{f}_0,\psi\right) \right\|_{L^{p(\cdot)}}
\sim \left\|\vec{f}_0\right\|_{H^{p(\cdot)}_W} = \left\|\vec{f}\right\|_{H^{p(\cdot)}_W},
\end{align*}
and consequently $H^{p(\cdot)}_W \subset L^{p(\cdot)}_W$, which hence completes the proof of Theorem \ref{H p W}.
\end{proof}

\section{Atomic Characterization}\label{sec atom}
In this section, we establish the atomic characterization of matrix-weighted variable Hardy spaces.
We first introduce the concept of $(p(\cdot),q,s)_W$-atoms
(see \cite[Definition 3.1]{bcyy24} for the case when $p$ is a constant exponent).
\begin{definition}\label{def atom}
Let $p(\cdot) \in {\mathcal P}_0$, $q\in [1,\infty]$, $s\in {\mathbb{Z}}_+$,
and $W\in{\mathscr A}_{p(\cdot),\infty}$.
A function $\vec{a}$ is called a \emph{$(p(\cdot),q,s)_W$-atom} supported in a cube $Q$ if
\begin{itemize}
\item[{\rm (i)}] ${\mathop\mathrm{\,supp\,}} \vec{a} \subset Q$,
\item[{\rm (ii)}] $\{\int_{Q} \|\, |W(\cdot) \vec{a}(x)|{\mathbf{1}}_Q\|_{L^{p(\cdot)}}^q\,dx\}^\frac1q \leq |Q|^{\frac1q}$
with the usual modification when $q = \infty$,
\item[{\rm (iii)}] for any $\gamma\in {\mathbb{Z}}_+^n$
with $|\gamma|\leq s$,
$\int_{{\mathbb{R}^n}} x^\gamma \vec{a}(x)\,dx = \vec{0}$.
\end{itemize}
Moreover, let $\{A_Q\}_{{\rm cube}\ Q}$ be a family of reducing operators of order $p(\cdot)$ for $W$.
A function $\vec{a}$ is called a \emph{$(p(\cdot),q,s)_{{\mathbb{A}}}$-atom} supported in a cube $Q$
if $\vec{a}$ satisfies {\rm (i)}, {\rm (iii)}, and
\begin{itemize}
\item[{\rm (iv)}] $[\int_{Q} |A_Q \vec{a}(x)|^q\,dx]^\frac1q \leq \frac{|Q|^{\frac1q}}{\|{\mathbf{1}}_Q\|_{L^{p(\cdot)}}}$.
\end{itemize}
\end{definition}
\begin{remark}
\begin{itemize}
\item[{\rm (i)}] From Lemma \ref{eq reduc M}, we deduce that, for any $(p(\cdot),q,s)_{{\mathbb{A}}}$-atom $\vec{a}$,
$\vec{a}$ is also a harmless constant multiple of a $ (p(\cdot),q,s)_W$-atom.
\item[{\rm (ii)}] By Definition \ref{def atom} and H\"older's inequality, we immediately find that,
for any $q_1,q_2\in [1,\infty]$ with $q_1 \geq q_2$ and for any $(p(\cdot),q_1,s)_W$-atom $\vec{a}$,
$\vec{a}$ is also a $(p(\cdot),q_2,s)_W$-atom.
\item [{\rm (iii)}] In Definition \ref{def atom}, if $m = 1$ and $W$ is a scalar weight, then Definition \ref{def atom}{\rm (ii)}
coincides with $\| a \|_{L^q} \leq \frac{|Q|^{\frac1q}}{\|w{\mathbf{1}}_{Q}\|_{L^{p(\cdot)}}}$, and hence
the $(p(\cdot),q,s)_W$-atom in this case coincides with the atom of weighted variable Hardy spaces
(see, for instance, \cite[Definition 5.2]{h17}).
\end{itemize}
\end{remark}
We now state the main theorem of this section.
\begin{theorem}\label{atom con}
Let $p(\cdot) \in {\mathcal P}_0\cap LH$, $W\in {\mathscr A}_{p(\cdot),\infty}$, $r := \min\{1,p_-\}$,
and $s\in  [\lfloor d^{\rm upper}_{p(\cdot),\infty}(W) + n(\frac{1}{r} - 1)\rfloor,\infty)\cap{\mathbb{Z}}_+ $.
Then the following statements hold:
\begin{itemize}
\item[{\rm (i)}] For any sequence $\{\lambda_k\}_{k\in{\mathbb{Z}}} \subset {\mathbb{C}}$ satisfying
$$ \left\| \left\{\sum_{k\in{\mathbb{Z}}} \left[\frac{|\lambda_k| }{\|{\mathbf{1}}_{Q_k}\|_{L^{p(\cdot)}}}\right]^{r} {\mathbf{1}}_{Q_k} \right\}^{\frac1r} \right\|_{L^{p(\cdot)}} < \infty$$
and any sequence of $(p(\cdot),q,s)_W$-atoms $\{\vec{a}_k\}_{k\in{\mathbb{Z}}}$
supported, respectively, in cubes $\{Q_k\}_{k\in{\mathbb{Z}}}$
with $q\in (\max\{1, \frac{r_W p_+}{r_W - 1}\}, \infty]$,
there exists $\vec{f} \in H^{p(\cdot)}_W$ such that $\vec{f} = \sum_{k\in{\mathbb{Z}}} \lambda_k \vec{a}_k$
in both $H^{p(\cdot)}_W$ and $({\mathcal S}')^m$
and, moreover,
\begin{align*}
\left\| \vec{f} \right\|_{H^{p(\cdot)}_W}
\lesssim \left\| \left\{\sum_{k\in{\mathbb{Z}}} \left[\frac{|\lambda_k| }{\|{\mathbf{1}}_{Q_k}\|_{L^{p(\cdot)}}}\right]^{r} {\mathbf{1}}_{Q_k} \right\}^{\frac1r} \right\|_{L^{p(\cdot)}},
\end{align*}
where the implicit positive constant is independent of
$\{\lambda_k\}_{k\in{\mathbb{Z}}}$ and $\{\vec{a}_k\}_{k\in{\mathbb{Z}}}$.
\item[{\rm (ii)}]
For any $\vec{f} \in H^{p(\cdot)}_W$,
there exist a sequence $\{ \lambda_k \}_{k\in{\mathbb{Z}}}\subset [0,\infty)$
and a sequence of $(p(\cdot),\infty,s)_W$-atoms $\{\vec{a}_k\}_{k\in{\mathbb{Z}}}$ supported,
respectively, in cubes $\{Q_k\}_{k\in{\mathbb{Z}}}$
such that $\vec{f} = \sum_{k\in{\mathbb{Z}}} \lambda_k \vec{a}_k$ in both $H^{p(\cdot)}_W$ and $(\mathcal S')^m$
and
\begin{align*}
\left\|\left\{\sum_{k\in{\mathbb{Z}}} \left[\frac{|\lambda_{k}|}{\|  {\mathbf{1}}_{Q_k} \|_{L^{p(\cdot)}} } \right]^r {\mathbf{1}}_{Q_k} \right\}^\frac1r\right\|_{L^{p(\cdot)}}
\lesssim \left\| \vec{f} \right\|_{H^{p(\cdot)}_W},
\end{align*}
where the implicit positive constant is independent of $\vec{f}$.
\end{itemize}
\end{theorem}

\begin{remark}
If $p(\cdot)\equiv p$ with $p\in(0,1]$ is a constant exponent, then,
in this case, Theorem \ref{atom con}{\rm (i)} coincides with
\cite[Theorem 3.5(i)]{bcyy24} and
Theorem \ref{atom con}{\rm (ii)} extends
\cite[Theorem 3.5(ii)]{bcyy24} from matrix $A_p$ weights
(see \cite[p.\,490]{fr21} for their definitions)
to the weaker class of matrix $A_{p,\infty}$ weights.
Moreover, in the scalar-valued unweighted
case, namely when $m=1$ and $W\equiv1$, Theorem \ref{atom con} in this case
coincides with
\cite[Theorems 4.5 and 4.6]{ns12}.
On the other hand, for more general
$p(\cdot) \in {\mathcal P}_0\cap LH$ and matrix
$\mathscr A_{p(\cdot),\infty}$ weights $W$, Theorem \ref{atom con} is new.
\end{remark}

To prove Theorem \ref{atom con}, we need more tools.
For any variable exponent $q(\cdot)$,
the \emph{variable maximal operator $ \mathcal{M}_{q(\cdot)} $}
is defined by setting, for any $f \in L^{q(\cdot)}_{\rm loc}$ and $x \in \mathbb{R}^n$,
\begin{align*}
\mathcal{M}_{q(\cdot)}(f)(x) := \sup_{x\in Q} \frac{1}{\|\mathbf{1}_Q\|_{L^{q(\cdot)}}} \left\|f\mathbf{1}_Q\right\|_{L^{q(\cdot)}},
\end{align*}
where the supremum is taken over all cubes $Q$ containing $x$.
The following is on the boundedness of ${\mathcal M}_{q(\cdot)}$ in variable Lebesgue spaces,
which is exactly \cite[Theorem 7.3.27]{dhr17}.
\begin{lemma}\label{Mq bound}
Let $p(\cdot),q(\cdot),r(\cdot)\in \mathcal{P}\cap LH$
such that $p(\cdot) = r(\cdot)q(\cdot)$ and $r_- \in (1,\infty)$.
Then, for any $f\in L^{p(\cdot)}$,
\begin{align*}
\left\|\mathcal{M}_{q(\cdot)}(f)\right\|_{L^{p(\cdot)}} \lesssim \left\|f\right\|_{L^{p(\cdot)}},
\end{align*}
where the implicit positive constant depends only on $p(\cdot)$, $q(\cdot)$, $r(\cdot)$, and $n$.
\end{lemma}
We have the following two substitutes
for the Fefferman--Stein type vector-valued inequality
in matrix-weighted variable Hardy spaces.
\begin{lemma}\label{f dec eq le}
Let $p(\cdot) \in {\mathcal P}_0 \cap LH$, $r := \min\{1,p_-\}$, $W\in {\mathscr A}_{p(\cdot),\infty}$,
and $\{A_Q\}$ be a family of reducing operators of order $p(\cdot)$ for $W$.
Assume that $\{\lambda_{k}\}_{k\in{\mathbb{Z}}}\subset [0,\infty)$ and $\{Q_k\}_{k\in{\mathbb{Z}}}$
with $Q_k := Q(c_k,l_k)$ is a sequence of cubes in ${\mathbb{R}^n}$.
Then, for any $L\in (d^{\rm upper}_{p(\cdot),\infty}(W) + \frac{n}{r},\infty)$,
\begin{align*}
\left\| \sum_{k \in {\mathbb{Z}}} \lambda_k \left\|W(\cdot) A_{Q_k}^{-1}\right\| \left(\frac{l_k}{l_k + |\cdot - c_k| }\right)^L \right\|_{L^{p(\cdot)}}
\lesssim \left\|\left(\sum_{k\in{\mathbb{Z}}} \lambda_k^r {\mathbf{1}}_{Q_k}\right)^\frac1r\right\|_{L^{p(\cdot)}},
\end{align*}
where the implicit positive constant is independent of $\{\lambda_{k}\}_{k\in{\mathbb{Z}}}$ and $\{Q_k\}_{k\in{\mathbb{Z}}}$.
\end{lemma}
\begin{proof}
Using Lemmas \ref{fg Lp} and \ref{con f} and the assumption $r\in (0,1]$,
we obtain
\begin{align}\label{f dec eq 20}
&\left\| \sum_{k\in{\mathbb{Z}}} \lambda_k\left\|W(\cdot) A_{Q_k}^{-1} \right\| \left(\frac{l_k}{l_k +|\cdot-c_k|}\right)^L \right\|^r_{L^{p(\cdot)}}{\nonumber}\\
&\quad = \left\| \left[ \sum_{k\in{\mathbb{Z}}} \lambda_k\left\|W(\cdot) A_{Q_k}^{-1} \right\| \left(\frac{l_k}{l_k +|\cdot-c_k|}\right)^L \right]^r\right\|_{L^{\frac{p(\cdot)}{r}}}{\nonumber}\\
&\quad \sim \sup_{\|g\|_{L^{(\frac{p(\cdot)}{r})'}}\leq 1}
\int_{{\mathbb{R}^n}}\left[ \sum_{k\in{\mathbb{Z}}} \lambda_k \left\|W(x) A_{Q_k}^{-1} \right\| \left(\frac{l_k}{l_k +|x-c_k|}\right)^L\right]^{r} |g(x)|\,dx {\nonumber} \\
&\quad \leq \sup_{\|g\|_{L^{(\frac{p(\cdot)}{r})'}}\leq 1}
\int_{{\mathbb{R}^n}} \sum_{k\in{\mathbb{Z}}} \lambda_k^r \left\|W(x) A_{Q_k}^{-1} \right\|^{r} \left(\frac{l_k}{l_k +|x-c_k|}\right)^{rL}  |g(x)|\,dx.
\end{align}
Now, fix $g\in L^{(\frac{p(\cdot)}{r})'}$.
Note that, for any $x\in {\mathbb{R}^n}$, we have
$$\left(\frac{l_k}{l_k + |x-c_k|}\right)^{rL}
\lesssim {\mathbf{1}}_{Q_k}(x) +  \sum_{j\in{\mathbb{N}}} 2^{-rjL} {\mathbf{1}}_{(2^jQ_k) \setminus (2^{j-1}Q_k)}(x)
\le\sum_{j\in{\mathbb{Z}}_+} 2^{-rjL} {\mathbf{1}}_{2^jQ_k}(x), $$
which, combined with Tonelli's theorem, further implies that
\begin{align}\label{f dec eq 19}
&\int_{{\mathbb{R}^n}} \sum_{k\in{\mathbb{Z}}}\lambda_k^r \left\|W(x) A_{Q_k}^{-1} \right\|^{r} \left(\frac{l_k}{l_k + |x-c_k|}\right)^{rL} |g(x)|\,dx{\nonumber}\\
&\quad \lesssim \int_{{\mathbb{R}^n}} \sum_{k\in{\mathbb{Z}}} \sum_{j\in{\mathbb{Z}}_+} 2^{-rjL} \lambda_k^r \left\|W(x) A_{Q_k}^{-1} \right\|^{r} |g(x)| {\mathbf{1}}_{2^jQ_k}(x)\,dx{\nonumber}\\
&\quad = \sum_{k\in{\mathbb{Z}}} \lambda_k^r \sum_{j\in{\mathbb{Z}}_+} 2^{-rjL} \int_{2^jQ_k}  \left\|W(x) A_{Q_k}^{-1} \right\|^{r} |g(x)|\,dx{\nonumber}\\
&\quad \leq \sum_{k\in{\mathbb{Z}}} \lambda_k^r \sum_{j\in{\mathbb{Z}}_+} 2^{-j(rL-n)} |Q_k| \fint_{2^jQ_k}  \left\|W(x) A_{Q_k}^{-1} \right\|^{r} |g(x)|\,dx.
\end{align}
Let $r_W$ be the same as in \eqref{rW}.
Then, using Lemmas \ref{Holder} and \ref{est Q} with $p(\cdot)$ therein replaced by $\frac{r_W p(\cdot)}{r}$
and using Lemma \ref{con f}, we obtain, for any $k\in{\mathbb{Z}}$ and $j\in{\mathbb{Z}}_+$,
\begin{align*}
&\fint_{2^jQ_k}  \left\|W(x) A_{Q_k}^{-1} \right\|^{r} |g(x)|\,dx \\
&\quad \lesssim \frac{1}{\|{\mathbf{1}}_{2^{j}Q_k}\|_{L^{\frac{r_Wp(\cdot)}{r}}}} \left\|\, \left\|W(x) A_{Q_k}^{-1} \right\|^{r} {\mathbf{1}}_{2^jQ_k} \right\|_{L^{\frac{r_Wp(\cdot)}{r}}}
 \frac{1}{\|{\mathbf{1}}_{2^{j}Q_k}\|_{L^{(\frac{r_Wp(\cdot)}{r})'}}} \left\| g{\mathbf{1}}_{2^jQ_k} \right\|_{L^{(\frac{r_Wp(\cdot)}{r})'}}\\
&\quad = \left[\frac{1}{\|{\mathbf{1}}_{2^{j}Q_k}\|_{L^{r_Wp(\cdot)}}} \left\|\, \left\|W(x) A_{Q_k}^{-1} \right\| {\mathbf{1}}_{2^jQ_k} \right\|_{L^{r_Wp(\cdot)}}\right]^{r}
\frac{1}{\|{\mathbf{1}}_{2^{j}Q_k}\|_{L^{(\frac{r_Wp(\cdot)}{r})'}}} \left\| g{\mathbf{1}}_{2^jQ_k} \right\|_{L^{(\frac{r_Wp(\cdot)}{r})'}},
\end{align*}
which, together with Lemma \ref{reverse Holder}, further implies that
\begin{align*}
&\fint_{2^jQ_k}  \left\|W(x) A_{Q_k}^{-1} \right\|^{r} |g(x)|\,dx \\
&\quad \lesssim \left[\frac{1}{\|{\mathbf{1}}_{2^{j}Q_k}\|_{L^{p(\cdot)}}} \left\|\, \left\|W(x) A_{Q_k}^{-1} \right\| {\mathbf{1}}_{2^jQ_k} \right\|_{L^{p(\cdot)}}\right]^{r}
 \frac{1}{\|{\mathbf{1}}_{2^{j}Q_k}\|_{L^{(\frac{r_Wp(\cdot)}{r})'}}} \left\| g{\mathbf{1}}_{2^jQ_k} \right\|_{L^{(\frac{r_Wp(\cdot)}{r})'}}.
\end{align*}
From this and Lemma \ref{eq reduc M} with $M := A_{Q_k}^{-1}$ and from Lemma \ref{QP5},
we infer that
\begin{align*}
\fint_{2^jQ_k}  \left\|W(x) A_{Q_k}^{-1} \right\|^{r} |g(x)|\,dx
& \lesssim \left\| A_{2^jQ_k} A_{Q_k}^{-1} \right\|^{r}
 \frac{1}{\|{\mathbf{1}}_{2^{j}Q_k}\|_{L^{(\frac{r_Wp(\cdot)}{r})'}}} \left\| g{\mathbf{1}}_{2^jQ_k} \right\|_{L^{(\frac{r_Wp(\cdot)}{r})'}}\\
& \lesssim 2^{d_2rj} \frac{1}{\|{\mathbf{1}}_{2^{j}Q_k}\|_{L^{(\frac{r_Wp(\cdot)}{r})'}}} \left\| g{\mathbf{1}}_{2^jQ_k} \right\|_{L^{(\frac{r_Wp(\cdot)}{r})'}},
\end{align*}
where $d_2$ is as in Lemma \ref{QP5} satisfying $ d_2 + \frac{n}{r} < L$.
Applying this, \eqref{f dec eq 19}, and Tonelli's theorem,
we conclude that
\begin{align*}
&\int_{{\mathbb{R}^n}} \sum_{k\in{\mathbb{Z}}} \lambda_k^r \left\|W(x) A_{Q_k}^{-1} \right\|^{r} \left(\frac{l_k}{l_k + |x-c_k|}\right)^{rL} |g(x)|\,dx{\nonumber}\\
&\quad \lesssim \sum_{k\in{\mathbb{Z}}} \lambda_k^r \sum_{j\in{\mathbb{Z}}_+} 2^{-rj(L-d_2-\frac{n}{r})}
\frac{1}{\|{\mathbf{1}}_{2^{j}Q_k}\|_{L^{(\frac{r_Wp(\cdot)}{r})'}}} \left\| g{\mathbf{1}}_{2^jQ_k} \right\|_{L^{(\frac{r_Wp(\cdot)}{r})'}} \left|Q_k\right|\\
&\quad \leq \sum_{j\in{\mathbb{Z}}_+} 2^{-rj(L-d_2-\frac{n}{r})} \sum_{k\in{\mathbb{Z}}} \int_{Q_k} \lambda_k^r {\mathcal M}_{(\frac{r_Wp(\cdot)}{r})'} \left( g \right)(y) \,dy\\
&\quad = \sum_{j\in{\mathbb{Z}}_+} 2^{-rj(L-d_2-\frac{n}{r})} \int_{{\mathbb{R}^n}} \sum_{k\in{\mathbb{Z}}} \lambda_k^r {\mathbf{1}}_{Q_k}(y) {\mathcal M}_{(\frac{r_Wp(\cdot)}{r})'} \left( g \right)(y) \,dy.
\end{align*}
This, together with $L \in (d_2 + \frac{n}{r},\infty)$
and Lemmas \ref{Holder} and \ref{Mq bound},
further implies that
\begin{align*}
&\int_{{\mathbb{R}^n}} \sum_{k\in{\mathbb{Z}}} \lambda_k^r \left\|W(x) A_{Q_k}^{-1} \right\|^{r} \left(\frac{l_k}{l_k + |x-c_k|}\right)^{rL} |g(x)|\,dx{\nonumber}\\
&\quad \lesssim \left\|\sum_{k\in{\mathbb{Z}}} \lambda_k^r {\mathbf{1}}_{Q_k}\right\|_{L^{\frac{p(\cdot)}{r}}}
 \left\|{\mathcal M}_{(\frac{r_Wp(\cdot)}{r})'} \left( g \right)\right\|_{L^{(\frac{p(\cdot)}{r})'}}
\lesssim \left\|\sum_{k\in{\mathbb{Z}}} \lambda_k^r {\mathbf{1}}_{Q_k}\right\|_{L^{\frac{p(\cdot)}{r}}} \left\| g\right\|_{L^{(\frac{p(\cdot)}{r})'}}.
\end{align*}
Combining this with \eqref{f dec eq 20} yields
\begin{align*}
&\left\| \left[ \sum_{k\in{\mathbb{Z}}} \lambda_k\left\|W(\cdot) A_{Q_k}^{-1} \right\| \left(\frac{l_k}{l_k + |\cdot-c_k|}\right)^L \right]^{r}\right\|_{L^{\frac{p(\cdot)}{r}}}\\
&\quad \lesssim \sup_{\|g\|_{L^{(\frac{p(\cdot)}{r})'}}\leq 1} \left\|\sum_{k\in{\mathbb{Z}}} \lambda_k^r {\mathbf{1}}_{Q_k}\right\|_{L^{\frac{p(\cdot)}{r}}}
\left\| g\right\|_{L^{(\frac{p(\cdot)}{r})'}}
\lesssim \left\|\sum_{k\in{\mathbb{Z}}} \lambda_k^r {\mathbf{1}}_{Q_k}\right\|_{L^{\frac{p(\cdot)}{r}}},
\end{align*}
which, together with Lemma \ref{con f}, further implies that
\begin{align*}
\left\|\sum_{k\in{\mathbb{Z}}} \lambda_k \left\|W(\cdot) A_{Q_k}^{-1} \right\| \left(\frac{l_k}{l_k + |\cdot-c_k|}\right)^L \right\|_{L^{p(\cdot)}}
\lesssim \left\|\left\{\sum_{k\in{\mathbb{Z}}} \lambda_k^r {\mathbf{1}}_{Q_k}\right\}^\frac1r\right\|_{L^{p(\cdot)}}.
\end{align*}
 This finishes the proof of Lemma \ref{f dec eq le}.
\end{proof}

\begin{lemma}\label{a atom eq le}
Let $p(\cdot) \in {\mathcal P}_0 \cap LH$, $r := \min\{1,p_-\}$, and $W\in {\mathscr A}_{p(\cdot),\infty}$.
Assume that $\{\lambda_{k}\}_{k\in{\mathbb{Z}}}\subset [0,\infty)$ and $\{Q_k\}_{k\in{\mathbb{Z}}}$
is a sequence of cubes in ${\mathbb{R}^n}$.
Let $T$ be a bounded operator on $L^q$ with some $q \in (\frac{r_Wp_+}{r_W - 1},\infty)$
and let $\{\vec{a}_k\}_{k\in{\mathbb{Z}}}$ be a sequence of measurable functions.
Then
\begin{align*}
\left\| \sum_{k \in{\mathbb{Z}}} \lambda_k  \left\| W(\cdot) A_{Q_k}^{-1} \right\| \left|T\vec{a}_k\right| {\mathbf{1}}_{2\sqrt{n} Q_k} \right\|_{L^{p(\cdot)}}
\lesssim \left\|\left[\sum_{k \in{\mathbb{Z}}} \left(|Q_k|^{-\frac1q}\lambda_k\left\|\vec{a}_k\right\|_{L^q} \right)^r {\mathbf{1}}_{Q_k} \right]^\frac1r\right\|_{L^{p(\cdot)}},
\end{align*}
where the implicit positive constant is independent of $\{\lambda_{k}\}_{k\in{\mathbb{Z}}}$, $\{Q_k\}_{k\in{\mathbb{Z}}}$, and $\left\{\vec{a}_k\right\}_{k\in{\mathbb{Z}}}$.
\end{lemma}
\begin{proof}
Using Lemmas \ref{con f} and \ref{fg Lp},
we find that
\begin{align}\label{a atom eq 10}
&\left\| \sum_{k\in{\mathbb{Z}}} \lambda_k \left\| W(\cdot) A_{Q_k}^{-1} \right\|  \left|T\vec{a}_k\right|{\mathbf{1}}_{2\sqrt{n} Q_k} \right\|^r_{L^{p(\cdot)}}{\nonumber} \\
& \quad = \left\| \left[\sum_{k\in{\mathbb{Z}}} \lambda_k \left\| W(\cdot) A_{Q_k}^{-1} \right\|  \left|T\vec{a}_k\right|{\mathbf{1}}_{2\sqrt{n} Q_k}\right]^r \right\|_{L^{\frac{p(\cdot)}{r}}}{\nonumber}\\
& \quad \sim \sup_{\|g\|_{L^{(\frac{p(\cdot)}{r})'}}\leq 1} \int_{{\mathbb{R}^n}} \left[\sum_{k \in{\mathbb{Z}}} \lambda_k \left\| W(x) A_{Q_k}^{-1} \right\|  \left|T\vec{a}_k(x)\right|{\mathbf{1}}_{2\sqrt{n} Q_k}(x)\right]^r |g(x)|\,dx.
\end{align}
Now, fix $g\in L^{(\frac{p(\cdot)}{r})'}$.
Then, by $r\in(0,1]$
and Tonelli's theorem,
we obtain
\begin{align}\label{a atom eq 9}
&\int_{{\mathbb{R}^n}} \left[\sum_{k \in{\mathbb{Z}}} \lambda_k \left\| W(x) A_{Q_k}^{-1} \right\|  \left|T\vec{a}_k(x)\right| {\mathbf{1}}_{2\sqrt{n} Q_k}(x)\right]^r |g(x)|\,dx{\nonumber}\\
&\quad \leq \int_{{\mathbb{R}^n}} \sum_{k \in{\mathbb{Z}}} \lambda_k^r \left\| W(x) A_{Q_k}^{-1} \right\|^r  \left|T\vec{a}_k(x)\right|^r {\mathbf{1}}_{2\sqrt{n} Q_k}(x) |g(x)|\,dx{\nonumber}\\
&\quad = \sum_{k \in{\mathbb{Z}}} \lambda_k^r \int_{2\sqrt{n} Q_k}  \left\| W(x) A_{Q_k}^{-1} \right\|^r  \left|T\vec{a}_k(x)\right|^r  |g(x)|\,dx.
\end{align}
Using H\"older's inequality and the boundedness of $T$ on $L^q$,
we further conclude that, for any $k\in{\mathbb{Z}}$,
\begin{align}\label{a atom eq 8}
&\int_{{\mathbb{R}^n}} \left\| W(x) A_{Q_k}^{-1} \right\|^r  \left|T\vec{a}_k(x)\right|^r {\mathbf{1}}_{2\sqrt{n} Q_k}(x) |g(x)|\,dx{\nonumber}\\
&\quad \leq \left[\int_{2\sqrt{n} Q_k} \left\| W(x) A_{Q_k}^{-1} \right\|^{r(\frac{q}{r})'} |g(x)|^{(\frac{q}{r})'}\,dx \right]^{1 - \frac{r}{q}}
\left\{\int_{{\mathbb{R}^n}} \left|T\vec{a}_k(x)\right|^{q} \,dx\right\}^{\frac{r}{q}}{\nonumber}\\
&\quad \lesssim \left[\int_{2\sqrt{n} Q_k} \left\| W(x) A_{Q_k}^{-1} \right\|^{r(\frac{q}{r})'} |g(x)|^{(\frac{q}{r})'}\,dx \right]^{1 - \frac{r}{q}}
\left\|\vec{a}_k\right\|_{L^q}^r {\nonumber}\\
&\quad \lesssim \left[\fint_{2\sqrt{n} Q_k} \left\| W(x) A_{Q_k}^{-1} \right\|^{r(\frac{q}{r})'} |g(x)|^{(\frac{q}{r})'}\,dx \right]^{1 - \frac{r}{q}}
|Q_k|^{1 - \frac{r}{q} }\left\|\vec{a}_k\right\|_{L^q}^r.
\end{align}
Now, let $d(\cdot)\in{\mathcal P}_0$ be such that $\frac{1}{d(\cdot)} := 1 - \frac{r}{q} - \frac{r}{r_W p(\cdot)}$.
Since $q\in (r_W'p_+,\infty)$, it follows that $d(\cdot) < [\frac{p(\cdot)}{r}]' $.
This, together with Lemmas \ref{est Q}, \ref{Holder}, and \ref{con f}, further implies that
\begin{align*}
\left[\fint_{2\sqrt{n} Q_k} \left\| W(x) A_{Q_k}^{-1} \right\|^{r(\frac{q}{r})'} |g(x)|^{(\frac{q}{r})'}\,dx \right]^{1 - \frac{r}{q}}
& \lesssim \left[\frac{\|\, \| W(\cdot) A_{Q_k}^{-1}\|{\mathbf{1}}_{2\sqrt{n} Q_k} \|_{L^{r_Wp(\cdot)}}}{\|{\mathbf{1}}_{2\sqrt{n} Q_k}\|_{L^{r_W p(\cdot)}}}\right]^r
\frac{\| g {\mathbf{1}}_{2\sqrt{n} Q_k} \|_{L^{d(\cdot)}}}{\|{\mathbf{1}}_{2\sqrt{n} Q_k}\|_{L^{d(\cdot)}}} .
\end{align*}
Applying this and Lemmas \ref{reverse Holder}, \ref{eq reduc M}, and \ref{QP5} yields
\begin{align*}
&\left[\fint_{2\sqrt{n} Q_k} \left\| W(x) A_{Q_k}^{-1} \right\|^{r(\frac{q}{r})'} |g(x)|^{(\frac{q}{r})'}\,dx \right]^{1 - \frac{r}{q}} \\
&\quad \lesssim \left[\frac{\|\, \| W(\cdot) A_{Q_k}^{-1}\|{\mathbf{1}}_{2\sqrt{n} Q_k} \|_{L^{p(\cdot)}}}{\|{\mathbf{1}}_{2\sqrt{n} Q_k}\|_{L^{p(\cdot)}}}\right]^r
\frac{\| g {\mathbf{1}}_{2\sqrt{n} Q_k} \|_{L^{d(\cdot)}}}{\|{\mathbf{1}}_{2\sqrt{n} Q_k}\|_{L^{d(\cdot)}}}\\
&\quad \lesssim \left\| A_{2\sqrt{n} Q_k} A_{Q_k}^{-1} \right\|^r \frac{\| g {\mathbf{1}}_{2\sqrt{n} Q_k} \|_{L^{d(\cdot)}}}{\|{\mathbf{1}}_{2\sqrt{n} Q_k}\|_{L^{d(\cdot)}}}
\lesssim \frac{\| g {\mathbf{1}}_{2\sqrt{n} Q_k} \|_{L^{d(\cdot)}}}{\|{\mathbf{1}}_{2\sqrt{n} Q_k}\|_{L^{d(\cdot)}}},
\end{align*}
which, combined with \eqref{a atom eq 8}, further implies that
\begin{align*}
&\int_{{\mathbb{R}^n}} \left\| W(x) A_{Q_k}^{-1} \right\|^r  \left|T\vec{a}_k(x)\right|^r {\mathbf{1}}_{2\sqrt{n} Q_k}(x) |g(x)|\,dx\\
&\quad \lesssim |Q_k|^{1 - \frac{r}{q} }\left\|\vec{a}_k\right\|_{L^q}^r \frac{\| g {\mathbf{1}}_{2\sqrt{n} Q_k} \|_{L^{d(\cdot)}}}{\|{\mathbf{1}}_{2\sqrt{n} Q_k}\|_{L^{d(\cdot)}}}
\leq \int_{Q_k} |Q_k|^{ - \frac{r}{q} }\left\|\vec{a}_k\right\|_{L^q}^r {\mathcal M}_{d(\cdot)} \left( g \right)(x)\,dx.
\end{align*}
From this, \eqref{a atom eq 9}, Tonelli's theorem,
Lemmas \ref{con f}, \ref{Holder}, and \ref{Mq bound},
and $d(\cdot) < [\frac{p(\cdot)}{r}]' $, we infer that
\begin{align*}
&\int_{{\mathbb{R}^n}} \left[\sum_{k \in{\mathbb{Z}}}\lambda_k \left\| W(x) A_{Q_k}^{-1} \right\|  \left|T \vec{a}_k(x)\right|{\mathbf{1}}_{2\sqrt{n} Q_k}(x)\right]^r |g(x)|\,dx\\
&\quad \lesssim \sum_{k \in{\mathbb{Z}}} \lambda_k^r \int_{Q_k} |Q_k|^{ - \frac{r}{q} }\left\|\vec{a}_k\right\|_{L^q}^r {\mathcal M}_{d(\cdot)} \left( g \right)(x)\,dx \\
&\quad  = \int_{{\mathbb{R}^n}} \sum_{k \in{\mathbb{Z}}} \left(|Q_k|^{-\frac1q}\lambda_k\left\|\vec{a}_k\right\|_{L^q} \right)^r {\mathbf{1}}_{Q_k} {\mathcal M}_{d(\cdot)} \left( g \right)(x)\,dx\\
&\quad \lesssim \left\|\sum_{k \in{\mathbb{Z}}} \left(|Q_k|^{-\frac1q}\lambda_k\left\|\vec{a}_k\right\|_{L^q} \right)^r {\mathbf{1}}_{Q_k} \right\|_{L^{\frac{p(\cdot)}{r}}} \left\|{\mathcal M}_{d(\cdot)} \left( g \right)\right\|_{L^{(\frac{p(\cdot)}{r})'}}\\
&\quad \lesssim \left\|\sum_{k \in{\mathbb{Z}}} \left(|Q_k|^{-\frac1q} \lambda_k\left\|\vec{a}_k\right\|_{L^q} \right)^r {\mathbf{1}}_{Q_k} \right\|_{L^{\frac{p(\cdot)}{r}}} \left\|g\right\|_{L^{(\frac{p(\cdot)}{r})'}},
\end{align*}
which, together with \eqref{a atom eq 10} and Lemma \ref{con f}, further implies that
\begin{align*}
\left\| \sum_{k \in\mathbb Z} \lambda_k \left\| W(\cdot) A_{Q_k}^{-1} \right\| \left| T\vec{a}_k\right|{\mathbf{1}}_{2\sqrt{n} Q_k} \right\|_{L^{p(\cdot)}}
& \lesssim \left\|\left[\sum_{k \in{\mathbb{Z}}} \left(|Q_k|^{-\frac1q}\lambda_k\left\|\vec{a}_k\right\|_{L^q} \right)^r {\mathbf{1}}_{Q_k} \right]^\frac1r\right\|_{L^{p(\cdot)}}.
\end{align*}
This finishes the proof of Lemma \ref{a atom eq le}.
\end{proof}

\begin{lemma}\label{W f var}
Let $p(\cdot) \in {\mathcal P}_0\cap LH$ and $N\in{\mathbb{Z}}_+$.
If $\varphi \in {\mathcal S}$ satisfies that ${\mathop\mathrm{\,supp\,}} \varphi \subset B(x_0,t)$
for some $x_0 \in {\mathbb{R}^n}$ and some $t\in (0,\infty)$
and, for any $\alpha \in {\mathbb{Z}}_+^n$ with $|\alpha| \leq N+1$,
\begin{align}\label{W f var eq 2}
\sup_{x\in{\mathbb{R}^n}} \left|\partial^\alpha \varphi(x)\right| \leq t^{-(n + |\alpha|)},
\end{align}
then there exists a positive constant $C$ such that,
for any $\vec{f} \in ({\mathcal S}')^m$ and $x\in{\mathbb{R}^n}$,
\begin{align}\label{W f var eq 1}
{\mathcal K}\left(\left\langle\vec{f},\varphi\right\rangle\right) \subset C \left( 2 + \frac{|x-x_0|}{t} \right)^{N+n+1} M_N^{{\mathcal K}} \left(\vec{f}\right)(x).
\end{align}
\end{lemma}
\begin{proof}
Let $C_{(x)} := ( 2 + \frac{|x-x_0|}{t} )^{N+n+1} \|\varphi\|_{{\mathcal S}_N} $ and $\phi^{(x)}(\cdot) := \frac{t^n}{C_{(x)}} \varphi(x - t\cdot)$.
Then it is obvious that, for any $x\in{\mathbb{R}^n}$,
$$\left\langle\vec{f},\varphi\right\rangle = C_{(x)}\phi^{(x)}_t \ast \vec{f}(x).$$
Therefore, using the definition of $M_N^{{\mathcal K}}$, we find that, to show \eqref{W f var eq 1},
it suffices to prove $\phi^{(x)} \in {\mathcal S}_N$.
By the assumption ${\mathop\mathrm{\,supp\,}} \varphi \subset B(x_0,t)$ and \eqref{W f var eq 2},
we conclude that, for any $\alpha\in{\mathbb{Z}}_+^n$ with $|\alpha| \leq N+1$,
\begin{align*}
&\sup_{y\in{\mathbb{R}^n}} \left(1+|y|\right)^{N+n+1} \left|\partial^\alpha \phi^{(x)}(y)\right|\\
&\quad= \frac{t^{n+|\alpha|}}{C_{(x)}} \sup_{y\in B(x_0,t)} \left( 1 + \frac{|x-y|}{t} \right)^{N+n+1} \left|\partial^\alpha \varphi(y)\right|\\
&\quad\leq \left( 2 + \frac{|x-x_0|}{t} \right)^{-(N+n+1)} \sup_{y\in B(x_0,t)} \left( 1+ \frac{|x-x_0|+|x_0-y|}{t} \right)^{N+n+1} \leq 1,
\end{align*}
which further implies that $\phi^{(x)} \in {\mathcal S}_N$ and hence completes the proof of Lemma \ref{W f var}.
\end{proof}

Now, let $t \in \{0,\frac13\}^n$.
Then the dyadic grid $\mathscr{Q}^t$ is defined by setting
$$ \mathscr{Q}^t := \left\{ 2^k\left([0,1)^n + m + (-1)^kt\right):\ k\in\mathbb{Z},
\ m\in\mathbb{Z}^n \right\}. $$
The following is known as the ``$\frac13$''-trick
(see, for instance, \cite[Lemma 4.3]{np25}).
\begin{lemma}\label{Q Qt}
For any cube $Q$ in $\mathbb{R}^n$,
there exist $t \in \{0,\frac13\}^n$ and $Q_t \in \mathscr{Q}^t$
such that $Q\subset Q_t$ and $l(Q_t) \leq 6l(Q)$.
\end{lemma}
Let ${\mathscr Q}$ be a given dyadic lattice in ${\mathbb{R}^n}$.
Then, for any dyadic cube $Q\in {\mathscr Q}$, let $s_Q := -\log_2 (l(Q))$ and $Q^\ast := \frac98 Q$.
The following construction of the stopping collection essentially comes from \cite{ccdo17},
which can be regarded as a \emph{refined Whitney decomposition}.
For the sake of completeness, we give some details here.
\begin{lemma}\label{stop coll}
Let ${\mathscr Q}$ be a fixed dyadic lattice in ${\mathbb{R}^n}$ and $\{Q_k\}_{k\in{\mathbb{N}}} \subset {\mathscr Q}$
be a sequence of pairwise disjoint cubes.
Then, for any open set $E \subset \bigcup_{k\in{\mathbb{N}}} 3Q_k$,
there exists a collection $S \subset {\mathscr Q}$ of dyadic cubes
such that the following properties hold:
\begin{itemize}
\item[{\rm (i)}] For any $L, L'\in S$, if $L\neq L'$, then $L\cap L' = \emptyset$.
\item[{\rm (ii)}] $E = \bigcup_{L\in S} L = \bigcup_{L\in S} 9L$.
\item[{\rm (iii)}] For any $L\in S$, $32L \cap E^{\complement} \neq \emptyset$.
\item[{\rm (iv)}] For any $L,L'\in S$, if $7L\cap 7L' \neq \emptyset$, then $|s_L - s_{L'}| < 8$.
\end{itemize}

\begin{itemize}
\item[{\rm (v)}] If $|3Q_k\cap E| < 2^{-4n}|Q_k|$ for any $k\in{\mathbb{N}}$, then,
for any $L\in S$, if $L^\ast \cap Q_k^\ast\neq \emptyset$ for some $k\in{\mathbb{N}}$,
then $32L \subset 3Q_k$ and $32L\cap (3Q_k\setminus E) \neq \emptyset$.
\end{itemize}
\end{lemma}
\begin{proof}
Since $E$ is an open set, it follows that, for any $x\in E$, there exists $Q_x\in {\mathscr Q}$ such that
$x\in Q_x$ and $9Q_x \subset E$.
Now, by the definition of ${\mathscr Q}$, let
$$ S := \left\{ L\in {\mathscr Q}:\ L \text{ is the maximal dyadic cube such that}\ 9L\subset E \right\}, $$
and hence $E = \bigcup_{L\in S} L = \bigcup_{L\in S} 9L$,
which completes the proof of {\rm (ii)}.

Using the maximality of $L$, we find that, for any $L,L'\in S$,
if $L\cap L' \neq \emptyset$, then $L = L'$. This finishes the proof of {\rm (i)}.

Next, we show {\rm (iii)} and {\rm (iv)}.
Observe that, by the construction of $S$, for any $L\in S$,
the dyadic parent $\widehat{L} \notin S$, which further implies that
$9\widehat{L} \cap E^\complement \neq \emptyset$.
Applying this and the fact that $9\widehat{L}\subset 32L$,
we conclude that $ 32L\cap  E^\complement \neq \emptyset$ and hence {\rm (iii)}.
Then we prove {\rm (iv)}.
Assume that $L,L'\in S$ and $7L\cap 7L'\neq \emptyset$ satisfies $|s_L - s_{L'}| \geq 8$
(without loss of generality, we may assume that $s_L \geq s_{L'} + 8$).
Then we obtain $32 L \subset 9L'\subset E$,
which contradicts {\rm (iii)} and hence
completes the proof of {\rm (iv)}.

Finally, we show {\rm (v)}.
Assume that $L\in S$ satisfies $L^\ast \cap Q_k^\ast \neq \emptyset$ for some $k\in{\mathbb{N}}$.
Then we find that $L^\ast \cap 3Q_k \neq \emptyset$ and hence $3L\cap 3Q_k \neq \emptyset$.
Since $3L$ (resp. $3Q_k$) is the $3^n$ translates of the dyadic cube $L$ (resp. $Q_k$),
we infer that there exists at least one dyadic cube $L'\in {\mathscr Q}$ (resp. $Q'\in{\mathscr Q}$) satisfying
$l(L') = l(L)$ (resp. $l(Q') = l(Q_k)$) and  $L'\subset 3L$ (resp. $Q'\subset 3Q_k$)
such that $L'\cap Q' \neq \emptyset$.
Note that, by the construction of $S$, $9L\subset E$ and hence $L'\subset E$.
Applying this and the assumption that $ |3Q_k\cap E| < 2^{-4n}|Q_k|$ for any $k\in{\mathbb{N}}$,
we conclude that $L'\subset Q'$ and $l(L') < \frac{1}{16}l(Q_k)$,
which further implies that $l(L) = l(L') < \frac{1}{16}l(Q_k)$.
Using this, we find that $ 32L\subset 3Q_k $, which, together with {\rm (iii)}, further implies that
$32L \cap [3Q_k\setminus E] \neq \emptyset$. This finishes the proof of Lemma \ref{stop coll}.
\end{proof}

We deduce the following lemma directly from
\cite[p.\,102]{s93}; we omit the details here.
\begin{lemma}\label{stop eta}
Let $E$ and $S$ be the same as in Lemma \ref{stop coll}.
Then there exists a sequence of measurable functions
$\{\eta_L:\ {\mathbb{R}^n}\to [0,1]\}_{L\in S} \subset C^\infty_{\rm c}$
such that
${\mathbf{1}}_{E} = \sum_{L\in S} \eta_L $,
${\mathop\mathrm{\,supp\,}} \eta_L \subset L^\ast$
for any $L\in S$,
\begin{align}\label{1220}
\sup_{x\in{\mathbb{R}^n}} |\partial^\alpha \eta_L(x)| \lesssim [l(L)]^{-|\alpha|}
\end{align}
for any $L\in S$ and $\alpha\in{\mathbb{Z}}_+^n$
with the implicit positive constant depending only on $\alpha$,
and
\begin{align}\label{1221}
\int_{{\mathbb{R}^n}} \eta_L(x) \,dx \sim |L|
\end{align}
for any $L\in S, $
where the positive equivalence constants depend only on $n$.
\end{lemma}
Let $\eta_L$ be the same as in Lemma \ref{stop eta}
and
$$\widetilde{\eta}_L := \frac{\eta_L}{\int_{{\mathbb{R}^n}} \eta_L(x)\,dx}.$$
For any $s\in{\mathbb{Z}}_+$, let ${\mathscr P}_s$ be the set of all polynomials on ${\mathbb{R}^n}$ of total
degree not greater than $s$.
Let ${\mathcal H}_s$ be ${\mathscr P}_s$ regarded as a subspace of the Hilbert space $L^2(L^\ast,\widetilde{\eta}_L\,dx)$.
The following is precisely presented in \cite[p.\,104]{s93}.
\begin{lemma}\label{stop e}
Let $L$ be a cube and $\eta_L$ satisfy ${\mathop\mathrm{\,supp\,}} \eta_L \subset L^\ast$,
\eqref{1220}, and \eqref{1221}.
Then, for any $s\in {\mathbb{Z}}_+$, there exists a sequence of polynomials
$\{e^{(L)}_i\}_{i = 1}^M \subset {\mathscr P}_s$, where $M\in{\mathbb{N}}$ depends only on $s$ and $n$,
such that $\{e^{(L)}_i\}_{i = 1}^M$ is an orthonormal basis of ${\mathcal H}_s$,
that is, for any $i,j\in\{1,\dots,M\}$,
\begin{align*}
\left\langle e_i^{(L)},e_j^{(L)}\widetilde{\eta}_L \right\rangle =
\begin{cases}
\displaystyle 1 &\text{if}\ i = j, \\
\displaystyle 0 &\text{otherwise}
\end{cases}
\end{align*}
and, moreover, for any $i\in\{1,\dots,M\}$ and $\alpha\in{\mathbb{Z}}_+^n$,
\begin{align*}
\sup_{x\in{\mathbb{R}^n}}|\partial^\alpha e_i^{(L)}(x)| \lesssim [l(L)]^{-|\alpha|},
\end{align*}
where the implicit positive constant depends only on $\alpha$.
\end{lemma}
Applying this and Lemma \ref{W f var}, we conclude the following result.
\begin{lemma}\label{stop P}
Let $L$, $\eta_L$, $s$, and $\{e^{(L)}_i\}_{i = 1}^M$ be the same
as in Lemma \ref{stop e}.
For any $\vec{h} \in ({\mathcal S}')^m$, let
$$P_L(\vec{h}) := \sum_{i = 1}^M \left\langle \vec{h}, e^{(L)}_i \widetilde{\eta}_L\right\rangle e_i^{(L)}.$$
Then, for any $q\in {\mathscr P}_s$,
$$\left\langle \left[\vec{h} - P_L\left(\vec{h}\right)\right]\eta_L, q \right\rangle = 0. $$
Moreover, for any given $N\in\mathbb Z_+$, there exists a positive constant $C$,
independent of $L$ and $\vec{h}$,
such that, for any $i\in\{1,\ldots,M\}$ and $y\in {\mathbb{R}^n}$,
\begin{align*}
\left\langle \vec{h}, e_i^{(L)}\widetilde{\eta}_L \right\rangle \in C \left( 2 + \frac{|y-c_L|}{l(L)} \right)^{N+n+1} M_N^{{\mathcal K}} \left(\vec{h}\right)(y),
\end{align*}
where $c_L$ is the center of $L$.
\end{lemma}
\begin{proof}
By Lemma \ref{stop e}, we have,
for any $q := \sum_{j = 1}^M c_je^{(L)}_j\in{\mathscr P}_s$ and $\vec{h} \in ({\mathcal S}')^m$,
\begin{align*}
\left\langle \left[ P_L \vec{h} - \vec{h} \right]\eta_L,q \right\rangle
&= \sum_{j = 1}^M c_j \left\langle P_L \vec{h}, e^{(L)}_j\eta_L \right\rangle - \sum_{j = 1}^M c_j \left\langle \vec{h},e^{(L)}_j \eta_L \right\rangle{\nonumber}\\
& = \sum_{j = 1}^M c_j \left\langle \vec{h}, e^{(L)}_j\widetilde{\eta}_L \right\rangle \left\langle e^{(L)}_j,e^{(L)}_j\eta_L \right\rangle - \sum_{j = 1}^M c_j \left\langle \vec{h},e^{(L)}_j \eta_L \right\rangle
= \vec{0}.
\end{align*}
Then, using Lemmas \ref{stop eta} and \ref{stop e},
we obtain, for any $\alpha\in{\mathbb{Z}}_+^n$ with $|\alpha| \leq N+1$ and for any $i\in\{1,\dots,M\}$,
\begin{align*}
\sup_{x\in{\mathbb{R}^n}} \left|\partial^\alpha \left( e^{(L)}_i \widetilde{\eta}_L \right)(x)\right|
\sim \sup_{x\in{\mathbb{R}^n}} \left|\sum_{\beta \leq \alpha} \partial^\beta e^{(L)}_i(x) \partial^{\alpha-\beta} \widetilde{\eta}_L(x) \right|
\lesssim \sum_{\beta \leq \alpha} [l(L)]^{-n-|\beta| - |\alpha-\beta|}
\sim [l(L)]^{-(n+|\alpha|)},
\end{align*}
which, together with Lemma \ref{W f var}, further implies that
\begin{align*}
\left\langle \vec{h}, e_i^{(L)}\widetilde{\eta}_L \right\rangle \in C \left( 2 + \frac{|y-c_L|}{l(L)} \right)^{N+n+1} M_N^{{\mathcal K}} \left(\vec{h}\right)(y).
\end{align*}
This finishes the proof of Lemma \ref{stop P}.
\end{proof}
\begin{lemma}\label{stop f}
Let $L$ and $\eta_L$ be the same as in Lemma \ref{stop e}
and let $N\in\mathbb Z_+$.
Assume that $\vec{f} \in (L^{1}_{\rm loc})^m \cap H^{p(\cdot)}_W$
and let $\psi\in{\mathcal S}$ be supported in $B(\mathbf{0},1)$ and $\int_{{\mathbb{R}^n}} \psi(x)\,dx = 1$.
Then there exists a positive constant $C$ such that, for any $x\in L^\ast$,
\begin{align*}
M^{{\mathcal K}} \left( \vec{f}\eta_L,\psi \right)(x) \subset C M^{{\mathcal K}}_N \left( \vec{f} \right)(x).
\end{align*}
Moreover, letting $P_L$ be as in Lemma \ref{stop P},
then, for any $x\in (L^\ast)^\complement$ and $y\in {\mathbb{R}^n}$,
\begin{align*}
M^{{\mathcal K}} \left( \left[\vec{f} - P_L\left(\vec{f}\right)\right]\eta_L,\psi \right)(x) \subset
C \left( \frac{l(L)}{l(L) + |x - c_L|} \right)^{n+s+1} \left( 2 + \frac{|y - c_L|}{l(L)} \right)^{N+n+1}
 M^{{\mathcal K}}_N \left( \vec{f}\right)(y).
\end{align*}
\end{lemma}
\begin{proof}
We first consider the case $x\in L^\ast$.
Indeed, for any $x\in{\mathbb{R}^n}$,
\begin{align}\label{stop f eq 1}
M^{{\mathcal K}} \left( \vec{f}\eta_L ,\psi \right)(x)
& = \overline{{\mathop\mathrm{\,conv\,}}} \bigcup_{t\in (0,\infty)} {\mathcal K}\left(\psi_t\ast \left(\vec{f}\eta_L\right)(x)\right) {\nonumber}\\
&\subset \overline{{\mathop\mathrm{\,conv\,}}} \bigcup_{t\in (0,l(L)]} {\mathcal K}\left(\psi_t\ast \left(\vec{f}\eta_L\right)(x)\right)
+ \overline{{\mathop\mathrm{\,conv\,}}} \bigcup_{t\in (l(L),\infty)} {\mathcal K}\left( \psi_t\ast \left(\vec{f}\eta_L\right)(x)\right) {\nonumber}\\
&=: I_1(x) + I_2(x).
\end{align}
For any $y\in{\mathbb{R}^n}$, $x\in L^\ast$, and $t\in (0,\infty)$,
let $\zeta_{(t,x)}(y) := \psi_t(x - y)\eta_L(y)$, and hence
$$ \psi_t \ast(\vec{f} \eta_L)(x)
= \langle \vec{f}, \psi_t(x - \cdot)\eta_L \rangle . $$
Then, from the assumptions that ${\mathop\mathrm{\,supp\,}} \psi \subset B(\mathbf{0},1)$
and $\eta_L,\psi\in{\mathcal S}$,
it follows that $\zeta_{(t,x)} \in {\mathcal S}$ and $ {\mathop\mathrm{\,supp\,}} \zeta_{(t,x)} \subset B(x,t) $.
By this and Lemma \ref{stop eta}, we find that, for any $t\in (0,l(L)]$
and $\alpha\in{\mathbb{Z}}_+^n$ with $|\alpha| \leq N+1$,
\begin{align*}
\sup_{y\in{\mathbb{R}^n}} \left|\partial^\alpha \zeta_{(t,x)}(y)\right|
&\sim \sup_{y\in{\mathbb{R}^n}} \left| \sum_{\beta\leq \alpha} t^{-(n + |\beta|)}
\partial^\beta \psi\left( \frac{x-y}{t} \right) \partial^{\alpha-\beta} \eta_L(y) \right|{\nonumber} \\
&\lesssim \sum_{\beta\leq \alpha} t^{-(n + |\beta|)} [l(L)]^{-|\alpha-\beta|} \lesssim t^{-(n+|\alpha|)},
\end{align*}
which, combined with Lemma \ref{W f var},
further implies that there exists a positive constant $C$ such that, for any $x\in L^\ast$,
$\psi_t \ast(\vec{f} \eta_L)(x) \in C M_N^{{\mathcal K}} (\vec{f})(x)$.
This further yields, for any $x\in L^\ast$,
\begin{align*}
I_1(x) \subset C M_N^{{\mathcal K}} \left(\vec{f}\right)(x).
\end{align*}
If $t\in (l(L),\infty)$,
since ${\mathop\mathrm{\,supp\,}} \eta_L \subset L^\ast $,
we infer that
${\mathop\mathrm{\,supp\,}} \zeta_{(t,x)} \subset {\mathop\mathrm{\,supp\,}} \eta_L \subset L^\ast. $
Using this and Lemma \ref{stop eta},
we obtain, for any $\alpha\in{\mathbb{Z}}_+^n$ with $|\alpha| \leq N+1$
\begin{align*}
\sup_{y\in{\mathbb{R}^n}} \left|\partial^\alpha \zeta_{(t,x)}(y)\right|
&\sim \sup_{y\in{\mathbb{R}^n}} \left| \sum_{\beta\leq \alpha} t^{-(n + |\beta|)} \partial^\beta \psi\left( \frac{x-y}{t} \right) \partial^{\alpha-\beta} \eta_L(y) \right|{\nonumber}\\
&\lesssim \sum_{\beta\leq \alpha} t^{-(n + |\beta|)} [l(L)]^{-|\alpha-\beta|} \lesssim [l(L)]^{-(n+|\alpha|)},
\end{align*}
which, together with Lemma \ref{W f var},
further implies that there exists a positive constant $C$ such that,
for any $x \in L^\ast$,
$\psi_t \ast(\vec{f} \eta_L)(x) \in C M_N^{{\mathcal K}} (\vec{f})(x)$
and hence, for any $x \in L^\ast$,
\begin{align*}
I_2(x) \in C M_N^{{\mathcal K}} \left(\vec{f}\right)(x).
\end{align*}
Applying this and \eqref{stop f eq 1},
we conclude that, for any $x\in L^\ast$,
\begin{align*}
M^{{\mathcal K}} \left( \vec{f}\eta_L,\psi \right)(x) \subset C M^{{\mathcal K}}_N \left( \vec{f} \right)(x).
\end{align*}

Then we consider the case $x\in (L^\ast)^\complement$.
Let $q_{(t,x)}$ be the $s$-th degree Taylor polynomial of $\psi_t(x - \cdot)$ centered at $x$.
Then, from Lemma \ref{stop P}, it follows that,
for any $x\in (L^\ast)^\complement$ and $t\in (0,\infty)$,
\begin{align}\label{stop f eq 2}
\psi_t \ast \left(\left[\vec{f}  - P_L\left( \vec{f} \right) \right]\eta_L\right)(x)
& = \left\langle \vec{f} - P_L\left( \vec{f} \right), \eta_L(\cdot)\left[\psi_t(x-\cdot) - q_{(t,x)}\right] (\cdot) \right\rangle {\nonumber}\\
& = \left\langle \vec{f}, \phi_{(t,x)}\right\rangle
- \left\langle P_L\left( \vec{f} \right), \phi_{(t,x)} \right\rangle,
\end{align}
where $\phi_{(t,x)}(\cdot) := \eta_L(\cdot)[\psi_t(x - \cdot) - q_{(t,x)}(\cdot)]$.
Using the definition of $\phi_{(t,x)}$ and the assumption that ${\mathop\mathrm{\,supp\,}} \eta_L\subset L^\ast$,
we find that, for any $t\in(0,\infty)$ and $x\in (L^\ast)^\complement$,
\begin{align*}
{\mathop\mathrm{\,supp\,}} \phi_{(t,x)} \subset {\mathop\mathrm{\,supp\,}} \eta_L \subset B(c_L, r_L),
\end{align*}
where $c_L$ is the center of $L$ and $r_L := \sqrt{n}l(L)$.
From this and the formula in \cite[p.\,105]{s93}, we infer that, for any $\alpha \in {\mathbb{Z}}_+^n$,
$t\in (0,\infty)$, and $x\in (L^\ast)^\complement$,
\begin{align*}
\sup_{y\in{\mathbb{R}^n}} \left|\partial^\alpha \phi_{(t,x)}(y)\right|
\lesssim \frac{l(L)^{n+s+1}}{|x-c_L|^{n+s+1}} l(L)^{-(n+|\alpha|)}
\sim \frac{l(L)^{n+s+1}}{|x-c_L|^{n+s+1}} r_L^{-(n+|\alpha|)},
\end{align*}
where the implicit positive constant is independent of $t$, $x$, and $L$.
Applying this and Lemma \ref{W f var},
we conclude that there exists a positive constant $C$, independent of $\vec{f}$,
$t$, $x$, and $L$,
such that, for any $t\in (0,\infty)$, $x\in (L^\ast)^{\complement}$, and $y\in{\mathbb{R}^n}$,
\begin{align}\label{stop f eq 3}
\left\langle \vec{f}, \phi_{(t,x)} \right\rangle
&\in C \frac{l(L)^{n+s+1}}{|x-c_L|^{n+s+1}} \left( 2 + \frac{ |y - c_L| }{l(L)} \right)^{N+n+1}
 M^{{\mathcal K}}_N \left( \vec{f} \right)(y).
\end{align}
Observe that, by the definition of $P_L$,
we find that, for any $t\in (0,\infty)$ and $x\in {\mathbb{R}^n}$,
\begin{align*}
\left\langle P_L \vec{f}, \phi_{(t,x)} \right\rangle
= \sum_{i = 1}^M \left\langle \vec{f},e_i^{(L)} \widetilde{\eta}_L \right\rangle \left\langle e_i^{(L)}, \phi_{(t,x)} \right\rangle ,
\end{align*}
which, combined with Lemmas \ref{stop e} and \ref{stop P},
further implies that there exists a positive constant $C$ such that, for any $y\in {\mathbb{R}^n}$,
\begin{align*}
\left\langle P_L \vec{f}, \phi_{(t,x)} \right\rangle
&\in \sum_{i = 1}^M {\mathcal K}\left(\left\langle \vec{f},e_i^{(L)} \widetilde{\eta}_L \right\rangle\right)
\left|\left\langle e_i^{(L)}, \phi_{(t,x)} \right\rangle\right|\\
& \subset C\frac{l(L)^{n+s+1}}{|x-c_L|^{n+s+1}} \left( 2 + \frac{ |y - c_L| }{l(L)} \right)^{N+n+1}
M^{{\mathcal K}}_N \left( \vec{f} \right)(y) .
\end{align*}
Applying this, \eqref{stop f eq 2}, and \eqref{stop f eq 3},
we conclude that there exists a positive constant $C$ such that, for any $y\in {\mathbb{R}^n}$,
\begin{align*}
M^{{\mathcal K}} \left( \left[\vec{f} - P_L\left(\vec{f}\right)\right]\eta_L,\psi \right)(x) \subset
C \left( \frac{l(L)}{l(L) + |x - c_L|} \right)^{n+s+1} \left( 2 + \frac{|y - c_L|}{l(L)} \right)^{N+n+1}
 M^{{\mathcal K}}_N \left( \vec{f}\right)(y).
\end{align*}
This finishes the proof of Lemma \ref{stop f}.
\end{proof}

Let
$$ \mathcal{S}_\infty := \left\{ \phi\in \mathcal{S}:\ \int_{\mathbb{R}^n} x^{\gamma} \phi(x)\,dx = 0\
\text{for any}\ \gamma\in \mathbb{Z}_+^n \right\} $$
and
$$\widehat{\mathcal{D}}_0 :=\left\{\phi \in {\mathcal S}_\infty:\  {\mathop\mathrm{\,supp\,}} \widehat{\phi}\ \ \text{is compact}\right\}.$$
We have the following dense subset of $H^{p(\cdot)}_W$
(see, for instance, \cite[Chapter VII, Theorem 1]{st89} for
the scalar-valued case).
\begin{proposition}\label{H dense}
Let $p(\cdot) \in {\mathcal P}_0\cap LH$ and $W\in {\mathscr A}_{p(\cdot),\infty}$.
Then $(\widehat{\mathcal{D}}_0)^m\cap H^{p(\cdot)}_W$ is dense in $H^{p(\cdot)}_W$.
\end{proposition}
To prove Proposition \ref{H dense}, for any $l\in(0,\infty)$,
we introduce the following \emph{convex body valued operator $M_l^{\mathcal K}$} by setting,
for any $\vec{f}\in ({\mathcal S}')^m$, $\phi\in{\mathcal S}$, $x\in{\mathbb{R}^n}$,
and $t\in (0,\infty)$,
\begin{align*}
M_l^{{\mathcal K}} (\vec{f},\phi,x,t) :=  \overline{{\mathop\mathrm{\,conv\,}}}\left(\bigcup_{y\in {\mathbb{R}^n}} {\mathcal K}\left(\vec{f}\ast\phi_t(y)\right)\left( 1 + \frac{|x-y|}{t} \right)^{-l}\right).
\end{align*}

\begin{lemma}\label{G Ml}
Let $l\in(0,\infty)$, $\vec{f}\in ({\mathcal S}')^m$, and $\phi\in{\mathcal S}$.
Then, for any $\psi \in {\mathcal S}$, there exists a positive constant $C$,
independent of $\vec{f}$ and $\phi$, such that,
for any $t,s\in (0,\infty)$ with $t < s$ and for any $x,y\in{\mathbb{R}^n}$ with $|x-y| < s$,
\begin{align*}
\left(\vec{f}\ast \phi_s\right) \ast \psi_t(y) \in C M^{{\mathcal K}}_l (\vec{f},\phi,x,s).
\end{align*}
\end{lemma}
\begin{proof}
For any $t\in (0,\infty)$ and $x,y \in {\mathbb{R}^n}$ with $|x-y| < s$,
\begin{align*}
\left(\vec{f}\ast \phi_s\right) \ast \psi_t(y)
& = \int_{{\mathbb{R}^n}} \left(\vec{f}\ast \phi_s\right)(y-z) \psi_t(z)\,dz \\
& = \int_{{\mathbb{R}^n}} \left(\vec{f}\ast \phi_s\right)(y-z) \left( 1 + \frac{|x-y+z|}{s} \right)^{-l} \left( 1 + \frac{|x-y+z|}{s} \right)^{l} \psi_t(z)\,dz\\
& \in \left|\int_{{\mathbb{R}^n}} \left( 1 + \frac{|x-y+z|}{s} \right)^{l} \psi_t(z)\,dz\right|M^{{\mathcal K}}_l (\vec{f},\phi,x,s)
\subset CM^{{\mathcal K}}_l (\vec{f},\phi,x,s).
\end{align*}
This finishes the proof of Lemma \ref{G Ml}.
\end{proof}
Now, we show Proposition \ref{H dense}
by borrowing some ideas from the proof of \cite[Chapter VII, Theorem 1]{st89}.
\begin{proof}[Proof of Proposition \ref{H dense}]
To prove Proposition \ref{H dense}, we first show that
the class of functions, whose Fourier transforms both have compact supports
and vanish in the origin, is dense in $H^{p(\cdot)}_W$.

Now, let $\phi \in {\mathcal S}$ with ${\mathop\mathrm{\,supp\,}} \widehat{\phi} \subset B(\mathbf{0},2)$
and $\widehat{\phi}(\xi) = 1$ for any $\xi \in B(\mathbf{0},1)$.
Then, for any $\vec{f} := (f_1, \dots,f_m) \in H^{p(\cdot)}_W$,
we claim that
\begin{align}\label{H dense eq 1}
\left\| \vec{f} - \vec{f}\ast \phi_s + \vec{f} \ast \phi_{\frac1s} \right\|_{H^{p(\cdot)}_W}
\to 0 \ \ \text{as}\ \ s\to 0^+,
\end{align}
here and thereafter, the \emph{notation $s\to0^+$} means that $s\in(0,\infty)$ and $s\to0$.
Indeed, for any $s\in (0,\infty)$, we have
\begin{align*}
\left\| \vec{f} - \vec{f}\ast \phi_s + \vec{f} \ast \phi_{\frac1s} \right\|_{H^{p(\cdot)}_W}
\lesssim \left\| \vec{f} - \vec{f}\ast \phi_s \right\|_{H^{p(\cdot)}_W}
+ \left\| \vec{f} \ast \phi_{\frac1s} \right\|_{H^{p(\cdot)}_W}
=: I(s) + II(s)
\end{align*}
and hence, to prove \eqref{H dense eq 1},
we only need to show that $I(s) \to 0$ and $II(s) \to 0$ as $s \to 0^+$.

We first prove that $I(s) \to 0$ as $s \to 0^+$.
Let $l\in(\frac{n}{\alpha_W},\infty)$ be sufficiently large.
Note that, for any $t,s\in (0,\infty)$,
$$ \left(\vec{f}\ast \phi_s\right) \ast \psi_t = \left(\vec{f}\ast \psi_t\right) \ast \phi_s .$$
By this and Lemma \ref{G Ml}, we conclude that there exists a positive constant $C$ such that,
for any $s,t \in (0,\infty)$ and $x,y\in {\mathbb{R}^n}$ with $|x-y| < t$,
if $s \geq t$, then
\begin{align*}
\left( \vec{f} \ast \phi_s \ast \psi_t \right)(y) \in C M^{{\mathcal K}}_l (\vec{f}, \phi , x, s)
\end{align*}
and, if $s < t$, then
\begin{align*}
\left( \vec{f} \ast \phi_s \ast \psi_t \right)(y) \in C M^{{\mathcal K}}_l (\vec{f}, \psi , x, t).
\end{align*}
From these and the definition of $M^{**,\mathcal K}_l$, it follows that there exists a positive constant $C$, independent of $\vec{f}$,
such that, for any $s\in (0,\infty)$ and $x\in{\mathbb{R}^n}$,
\begin{align}\label{2102}
M^{\ast,{\mathcal K}}_1 \left( \vec{f} \ast \phi_s, \psi \right)(x)
&= \overline{{\mathop\mathrm{\,conv\,}}}\left( \bigcup_{t\in (0,\infty)} \bigcup_{y\in B(x,t)} {\mathcal K}\left(\vec{f} \ast \phi_s \ast \psi_t\right) (y) \right){\nonumber}\\
& \subset C M^{{\mathcal K}}_l (\vec{f}, \phi , x, s) + C \overline{{\mathop\mathrm{\,conv\,}}}\left(\bigcup_{t\in [s,\infty)} M^{{\mathcal K}}_l (\vec{f}, \psi , x, t)\right){\nonumber}\\
&\subset C M^{\ast\ast,{\mathcal K}}_l \left(\vec{f}, \phi\right) (x) + C \overline{{\mathop\mathrm{\,conv\,}}}\left(\bigcup_{t\in [s,\infty)} M^{{\mathcal K}}_l (\vec{f}, \psi , x, t)\right),
\end{align}
where $M^{\ast,{\mathcal K}}_1$ is defined as in Definition \ref{def convex max} with $a := 1$,
which, together with Lemma \ref{A F}, further implies that
\begin{align}\label{H dense eq 3}
\left(M^{\ast}_1\right)_W \left( \vec{f} \ast \phi_s, \psi \right)(x)
&\lesssim \left(M^{\ast\ast}_l\right)_W \left(\vec{f}, \phi\right) (x) +  \sup_{t\in [s,\infty)} \left|W(x)M^{{\mathcal K}}_l (\vec{f}, \psi , x, t)\right|{\nonumber} \\
&\leq \left(M^{\ast\ast}_l\right)_W \left(\vec{f}, \phi\right) (x) + \left(M^{\ast\ast}_l\right)_W \left(\vec{f}, \psi\right) (x).
\end{align}
Using this, we find that, for any $s \in (0,\infty)$ and $x\in{\mathbb{R}^n}$,
\begin{align}\label{H dense eq 12}
(M_1^*)_W(\vec{f} - \vec{f} \ast \phi_s,\psi)(x)
&\leq (M_1^*)_W(\vec{f}, \psi)(x) + (M_1^*)_W(\vec{f} \ast \phi_s, \psi)(x){\nonumber}\\
& \lesssim (M_1^*)_W(\vec{f}, \psi)(x) + \left(M^{\ast\ast}_l\right)_W \left(\vec{f}, \phi\right) (x)
+ \left(M^{\ast\ast}_l\right)_W \left(\vec{f}, \psi\right) (x).
\end{align}
From Theorem \ref{H equal} and the assumption $\vec{f} \in H^{p(\cdot)}_W$,
we deduce that $(M^{\ast\ast}_l)_W (\vec{f}, \phi)$ and $(M^{\ast\ast}_l)_W (\vec{f}, \psi)$
are in $L^{p(\cdot)}$.
Hence, applying this, \eqref{H dense eq 12}, and the Lebesgue dominated convergence theorem
in the setting of variable Lebesgue spaces (see, for instance, \cite[Theorem 2.62]{cf13}),
we conclude that, to prove $ I(s) \to 0$ as $s \to 0^+$,
we only need to show that, for almost every $x \in {\mathbb{R}^n}$,
\begin{align*}
\left(M_1^*\right)_W\left(\vec{f} - \vec{f} \ast \phi_s,\psi\right)(x) \to 0 \ \ \text{as}\ \ s \to 0^+.
\end{align*}
Indeed, by the proof of
\cite[pp.\,107--108]{st89},
we immediately obtain, for any $k\in \{1,\dots,m\}$ and almost every $x\in\mathbb R^n$,
\begin{align*}
M_1^\ast \left(f_k - f_k \ast \phi_s,\psi\right)(x) \to 0\ \ \text{as}\ \ s\to 0^+,
\end{align*}
where $M_1^\ast$ denotes the scalar-valued unweighted
grand non-tangential maximal function.
This, combined with the definition of $M^{*,{\mathcal K}}_1$,
further implies that
$|M_1^{*,{\mathcal K}}(\vec{f}- \vec{f} \ast \phi_s,\psi)(x)| \to 0$ as $s\to 0^+$.
Using this, we conclude that, for almost every $x\in\mathbb R^n$,
\begin{align*}
\left(M_1^*\right)_W\left(\vec{f}- \vec{f} \ast \phi_s,\psi\right)(x)
\leq \|W(x)\| \left|M_1^{*,{\mathcal K}} \left(\vec f- \vec f \ast \phi_s,\psi\right)(x)\right|
\to 0\ \ \text{as}\ \ s\to 0^+
\end{align*}
and hence $I(s) \to 0$ as $s \to 0^+$.

Next, we estimate $II(s)$.
From \eqref{H dense eq 3}, we infer that,
for any $s\in (0,\infty)$ and $x \in {\mathbb{R}^n}$,
\begin{align*}
\left(M_1^\ast\right)_W\left(\vec{f} \ast \phi_{\frac1s},\psi\right)(x)
\lesssim \left(M_l^{\ast\ast}\right)_W\left(\vec{f} ,\phi\right)(x) + \left(M_l^{\ast\ast}\right)_W\left(\vec{f} ,\psi\right)(x).
\end{align*}
This, together with the the Lebesgue dominated convergence theorem
in the setting of variable Lebesgue spaces,
further implies that, to prove $II(s) \to 0$ as $s\to 0^+$,
it is sufficient to show that, for almost every $x\in {\mathbb{R}^n}$,
$(M_1^\ast)_W(\vec{f} \ast \phi_{\frac1s},\psi)(x) \to 0$ as $s \to 0^+, $
which is equivalent to
\begin{align*}
\left(M_1^\ast\right)_W\left(\vec{f} \ast \phi_{s},\psi\right)(x) \to 0 \ \ \text{as}\ \ s \to \infty.
\end{align*}
Let $x$ be any given point in ${\mathbb{R}^n}$ and,
for any $r \in (1,\infty)$, let $Q_r := Q(x,2r)$ and $Q_0 := Q(x,2)$.
Observe that, by the definition of $M^{{\mathcal K}}_l$,
for any $t\in (0,\infty)$ and $y\in{\mathbb{R}^n}$ with $y\in B(x,t)$,
there exists a positive constant $C$, depending only on $l$, such that
\begin{align}\label{H dense eq 4}
M^{{\mathcal K}}_l \left( \vec{f}, \phi,x,t \right)\subset CM^{{\mathcal K}}_l \left( \vec{f}, \phi,y,t \right).
\end{align}
Moreover, by Lemmas \ref{eq reduc M} and \ref{est Q}
and \cite[Corollary 2.28]{cf13}, we have,
for any $r\in (1,\infty)$,
\begin{align*}
\frac{\|{\mathbf{1}}_{Q_0}\|_{L^{p(\cdot)}}}{\|{\mathbf{1}}_{Q_r}\|_{L^{p(\cdot)}}} \left\|A_{Q_0} A_{Q_r}^{-1}\right\|
&\sim \frac{\|{\mathbf{1}}_{Q_0}\|_{L^{p(\cdot)}}}{\|{\mathbf{1}}_{Q_r}\|_{L^{p(\cdot)}}} \frac{1}{\|{\mathbf{1}}_{Q_0}\|_{L^{p(\cdot)}}} \left\|\,\left\| W(\cdot) A_{Q_r}^{-1} \right\|{\mathbf{1}}_{Q_0} \right\|_{L^{p(\cdot)}}\\
&\lesssim  \frac{\|{\mathbf{1}}_{Q_0}\|_{L^{p(\cdot)}}}{\|{\mathbf{1}}_{Q_r}\|_{L^{p(\cdot)}}}
\frac{\|{\mathbf{1}}_{Q_0}\|_{L^{(r_W)'p(\cdot)}}}{\|{\mathbf{1}}_{Q_0}\|_{L^{p(\cdot)}}} \left\|\,\left\| W(\cdot) A_{Q_r}^{-1} \right\|{\mathbf{1}}_{Q_0} \right\|_{L^{r_Wp(\cdot)}} \\
&\sim \frac{\|{\mathbf{1}}_{Q_0}\|_{L^{p(\cdot)}}}{\|{\mathbf{1}}_{Q_r}\|_{L^{p(\cdot)}}}
\frac{1}{\|{\mathbf{1}}_{Q_0}\|_{L^{r_Wp(\cdot)}}} \left\|\,\left\| W(\cdot) A_{Q_r}^{-1} \right\|{\mathbf{1}}_{Q_0} \right\|_{L^{r_Wp(\cdot)}} ,
\end{align*}
where $r_W$ is the same as in Lemma \ref{reverse Holder}.
Applying this and Lemmas \ref{con f}, \ref{reverse Holder}, and \ref{eq reduc M},
we conclude that
\begin{align}\label{H dense eq 13}
\frac{\|{\mathbf{1}}_{Q_0}\|_{L^{p(\cdot)}}}{\|{\mathbf{1}}_{Q_r}\|_{L^{p(\cdot)}}} \left\|A_{Q_0} A_{Q_r}^{-1}\right\|
&\lesssim \left[\frac{\|{\mathbf{1}}_{Q_0}\|_{L^{p(\cdot)}}}{\|{\mathbf{1}}_{Q_r}\|_{L^{p(\cdot)}}}\right]^{1 - \frac{1}{r_W}}
\frac{1}{\|{\mathbf{1}}_{Q_r}\|_{L^{r_Wp(\cdot)}}} \left\|\,\left\| W(\cdot) A_{Q_r}^{-1} \right\|{\mathbf{1}}_{Q_r} \right\|_{L^{r_Wp(\cdot)}}{\nonumber} \\
&\lesssim \left[\frac{\|{\mathbf{1}}_{Q_0}\|_{L^{p(\cdot)}}}{\|{\mathbf{1}}_{Q_r}\|_{L^{p(\cdot)}}}\right]^{1 - \frac{1}{r_W}}
\frac{1}{\|{\mathbf{1}}_{Q_r}\|_{L^{p(\cdot)}}} \left\|\,\left\| W(\cdot) A_{Q_r}^{-1} \right\|{\mathbf{1}}_{Q_r} \right\|_{L^{p(\cdot)}}{\nonumber}\\
&\sim \left[\frac{\|{\mathbf{1}}_{Q_0}\|_{L^{p(\cdot)}}}{\|{\mathbf{1}}_{Q_r}\|_{L^{p(\cdot)}}}\right]^{1 - \frac{1}{r_W}} \left\|A_{Q_r} A_{Q_r}^{-1}\right\|
 = \left[\frac{\|{\mathbf{1}}_{Q_0}\|_{L^{p(\cdot)}}}{\|{\mathbf{1}}_{Q_r}\|_{L^{p(\cdot)}}}\right]^{1 - \frac{1}{r_W}}.
\end{align}
Furthermore, from Lemma \ref{eq reduc M}, we deduce that,
for any convex body $F$ and any cube $Q$ in ${\mathbb{R}^n}$,
\begin{align*}
\left|A_Q F\right| = \sup_{\vec{z}\in F} \left|A_Q \vec{z}\right|
\sim \frac{1}{\|{\mathbf{1}}_Q\|_{L^{p(\cdot)}}} \sup_{\vec{z}\in F} \left\| W(\cdot)\vec{z} {\mathbf{1}}_Q \right\|_{L^{p(\cdot)}}
\leq \frac{1}{\|{\mathbf{1}}_Q\|_{L^{p(\cdot)}}} \left\| \left|W(\cdot)F\right|{\mathbf{1}}_Q \right\|_{L^{p(\cdot)}}.
\end{align*}
Combining this with \eqref{H dense eq 4}, \eqref{H dense eq 13}, and Theorem \ref{H equal},
we conclude that, for any $r\in (1,\infty)$,
\begin{align*}
&\left\|{\mathbf{1}}_{Q_0}\right\|_{L^{p(\cdot)}} \left|A_{Q_0} M^{{\mathcal K}}_l \left( \vec{f}, \phi,x,\sqrt{n}r \right)\right|\\
&\quad \leq \frac{\|{\mathbf{1}}_{Q_0}\|_{L^{p(\cdot)}}}{\|{\mathbf{1}}_{Q_r}\|_{L^{p(\cdot)}}} \left\|A_{Q_0} A_{Q_r}^{-1}\right\| \left\|{\mathbf{1}}_{Q_r}\right\|_{L^{p(\cdot)}}\left|A_{Q_r} M^{{\mathcal K}}_l \left( \vec{f}, \phi,x,\sqrt{n}r \right)\right|\\
&\quad \lesssim \left[\frac{\|{\mathbf{1}}_{Q_0}\|_{L^{p(\cdot)}}}{\|{\mathbf{1}}_{Q_r}\|_{L^{p(\cdot)}}}\right]^{1 - \frac{1}{r_W}} \left\|\, \left|W(\cdot)M^{{\mathcal K}}_l \left( \vec{f}, \phi,x,\sqrt{n}r \right)\right|{\mathbf{1}}_{Q_r} \right\|_{L^{p(\cdot)}}\\
&\quad \lesssim \left[\frac{\|{\mathbf{1}}_{Q_0}\|_{L^{p(\cdot)}}}{\|{\mathbf{1}}_{Q_r}\|_{L^{p(\cdot)}}}\right]^{1 - \frac{1}{r_W}} \left\|\, \left|W(\cdot)M^{{\mathcal K}}_l \left( \vec{f}, \phi,\cdot,\sqrt{n}r \right)\right|{\mathbf{1}}_{Q_r} \right\|_{L^{p(\cdot)}}\\
&\quad \leq \left[\frac{\|{\mathbf{1}}_{Q_0}\|_{L^{p(\cdot)}}}{\|{\mathbf{1}}_{Q_r}\|_{L^{p(\cdot)}}}\right]^{1 - \frac{1}{r_W}} \left\|\, \left|\left(M^{\ast\ast}_l\right)_W \left( \vec{f}, \phi\right)\right|{\mathbf{1}}_{Q_r} \right\|_{L^{p(\cdot)}}
\lesssim \left[\frac{\|{\mathbf{1}}_{Q_0}\|_{L^{p(\cdot)}}}{\|{\mathbf{1}}_{Q_r}\|_{L^{p(\cdot)}}}\right]^{1 - \frac{1}{r_W}} \left\| \vec{f} \right\|_{H^{p(\cdot)}_W},
\end{align*}
which further implies that
$\lim_{r\to \infty} |M_l^{{\mathcal K}} ( \vec{f}, \phi,x,r )| = 0. $
Using this and the second inclusion of \eqref{2102},
we conclude that, for almost every $x\in{\mathbb{R}^n}$,
\begin{align*}
(M_1^*)_W\left(\vec{f} \ast \phi_s,\psi\right)(x)
\lesssim \left|W(x) M_l^{{\mathcal K}} \left(\vec{f}, \phi , x, s\right)\right|
+ \sup_{t\in [s,\infty)} \left|W(x) M_l^{{\mathcal K}} \left(\vec{f}, \psi , x, t\right)\right| \to 0
\end{align*}
as $s \to \infty$. This finishes the estimate of $II(s)$
and hence the proof of the claim \eqref{H dense eq 1}.

Now, let $\vec{f}_{(s)} := \vec{f} \ast \phi_s - \vec{f} \ast \phi_{\frac1s}$.
Then, by the definition of $\phi$, we find that $\widehat{\vec{f}}_{(s)}$ has
a compact support and vanishes at the origin.
Moreover, using \eqref{H dense eq 1}, we conclude that
$\vec{f}_{(s)} \to \vec{f}$ in $H^{p(\cdot)}_W$ as $s\to0^+$.
Thus, the class of all functions $\vec g\in H^{p(\cdot)}_W$,
whose Fourier transforms have compact supports and vanish in a neighborhood
of the origin, is dense in $H^{p(\cdot)}_W$.

It remains to show that every element of this class can be approximated by
elements of $(\widehat{\mathcal D}_0)^m$ in $H^{p(\cdot)}_W$. To this end,
fix such a function $\vec g$. Let $\eta\in \mathcal{S}$ satisfy $\eta(0)=1$ and
${\mathop\mathrm{\,supp\,}} \widehat{\eta} \subset B(\mathbf{0},1)$.
For any $\epsilon\in (0,\infty)$ and $x\in{\mathbb{R}^n}$, let
$\vec{g}^{(\epsilon)}(x) := \vec{g}(x)\eta(\epsilon x)$
and hence $\widehat{\vec{g}^{(\epsilon)}} = \widehat{\vec{g}}\ast (\widehat{\eta})_\epsilon$.
By the compactness of ${\mathop\mathrm{\,supp\,}} \widehat{\vec g}$ and ${\mathop\mathrm{\,supp\,}}\widehat{\eta}$,
we conclude that $\widehat{\vec g^{(\epsilon)}}$ has compact support. Moreover, since
$\widehat{\vec g}$ vanishes in a neighborhood of the origin and
${\mathop\mathrm{\,supp\,}}\widehat{\eta}$ is compact, we deduce that
there exists a ball $B(\mathbf0,r)$ with $r\in(0,1)$
and $\epsilon_0\in(0,\infty)$ small enough
such that, for any $\epsilon\in(0,\epsilon_0)$,
\begin{align}\label{1041}
\widehat{\vec g^{(\epsilon)}}=0 \text{ in } B(\mathbf0,r).
\end{align}
Consequently,
$\vec g^{(\epsilon)}\in(\widehat{\mathcal D}_0)^m.$
Now, we claim that
$\|\vec g-\vec g^{(\epsilon)}\|_{H^{p(\cdot)}_W}\to0$ as
$\epsilon\to0^+.$
From the assumption that $ {\mathop\mathrm{\,supp\,}} \widehat{\vec g}$ is compact
and \cite[Theorem 2.3.21]{g14}, it follows that $\vec{g}$ is a function
and, moreover, there exists $M\in\mathbb Z_+$ such that,
for any $x\in{\mathbb{R}^n}$,
\begin{align}\label{H dense eq 10}
|\vec g(x)| \lesssim(1+|x|)^M,
\end{align}
where the implicit positive constant is independent of $x$.
Let $\phi\in {\mathcal S}$ satisfy that ${\mathop\mathrm{\,supp\,}} \widehat{\phi}$ has compact support
and $\int_{{\mathbb{R}^n}} \phi(x)\,dx = 1$.
Then, by Theorem \ref{H equal}, we only need to show that
\begin{align}\label{H dense eq 8}
\left\|M_W\left(\vec{g} - \vec{g}^{(\epsilon)}, \phi\right)\right\|_{L^{p(\cdot)}} \to 0
\ \ \text{as}\ \ \epsilon\to 0^+.
\end{align}
For any $\epsilon,t\in (0,\infty)$ and $x\in{\mathbb{R}^n}$, let
\begin{align}\label{H dense eq 5}
\vec{H}^{(\epsilon)}(x,t) := \left[ \left( \vec{g} - \vec{g}^{(\epsilon)} \right)  \ast \phi_t \right](x)
= \int_{{\mathbb{R}^n}} \vec{g}(x-y) \left[ 1 - \eta\left(\epsilon[x-y]\right) \right]\phi_t(y)\,dy.
\end{align}
Since $ \widehat{\phi} $ has compact support and \eqref{1041} holds
for any $\epsilon\in(0,\epsilon_0)$,
we infer that there exists $t_0\in (0,\infty)$ large enough
such that, for any $\epsilon\in (0,\epsilon_0)$, $t\in (t_0,\infty)$, and $\xi\in{\mathbb{R}^n}$,
$$
\left[ \widehat{\vec{g}}(\xi) - \widehat{\vec{g}^{(\epsilon)}}(\xi) \right] \widehat{\phi}(t\xi) =\vec{0},
$$
which further implies that, for any $\epsilon\in (0,\epsilon_0)$, $t\in (t_0,\infty)$, and $x\in{\mathbb{R}^n}$,
$ \vec{H}^{(\epsilon)}(x,t) = \vec{0}$.
This, combined with the definition of $(M_1)_W$,
further implies that, for any $\epsilon\in (0,\epsilon_0)$ and $x\in{\mathbb{R}^n}$,
\begin{align}\label{H dense eq 7}
M_W \left( \vec{g} - \vec{g}^{(\epsilon)}, \phi \right)(x)
= \sup_{t\in (0,t_0]} \left|W(x) \vec{H}^{(\epsilon)}(x,t)\right|.
\end{align}
Using Taylor's formula and $\eta \in {\mathcal S}$,
we obtain, for any $x,y\in{\mathbb{R}^n}$ and $\epsilon\in (0,\infty)$,
\begin{align}\label{H dense eq 6}
1 - \eta\left(\epsilon(x-y)\right)
&= 1- \eta(\epsilon x) - \sum_{0<|\alpha|<d}(-\epsilon)^{|\alpha|}
\frac{y^\alpha}{\alpha!}D^\alpha\eta(\epsilon x)
+R(\epsilon x,\epsilon y),
\end{align}
where $d\in\mathbb Z_+$ will be determined later. Here $R(\cdot,\cdot)$ is the Taylor remainder
satisfying, for any $x,y\in{\mathbb{R}^n}$,
$|R(x,y)|\lesssim |y|^d,$
where the implicit positive constant depends on $d$. Moreover, from
$\eta\in{\mathcal S}$ and \eqref{H dense eq 6},
we infer that, for any $\kappa\in(0,\infty)$ and $x,y\in{\mathbb{R}^n}$ with $|y|\leq |x|/2$,
\begin{align}\label{H dense eq 11}
|R(x,y)| \lesssim |y|^d(1+|x|)^{-\kappa},
\end{align}
where the implicit positive constant depends on $d$ and $\kappa$.
Using \eqref{H dense eq 5} and \eqref{H dense eq 6}, we obtain, for any
$x\in{\mathbb{R}^n}$ and $\epsilon,t\in(0,\infty)$,
\begin{align*}
\vec{H}^{(\epsilon)}(x,t)
&= [1-\eta(\epsilon x)]
\int_{{\mathbb{R}^n}}\phi_t(y)\vec g(x-y)\,dy  \\
&\quad
-\sum_{0<|\alpha|<d}
(-\epsilon t)^{|\alpha|}\frac{1}{\alpha!}D^\alpha\eta(\epsilon x)
\int_{{\mathbb{R}^n}}\left(\frac{y}{t}\right)^\alpha\phi_t(y)\vec g(x-y)\,dy \\
&\quad
+\int_{{\mathbb{R}^n}}R(\epsilon x,\epsilon y)\phi_t(y)\vec g(x-y)\,dy \\
&=: \vec{H}_0^{(\epsilon)}(x,t)
+\sum_{0<|\alpha|<d}\vec{H}_{1,\alpha}^{(\epsilon)}(x,t)
+ \vec{H}_2^{(\epsilon)}(x,t).
\end{align*}
This, together with \eqref{H dense eq 7}, further implies that, for any
$x\in{\mathbb{R}^n}$ and $\epsilon\in(0,\epsilon_0)$,
\begin{align}\label{H dense eq 14}
M_W\left( \vec{g} - \vec{g}^{(\epsilon)},\phi \right)(x)
& = \sup_{t\in(0,t_0]}\left|W(x)\vec{H}^{(\epsilon)}(x,t)\right|{\nonumber} \\
& \leq \sup_{t\in(0,t_0]}\left|W(x)\vec{H}_0^{(\epsilon)}(x,t)\right|
+\sum_{0<|\alpha|<d}\sup_{t\in(0,t_0]}
\left|W(x)\vec{H}_{1,\alpha}^{(\epsilon)}(x,t)\right|{\nonumber} \\
&\quad +\sup_{t\in(0,t_0]}\left|W(x)\vec{H}_2^{(\epsilon)}(x,t)\right|{\nonumber} \\
&=: N_0^{(\epsilon)}(x)
+\sum_{0<|\alpha|<d}N_{1,\alpha}^{(\epsilon)}(x)
+N_2^{(\epsilon)}(x).
\end{align}

We first estimate $N_0^{(\epsilon)}$. By the definition of $\vec{H}_0^{(\epsilon)}$,
for any $x\in{\mathbb{R}^n}$,
$$
N_0^{(\epsilon)}(x)
\leq \left|1-\eta(\epsilon x)\right|\sup_{t\in(0,t_0]}
\left|W(x)(\vec g\ast\phi_t)(x)\right|
\leq \left|1-\eta(\epsilon x)\right|M_W(\vec g,\phi)(x).
$$
Since $\eta(0)=1$ and $\eta\in \mathcal S$,
we deduce that, for any $x\in{\mathbb{R}^n}$, $|1-\eta(\epsilon x)|\to0$ as
$\epsilon\to0^+$. By Theorem \ref{H equal} and
the Lebesgue dominated convergence theorem on
variable Lebesgue spaces, we find that
\begin{align}\label{H dense eq N0}
\left\|N_0^{(\epsilon)}\right\|_{L^{p(\cdot)}}\to0
\quad\text{as } \epsilon\to0^+.
\end{align}

Next, we estimate $N_{1,\alpha}^{(\epsilon)}$. For any $\alpha\in{\mathbb{Z}}_+^n$ and $x\in{\mathbb{R}^n}$,
let $\phi^{(\alpha)}(x):=x^\alpha\phi(x)$. Using this, $\eta\in{\mathcal S}$, and the definition
of $\vec{H}_{1,\alpha}^{(\epsilon)}(x,t)$, we conclude that, for any
$\epsilon,t\in(0,\infty)$, $0<|\alpha|<d$, and $x\in{\mathbb{R}^n}$,
\begin{align}\label{H dense eq 9}
N_{1,\alpha}^{(\epsilon)}(x)
&=\sup_{t\in(0,t_0]}\left|W(x)\vec{H}_{1,\alpha}^{(\epsilon)}(x,t)\right|
\lesssim \epsilon^{|\alpha|}\sup_{t\in(0,t_0]}t^{|\alpha|}
\left|W(x)\vec g\ast\left(\phi^{(\alpha)}\right)_t(x)\right|
\lesssim \epsilon^{|\alpha|}M_W(\vec g,\phi^{(\alpha)})(x).
\end{align}
From the assumptions that $\phi\in{\mathcal S}$ and $\vec g\in H^{p(\cdot)}_W$, we infer that
$\|M_W(\vec g,\phi^{(\alpha)})\|_{L^{p(\cdot)}}<\infty,$
which, together with \eqref{H dense eq 9}, further implies that, for any $\alpha\in{\mathbb{Z}}_+^n$ with
$0<|\alpha|<d$,
\begin{align}\label{H dense eq N1}
\|N_{1,\alpha}^{(\epsilon)}\|_{L^{p(\cdot)}}\to0
\quad\text{as } \epsilon\to0^+.
\end{align}

Finally, we estimate $N_2^{(\epsilon)}$.
By Lemma \ref{f dec eq le}, we find that there exists $l'\in(0,\infty)$ large enough such that
\begin{align}\label{H dense eq W decay}
\left\|(1+|\cdot|)^{-l'}\|W(\cdot)\|\right\|_{L^{p(\cdot)}}<\infty.
\end{align}
Let $\kappa>M+l'$, $d>\kappa+1$, and $l>2M+d+n+l'.$
Using the conclusion in
\cite[pp.\,109--110]{st89}, \eqref{H dense eq 10}, and
\eqref{H dense eq 11}, we obtain, for any $\epsilon\in(0,\epsilon_0)$,
$t\in(0,t_0]$, and $x\in{\mathbb{R}^n}$,
\begin{align*}
\left|\vec{H}_2^{(\epsilon)}(x,t)\right|
\lesssim
t^d\left[
\epsilon^{d-\kappa}(1+|x|)^{M-\kappa}
+\epsilon^d(1+|x|)^{2M+d+n-l}
\right].
\end{align*}
From this and $t\in(0,t_0]$, we further deduce that
$$
N_2^{(\epsilon)}(x)
\lesssim
\left[
\epsilon^{d-\kappa}(1+|x|)^{-l'}
+\epsilon^d(1+|x|)^{-l'}
\right]\|W(x)\|.
$$
Combining this with \eqref{H dense eq W decay}, we conclude that
\begin{align*}
\|N_2^{(\epsilon)}\|_{L^{p(\cdot)}}
&\lesssim
\epsilon^{d-\kappa}
\left\|(1+|\cdot|)^{-l'}|W(\cdot)|\right\|_{L^{p(\cdot)}}
+\epsilon^d
\left\|(1+|\cdot|)^{-l'}|W(\cdot)|\right\|_{L^{p(\cdot)}}
\lesssim \epsilon^{d-\kappa}+\epsilon^d
\to 0
\end{align*}
as $\epsilon\to 0^+$. This, together with \eqref{H dense eq 14}, \eqref{H dense eq N0},
and \eqref{H dense eq N1},
further implies that
\begin{align*}
\left\|M_W\left(\vec{g} - \vec{g}^{(\epsilon)}, \phi\right)\right\|_{L^{p(\cdot)}}
\lesssim \left\|N_0^{(\epsilon)}\right\|_{L^{p(\cdot)}}
+ \sum_{0< |\alpha| <d} \left\|N_{1,\alpha}^{(\epsilon)}\right\|_{L^{p(\cdot)}}
+ \left\|N_2^{(\epsilon)}\right\|_{L^{p(\cdot)}} \to 0
\end{align*}
as $\epsilon \to 0^+$ and hence \eqref{H dense eq 8}.
This finishes the proof of Proposition \ref{H dense}.
\end{proof}

For any $s\in {\mathbb{Z}}_+$, let
$${\mathcal S}^{(s)}:= \left\{\vec{f}\in ({\mathcal S})^m :\ \int_{{\mathbb{R}^n}} x^\alpha \vec{f}(x)\,dx = 0,\ \alpha \in {\mathbb{Z}}_+^n\ \text{with}\ |\alpha| \leq s\right\} $$
and
$$\mathcal{O}_s := \left\{\vec{f}\in \left(C_{\rm c}^\infty\right)^m :\ \int_{{\mathbb{R}^n}} x^\alpha \vec{f}(x)\,dx = 0,\ \alpha \in {\mathbb{Z}}_+^n\ \text{with}\ |\alpha| \leq s \right\}. $$
\begin{proposition}\label{H dense O}
Let $p(\cdot) \in {\mathcal P}_0\cap LH$, $W\in {\mathscr A}_{p(\cdot),\infty}$, and $s\in{\mathbb{Z}}_+$.
Then $\mathcal{O}_s\cap H^{p(\cdot)}_W$ is dense in $H^{p(\cdot)}_W$.
\end{proposition}
\begin{proof}
Let $N\in(\frac{n}{\alpha_W},\infty)\cap\mathbb N$.
Note that, for any $s_1,s_2\in{\mathbb{Z}}_+$ with $s_1\geq s_2$,
$ \mathcal{O}_{s_1} \subset \mathcal{O}_{s_2} $.
Thus, without loss of generality, we may only consider the case $s\geq N+1$.
Moreover, by Proposition \ref{H dense} and the fact that
$(\widehat{\mathcal{D}}_0)^m\subset {\mathcal S}^{(s)}$,
we find that ${\mathcal S}^{(s)}\cap H^{p(\cdot)}_W$ is dense in $H^{p(\cdot)}_W$
and hence, to show Proposition \ref{H dense O},
we only need to prove that, for any $\vec{f}\in {\mathcal S}^{(s)}\cap H^{p(\cdot)}_W$,
there exists a sequence of functions $\{\vec{f}_k\}_{k\in{\mathbb{N}}}\subset \mathcal{O}_s$ such that
$\vec{f}_k \to \vec{f}$ in $H^{p(\cdot)}_W$ as $k\to\infty$.

Fix $\vec{f}\in {\mathcal S}^{(s)}\cap H^{p(\cdot)}_W$.
For any $k\in{\mathbb{N}}$, let $Q_k := Q(\mathbf{0},2k)$ and
$\eta_{k}:=\eta_{Q_k}$ satisfy
${\mathop\mathrm{\,supp\,}} \eta_k \subset Q_k^\ast$
and both \eqref{1220} and \eqref{1221} with $L:=Q_k$.
Using Lemmas \ref{stop e} and \ref{stop P},
we obtain $P_k:=P_{Q_k}$ as in Lemma \ref{stop P}
with $L := Q_k$ for any $k\in{\mathbb{N}}$.
Assume that $\vec{f}_k := [\vec{f} - P_k(\vec{f})]\eta_k $ for any $k\in{\mathbb{N}}$.
Then, by Lemma \ref{stop e}, we find that, for any polynomial $q\in{\mathscr P}_s$,
\begin{align*}
\left\langle \vec{f}_k, q \right\rangle = \left\langle \left[ P_k\left(\vec{f}\right) - \vec{f} \right]\eta_k , q \right\rangle = \vec{0}
\end{align*}
and hence $\vec{f}_k \in \mathcal{O}_s$ for any $k\in{\mathbb{N}}$.

Now, for any $k\in{\mathbb{N}}$, let $\vec{h}_k := (1 - \eta_k) \vec{f}$
and hence $\vec{f} = \vec{f}_k + \vec{h}_k + P_k(\vec{f})\eta_k $.
To show the desired density, we only need to prove
\begin{align}\label{4.31x}
\left\|\vec h_k+P_k\left(\vec f\right)\eta_k\right\|_{H^{p(\cdot)}_W}\to0
\quad\text{as } k\to\infty.
\end{align}
To this end, let $\psi\in {\mathcal S}$ support in $B(\mathbf{0},1)$ and $\int_{{\mathbb{R}^n}} \psi(x)\,dx = 1$.
Let $x$ be any given point in ${\mathbb{R}^n}$. Fix $k > 2|x|$.
Using the assumption that ${\mathop\mathrm{\,supp\,}} \psi \subset B(\mathbf{0},1)$
and the fact that $\vec{h}_k(y) = 0$ for any $y\in Q_k$,
we conclude that, for any $t\in (0,k - |x|)$,
$\psi_t \ast \vec{h}_k(x) = \vec{0} $.
Then we consider the case where $t\in (k - |x|,\infty)$.
Observe that, by the assumption that $\vec{f} \in ({\mathcal S})^m$ and $\psi \in {\mathcal S}$,
for any $M\in[n+1,\infty)\cap{\mathbb{N}}$ and $y\in {\mathbb{R}^n}$,
$ |\vec f(y)|\lesssim \frac{1}{(1 + |y|)^M} $ and $\|\psi\|_{L^\infty} \lesssim 1$.
Applying these yields, for any $t\in [k - |x|,\infty)$,
\begin{align*}
\left| \vec{h}_k \ast \psi_t(x) \right|
& \leq \int_{{\mathbb{R}^n}} \left| \vec f(y) \right| \left|\psi_t(x-y)\right|\,dy
\lesssim \int_{{\mathbb{R}^n}} \frac{1}{(1 + |y|)^M} t^{-n} \,dy\\
&\leq \left(\frac{1}{k - |x|}\right)^n \int_{{\mathbb{R}^n}} \frac{1}{(1 + |y|)^M}\,dy \lesssim \left(\frac{1}{k - |x|}\right)^n,
\end{align*}
which further implies that
\begin{align}\label{H dense O eq 8}
M_W\left( \vec{h}_k,\psi \right)(x) \to 0\ \ \text{as}\ \ k\to\infty.
\end{align}
Moreover, using Lemmas \ref{stop eta} and \ref{stop e},
we find that, for any $k\in{\mathbb{N}}$ and $t\in (0,\infty)$,
\begin{align}\label{H dense O eq 4}
\left| \psi_t \ast \left( e_i^{(Q_k)} \eta_k \right)(x) \right|
\leq \int_{{\mathbb{R}^n}} \left|\psi_t(y)\right| \left| e_i^{(Q_k)} (x-y) \right| \left| \eta_k (x-y)\right|\,dy
\lesssim \|\psi\|_{L^1},
\end{align}
which, combined with the definition of $P_k$ and $\psi \in \mathcal{S}$, further implies that, for any $k\in{\mathbb{N}}$ and $t\in(0,\infty)$,
\begin{align}\label{H dense O eq 1}
\left|\psi_t\ast \left[\eta_k P_k\left(\vec{f}\right) \right](x)\right|
&\leq \sum_{i = 1}^M \left\langle \vec{f}, e_i^{(Q_k)} \widetilde{\eta}_k \right\rangle \left[\psi_t \ast \left( e_i^{(Q_k)} \eta_k \right)(x)\right]
\lesssim \sum_{i = 1}^M \left|\left\langle\vec{f}, e_i^{(Q_k)} \widetilde{\eta}_k \right\rangle\right|.
\end{align}
Note that, by the fact that $\vec{f} \in {\mathcal S}^{(s)}\subset (L^1)^m$ and Lemmas \ref{stop eta} and \ref{stop e},
we have
$|\langle\vec{f}, e_i^{(Q_k)} \widetilde{\eta}_k \rangle| \lesssim |Q_k|^{-1} \|\vec{f}\|_{L^1} \to 0$
as $k\to \infty. $
Applying this and \eqref{H dense O eq 1}, we conclude that
$M_W (P_k (\vec{f})\eta_k,\psi)(x) \to 0$ as $k\to \infty, $
which, together with \eqref{H dense O eq 8},
further implies that, for any given $x\in {\mathbb{R}^n}$,
\begin{align}\label{H dense O eq 10}
M_W \left(\vec{h}_k + P_k \left(\vec{f}\right)\eta_k,\psi\right)(x) \to 0 \ \ \text{as}\ \ k\to \infty.
\end{align}

Since $\vec{h}_k + P_k (\vec{f})\eta_k = \vec{f} - [\vec{f} - P_k(\vec{f})]\eta_k$,
it follows that, for any $x\in{\mathbb{R}^n}$,
\begin{align}\label{H dense O eq 12}
M_W \left(\vec{h}_k + P_k \left(\vec{f}\right)\eta_k,\psi\right)(x)
&\leq M_W \left(\vec{f},\psi\right)(x) + M_W \left(\left[\vec{f} - P_k(\vec{f})\right]\eta_k,\psi\right)(x) {\nonumber} \\
&\lesssim \left(M_N\right)_W \left( \vec{f} \right)(x) + M_W \left(\left[\vec{f} - P_k(\vec{f})\right]\eta_k,\psi\right)(x).
\end{align}
To show \eqref{4.31x}, we need to establish an estimate 
of $ M_W([\vec{f} - P_k(\vec{f})]\eta_k,\psi) $.
To this end, we first consider the case $ x\in Q_k^\ast $.
From the definition of $P_k$,
we deduce that, for any $t\in (0,\infty)$ and $x\in Q_k^\ast$,
\begin{align}\label{H dense O eq 5}
\psi_t\ast \left[\eta_k P_k\left(\vec{f}\right) \right](x)
&= \sum_{i = 1}^M \left\langle \vec{f}, e_i^{(Q_k)} \widetilde{\eta}_k \right\rangle \left[\psi_t \ast \left( e_i^{(Q_k)} \eta_k \right)(x)\right] {\nonumber} \\
&\in \sum_{i = 1}^M  \left|\psi_t \ast \left( e_i^{(Q_k)} \eta_k \right)(x)\right| {\mathcal K}\left(\left\langle \vec{f}, e_i^{(Q_k)} \widetilde{\eta}_k \right\rangle\right).
\end{align}
Moreover, by Lemma \ref{stop P} with $y := x$ and $L := Q_k$, we find that,
for any $x\in Q^*_k$,
\begin{align*}
\left\langle \vec{f}, e_i^{(Q_k)} \widetilde{\eta}_k \right\rangle
\in C \left( 2 + \frac{|x|}{2k} \right)^{N+n+1} M_N^{{\mathcal K}} \left(\vec{f}\right)(x)
\subset C M_N^{{\mathcal K}} \left(\vec{f}\right)(x),
\end{align*}
where $C$ is a positive constant independent of $x$.
This, together with  \eqref{H dense O eq 5} and \eqref{H dense O eq 4},
further implies that, for any $x\in Q^\ast_k$,
\begin{align*}
M_W\left( P_k\left( \vec{f} \right)\eta_k ,\psi \right)(x) \lesssim (M_N)_W( \vec{f})(x).
\end{align*}
Combining this and Lemma \ref{stop f}, we conclude that,
for any $x\in Q_k^\ast$,
\begin{align}\label{H dense O eq 11}
M_W \left(\left[\vec{f} - P_k(\vec{f})\right]\eta_k,\psi\right)(x)
\leq M_W \left(\vec{f}\eta_k,\psi\right)(x) + M_W \left(P_k\left(\vec{f}\right)\eta_k,\psi\right)(x)
\lesssim (M_N)_W( \vec{f})(x),
\end{align}
which is a desired estimate.
Next, we consider the case $x\in (Q_k^\ast)^\complement$.
Using Lemma \ref{stop f} with $y := x$ and the assumption $s \geq N+1$,
we find that
\begin{align*}
M_W \left(\left[\vec{f} - P_k(\vec{f})\right]\eta_k,\psi\right)(x)
&\lesssim \left( \frac{k}{k + |x|} \right)^{n+s+1} \left( 2 + \frac{|x|}{k} \right)^{N+n+1}
 \left(M_N\right)_W \left( \vec{f}\right)(x)
\lesssim \left(M_N\right)_W \left( \vec{f}\right)(x).
\end{align*}
By this, \eqref{H dense O eq 12}, and \eqref{H dense O eq 11},
we obtain, for any $k\in{\mathbb{N}}$,
\begin{align*}
M_W \left(\vec{h}_k + P_k \left(\vec{f}\right)\eta_k,\psi\right)(x) \lesssim \left(M_N\right)_W \left( \vec{f}\right)(x),
\end{align*}
which is also a desired estimate.
Applying this, Theorem \ref{H equal}, \eqref{H dense O eq 10},
and the Lebesgue dominated convergence theorem in the setting of variable Lebesgue spaces,
we conclude that
$$\left\|\vec{h}_k + P_k(\vec{f})\eta_k\right\|_{H^{p(\cdot)}_W}
 \sim \left\| M_W \left(\vec{h}_k + P_k \left(\vec{f}\right)\eta_k,\psi\right) \right\|_{L^{p(\cdot)}}
\to 0\ \
\text{as}\ \ k\to \infty, $$
i.e., \eqref{4.31x} holds.
This finishes the proof of Proposition \ref{H dense O}.
\end{proof}

For any $K\in {\mathcal K}_{\mathrm{cs}}$ and $\vec{v} \in {\mathbb{C}}^m$, let
\begin{align*}
\rho_{K}(\vec{v}) := \sup_{\vec{w} \in K} \left|\left\langle \vec{v}, \vec{w}\right\rangle\right|.
\end{align*}
Then, by \cite[pp.\,12--13]{bc22}, we find that $\rho_K$ is a semi-norm if and only if $K\in {\mathcal K}_{\mathrm{bcs}}$
and, moreover, $\rho_K$ is a norm if and only if $K\in {\mathcal K}_{\mathrm{abcs}}$.

Let $u\in(0,\infty)$. For any $F\in L^u_{\rm{loc}}(\mathcal K)$, we now introduce its corresponding $u$-th convex body reducing operator,
whose existence is guaranteed by the following lemma, which plays an essential role
in the remainder of this article.
\begin{lemma}\label{M u exist}
Let $u\in(0,\infty)$ and $F \in L^{u}_{\rm loc}({\mathcal K})$ with
$F(x) \in {\mathcal K}_{\mathrm{abcs}}$ for almost every $x\in{\mathbb{R}^n}$.
Then, for any cube $Q\subset {\mathbb{R}^n}$, there exists a positive definite and self-adjoint matrix
$M^{(u)}_{Q}$ such that, for any vector $\vec{z}\in {\mathbb{C}}^m$,
\begin{align*}
\left|M^{(u)}_{Q}\vec{z}\right| \sim \left[\fint_{Q} \rho_{F(x)}(\vec{z})^u \,dx\right]^\frac1u,
\end{align*}
where the positive equivalence constants are independent of $F$, $Q$, and $\vec{z}$.
Moreover, for any positive-definite and self-adjoint matrix $M\in M_m$,
\begin{align*}
\left\|M^{(u)}_{Q}M\right\| \sim \left[\fint_{Q} \left| MF(x) \right|^u \,dx\right]^\frac1u,
\end{align*}
where the positive equivalence constants are independent of $F$, $Q$, and $M$.
The above matrix $M^{(u)}_{Q}$ is called the \emph{$u$-th convex body reducing operator} of $F$ over $Q$.
\end{lemma}
\begin{proof}
Since $F(x) \in {\mathcal K}_{\mathrm{abcs}}$ for almost every $x\in {\mathbb{R}^n}$,
we infer that $\rho_{F(x)}$ is a norm of ${\mathbb{C}}^m$ for almost every $x\in{\mathbb{R}^n}$.
Using this and \cite[Proposition 2.2]{bchyy26},
we obtain, for any cube $Q$ in ${\mathbb{R}^n}$,
there exists a positive-definite and self-adjoint matrix
$M^{(u)}_{Q}$ such that, for any vector $\vec{z}\in {\mathbb{C}}^m$,
\begin{align}\label{M u exist eq 1}
\left|M^{(u)}_{Q}\vec{z}\right| \sim \left[\fint_{Q} \rho_{F(x)}(\vec{z})^u \,dx\right]^\frac1u.
\end{align}

Observe that, for any positive-definite and self-adjoint matrix $M\in M_m$, we have
$$\left\|M^{(u)}_{Q} M\right\| \sim \sum_{i = 1}^m \left\|M^{(u)}_{Q} M \vec{e}_i\right\|, $$
where $\{\vec{e}_i\}_{i = 1}^m$ is an orthonormal basis of ${\mathbb{C}}^m$.
Applying this and \eqref{M u exist eq 1},
we find that
\begin{align*}
\left\|M^{(u)}_{Q} M\right\| &\sim \sum_{i = 1}^m \left\|M^{(u)}_{Q} M \vec{e}_i\right\|
 \sim \sum_{i = 1}^m \left[\fint_{Q} \rho_{F(x)}(M \vec{e}_i) ^u \,dx\right]^\frac1u
\sim  \left\{\fint_{Q} \left[ \sum_{i = 1}^m\rho_{F(x)}(M \vec{e}_i) \right]^u \,dx\right\}^\frac1u \\
& = \left\{\fint_{Q} \left[ \sum_{i = 1}^m\sup_{\vec{v} \in F(x)}\left|\left\langle M \vec{e}_i, \vec{v}\right\rangle\right| \right]^u \,dx\right\}^\frac1u
 = \left\{\fint_{Q} \left[ \sum_{i = 1}^m\sup_{\vec{v} \in F(x)}\left|\left\langle\vec{e}_i, M\vec{v}\right\rangle\right| \right]^u \,dx\right\}^\frac1u\\
&= \left\{\fint_{Q} \left[ \sum_{i = 1}^m\sup_{\vec{v} \in MF(x)}\left|\left\langle\vec{e}_i, \vec{v}\right\rangle\right| \right]^u \,dx\right\}^\frac1u
\sim \left[\fint_{Q} \left| MF(x) \right|^u \,dx\right]^\frac1u.
\end{align*}
This finishes the proof of Lemma \ref{M u exist}.
\end{proof}

For any $\alpha \in(0,\infty)$, we define the \emph{$\alpha$-convexification maximal operator ${\mathcal M}^{(\alpha)}$} by setting,
for any $f\in L^\alpha_{\rm loc}$ and $x\in{\mathbb{R}^n}$,
\begin{align*}
{\mathcal M}^{(\alpha)} \left( f \right)(x) := \sup_{x\in Q} \left[\fint_Q \left|f(y)\right|^\alpha\,dy\right]^\frac1\alpha,
\end{align*}
where the supremum is taken over all cubes $Q$ containing $x$.
Let $p(\cdot)\in {\mathcal P}_0$ and $W$ be a matrix weight.
Then the \emph{$\alpha$-convexification reducing Christ--Goldberg
convex body maximal operator $\widetilde{{\mathcal M}}^{(\alpha)}_W$} is defined by setting,
for any $F\in L^{\alpha}_{\rm loc}({\mathcal K})$ and $x\in {\mathbb{R}^n}$,
$$\widetilde{{\mathcal M}}^{(\alpha)}_W(F) (x)
:= \sup_{x\in Q} \left[\fint_Q |A_Q W^{-1}(y)F(y)|^\alpha\,dy\right]^{\frac{1}{\alpha}}, $$
where the supremum is taken over all cubes $Q$ in ${\mathbb{R}^n}$ containing $x$
and $A_Q$ is the reducing operator of order $p(\cdot)$ for $W$.
We have the following boundedness of $\widetilde{{\mathcal M}}^{(\alpha)}_W$
form $L^{p(\cdot)}({\mathcal K})$ to $L^{p(\cdot)}$.
\begin{lemma}\label{bound max wid}
Let $p(\cdot) \in \mathcal{P}_0\cap LH$.
Then, for any $W \in \mathscr{A}_{p(\cdot),\infty}$,
there exists $u \in (0,1]$ such that,
for any $F\in L^{p(\cdot)}({\mathcal K})$,
\begin{align*}
\left\| \widetilde{{\mathcal M}}_{W}^{(u)} \left( F \right) \right\|_{L^{p(\cdot)}}
\lesssim \left\|F\right\|_{L^{p(\cdot)}(\mathcal K)},
\end{align*}
where the implicit positive constant is independent of $F$.
\end{lemma}
\begin{proof}
By Lemma \ref{Apinfty u}, we conclude that there exists $\alpha\in (0,\infty)$ such that,
for any cube $Q$ in ${\mathbb{R}^n}$,
$$\fint_Q \left\|W^{-1}(x)A_Q\right\|^\alpha\,dx \lesssim 1.$$
Now, let $u\in (0,\min\{\frac \alpha2, \frac{p_-}{2}\})$.
Then, using Holder's inequality, we find that, for any $Q$ in ${\mathbb{R}^n}$,
\begin{align*}
\left[\fint_Q \left|A_Q W^{-1}(y)F(y)\right|^u\,dy\right]^{\frac{1}{u}}
\leq \left[\fint_Q \left\|A_Q W^{-1}(y)\right\|^{2u}\,dy\right]^{\frac{1}{2u}}
\left[\fint_Q \left|F\right|^{2u}\,dy\right]^{\frac{1}{2u}}
\lesssim \left[\fint_Q \left|F\right|^{2u}\,dy\right]^{\frac{1}{2u}},
\end{align*}
which further implies that, for any $x\in{\mathbb{R}^n}$,
\begin{align*}
\widetilde{{\mathcal M}}_{W}^{(u)} \left( F \right)(x)
\lesssim {\mathcal M}^{(2u)} \left( \left|F\right| \right)(x).
\end{align*}
Applying this, Lemma \ref{con f}, and the boundedness of the
Hardy--Littlewood maximal operator $\mathcal M$ on $L^{p(\cdot)}$
(see, for instance, \cite[Theorem 3.4]{cf13}),
we conclude that
\begin{align*}
\left\| \widetilde{{\mathcal M}}_{W}^{(u)} \left( F \right) \right\|_{L^{p(\cdot)}}
\lesssim \left\| {\mathcal M}^{(2u)} \left(\left|F\right| \right) \right\|_{L^{p(\cdot)}}
 = \left\| {\mathcal M}\left( |F|^{2u} \right) \right\|_{L^{\frac{p(\cdot)}{2u}}}^\frac{1}{2u}
\lesssim \left\| \left|F\right|^{2u} \right\|^{\frac{1}{2u}}_{L^{\frac{p(\cdot)}{2u}}}
= \left\|F\right\|_{L^{p(\cdot)}({\mathcal K})},
\end{align*}
which completes the proof of Lemma \ref{bound max wid}.
\end{proof}

Now, we give the proof of Theorem \ref{atom con}.
\begin{proof}[Proof of Theorem \ref{atom con}]
We first prove {\rm (i)}.
Let $\psi\in {\mathcal S}$ with ${\mathop\mathrm{\,supp\,}} \psi\subset B(\mathbf{0},1)$ and $\int_{{\mathbb{R}^n}} \psi(x)\,dx \neq 0$.
Assume that $\vec{a}_Q$ is a $(p(\cdot),q,s)_W$-atom supported in a cube $Q$.
Then, for any $x\in {\mathbb{R}^n}$, we have
\begin{align}\label{a atom eq 4}
M_W\left( \vec{a}_Q, \psi \right)(x)
&= \sup_{t\in (0,\infty)} \left|W(x) \psi_t\ast \vec{a}_Q(x) \right|
\leq \left\| W(x) A_Q^{-1} \right\|
\sup_{t\in (0,\infty)} \left|  A_Q \psi_t\ast \vec{a}_Q(x) \right|{\nonumber}\\
&\lesssim \left\| W(x) A_Q^{-1} \right\|  {\mathcal M}\left( A_Q \vec{a}_Q\right)(x).
\end{align}

Let $x\in {\mathbb{R}^n}\setminus 2\sqrt{n} Q$.
For any $t\in (0,\infty)$ and $x\in{\mathbb{R}^n}$,
let $q_{(t,x)}$ be the $s$-th degree Taylor polynomial of $y \mapsto \psi_t(x -y)$, centered at $y = c_Q$.
From this, we deduce that, for any $x\in {\mathbb{R}^n}\setminus 2\sqrt{n} Q$ and $y\in Q$,
\begin{align*}
\left| \psi_t (x - y) - q_{(t,x)}(y) \right| \lesssim \frac{[l(Q)]^{s+1}}{t^{n+s+1}}.
\end{align*}
Combining this, both {\rm (iii)} and {\rm (iv)} of Definition \ref{def atom}, and
H\"older's inequality,
we obtain, for any $x\in {\mathbb{R}^n}\setminus 2\sqrt{n} Q$,
\begin{align}\label{a atom eq 3}
\left| W(x) \psi_t\ast \vec{a}_Q(x) \right|&=
\left| W(x)\int_{{\mathbb{R}^n}} \left[ \psi_t(x-y) - q_{(t,x)}(y) \right] \vec{a}_Q(y)\,dy\right|\nonumber\\
&\lesssim \frac{[l(Q)]^{n+s+1}}{t^{n+s+1}} \left\| W(x) A_Q^{-1}\right\|
\fint_{Q} \left|A_Q  \vec{a}_Q(y)\right|\,dy{\nonumber}\\
&\leq \frac{[l(Q)]^{n+s+1}}{t^{n+s+1}} \left\| W(x) A_Q^{-1}\right\|
\left[\fint_{Q} \left| A_Q  \vec{a}_Q(y)\right|^q\,dy\right]^\frac1q{\nonumber}\\
&\lesssim \frac{[l(Q)]^{n+s+1}}{t^{n+s+1}}  \left\| W(x) A_Q^{-1}\right\| \frac{1}{\| {\mathbf{1}}_Q \|_{L^{p(\cdot)}}} .
\end{align}
By the fact that ${\mathop\mathrm{\,supp\,}} \psi_t \subset B(\mathbf{0},t)$,
we find that, for any $x\in {\mathbb{R}^n}\setminus 2\sqrt{n}Q$ and $t\in (0,\frac12 |x - c_Q|)$,
$ |\psi_t \ast \vec{a}_Q(x)| = 0 $.
Thus, using this and \eqref{a atom eq 3},
we conclude that, for any $x\in {\mathbb{R}^n}\setminus 2\sqrt{n}Q$,
\begin{align}\label{a atom eq 5}
M_W\left(  \vec{a}_Q,\psi \right) (x)
&= \sup_{t\in [\frac12 |x - c_Q|,\infty)} \left|W(x) \psi_t\ast \left(\vec{a}_Q\right)(x)\right|{\nonumber}\\
&\leq \frac{1}{\| {\mathbf{1}}_Q \|_{L^{p(\cdot)}}} \left\|W(x) A_Q^{-1}\right\| \sup_{t\in [\frac12 |x - c_Q|,\infty)} \frac{[l(Q)]^{n+s+1}}{t^{n+s+1}}{\nonumber}\\
&\lesssim \frac{1}{\| {\mathbf{1}}_Q \|_{L^{p(\cdot)}}} \left\|W(x) A_Q^{-1}\right\| \left[\frac{l(Q)}{l(Q) + |x - c_Q| }\right]^{n+s+1}.
\end{align}
Now, for any $k_1,k_2\in {\mathbb{Z}}$ with $k_1 \leq k_2$,
by \eqref{a atom eq 4} and \eqref{a atom eq 5}, we have
\begin{align}\label{a atom eq 11}
\left\| \sum_{k = k_1}^{k_2} |\lambda_k| M_W \left( \vec{a}_k,\psi \right) \right\|_{L^{p(\cdot)}}
&\lesssim \left\| \sum_{k = k_1}^{k_2} | \lambda_k | \left\| W(\cdot) A_{Q_k}^{-1} \right\|  {\mathcal M}\left( A_{Q_k} \vec{a}_k\right){\mathbf{1}}_{2\sqrt{n} Q_k} \right\|_{L^{p(\cdot)}}{\nonumber} \\
&\quad + \left\| \sum_{k = k_1}^{k_2}  \frac{| \lambda_k |}{\| {\mathbf{1}}_{Q_k} \|_{L^{p(\cdot)}}} \left\|W(\cdot) A_{Q_k}^{-1}\right\| \left[\frac{l(Q_k)}{l(Q_k) + |\cdot - c_{Q_k}| }\right]^{n+s+1} \right\|_{L^{p(\cdot)}}.
\end{align}
From Lemma \ref{f dec eq le} with $\lambda_k := \frac{|\lambda_k|}{\|{\mathbf{1}}_{Q_k}\|_{L^{p(\cdot)}}}$,
we infer that
\begin{align}\label{a atom eq 7}
&\left\| \sum_{k = k_1}^{k_2}  \frac{| \lambda_k |}{\| {\mathbf{1}}_{Q_k} \|_{L^{p(\cdot)}}} \left\|W(\cdot) A_{Q_k}^{-1}\right\| \left[\frac{l(Q_k)}{l(Q_k) + |\cdot - c_{Q_k}| }\right]^{n+s+1} \right\|_{L^{p(\cdot)}}{\nonumber} \\
&\quad \lesssim \left\| \sum_{k= k_1}^{k_2} \left[\frac{| \lambda_k |}{\| {\mathbf{1}}_{Q_k} \|_{L^{p(\cdot)}}}\right]^r {\mathbf{1}}_{Q_k} \right\|_{L^{\frac{p(\cdot)}{r}}}^{\frac1r}
 = \left\| \left\{\sum_{k= k_1}^{k_2} \left[\frac{| \lambda_k |}{\| {\mathbf{1}}_{Q_k} \|_{L^{p(\cdot)}}}\right]^r {\mathbf{1}}_{Q_k}\right\}^{\frac1r} \right\|_{L^{p(\cdot)}}.
\end{align}
Moreover, by Lemma \ref{a atom eq le} with $\vec{a}_k := A_{Q_k}\vec{a}_k$ and $T := {\mathcal M}$
and by the definition of $(p(\cdot),q,s)_W$-atoms, we find that
\begin{align*}
&\left\| \sum_{k = k_1}^{k_2} | \lambda_k | \left\| W(\cdot) A_{Q_k}^{-1} \right\|  {\mathcal M}\left( A_{Q_k} \vec{a}_k\right){\mathbf{1}}_{2\sqrt{n} Q_k} \right\|_{L^{p(\cdot)}} \\
&\quad \lesssim \left\|\left[\sum_{k = k_1}^{k_2} \left(|Q_k|^{-\frac1q}| \lambda_k |\left\|A_{Q_k} \vec{a}_k\right\|_{L^q} \right)^r {\mathbf{1}}_{Q_k} \right]^\frac1r\right\|_{L^{p(\cdot)}}
 \lesssim \left\|\left\{\sum_{k = k_1}^{k_2} \left[\frac{| \lambda_k |}{\|{\mathbf{1}}_{Q_k}\|_{L^{p(\cdot)}}} \right]^r {\mathbf{1}}_{Q_k} \right\}^\frac1r\right\|_{L^{p(\cdot)}}.
\end{align*}
This, together with \eqref{a atom eq 7} and \eqref{a atom eq 11},  further implies that
\begin{align}\label{a atom eq 12}
\left\| \sum_{k = k_1}^{k_2} |\lambda_k| M_W \left( \vec{a}_k,\psi \right) \right\|_{L^{p(\cdot)}}
\lesssim \left\| \left\{ \sum_{k = k_1}^{k_2} \left[\frac{| \lambda_k |}{\|{\mathbf{1}}_{Q_k}\|_{L^{p(\cdot)}}}\right]^r {\mathbf{1}}_{Q_k} \right\}^{\frac1r} \right\|_{L^{p(\cdot)}},
\end{align}
and hence $\{\sum_{k = k_1}^{k_2}\lambda_k \vec{a}_k \}_{k_1,k_2\in {\mathbb{Z}}}$ is a Cauchy sequences in $H^{p(\cdot)}_W$.
Using this and Propositions \ref{H embed} and \ref{H comp},
we conclude that $\vec f:=\sum_{k \in{\mathbb{Z}} } \lambda_k \vec{a}_k$ converges in $({\mathcal S}')^m$.
From this, Theorem \ref{H equal}, and \eqref{a atom eq 12}, we deduce that
$$ \left\|\vec{f}\right\|_{H^{p(\cdot)}_W} \lesssim \left\| \left\{\sum_{k\in{\mathbb{Z}}} \left[\frac{|\lambda_k| }{\|{\mathbf{1}}_{Q_k}\|_{L^{p(\cdot)}}}\right]^{r} {\mathbf{1}}_{Q_k} \right\}^{\frac1r} \right\|_{L^{p(\cdot)}}. $$
This finishes the proof of {\rm (i)}.

Next, we prove {\rm (ii)}. By Proposition \ref{H dense O}, it suffices to assume that
$\vec f\in \mathcal O_s\cap H_W^{p(\cdot)}$.
Since $\vec f$ is compactly supported, from Lemma \ref{Q Qt},
it follows that there exist a dyadic lattice
$\mathscr Q^t$ and a cube $Q_0\in \mathscr Q^t$ such that
${\mathop\mathrm{\,supp\,}} \vec f\subset Q_0$.
In what follows, we simply write $\mathscr Q:=\mathscr Q^t$.

By Lemmas \ref{Apinfty u} and \ref{bound max wid},
there exists $u\in (0,\frac{p_-}{2})$ such that
$\sup_{Q} \fint_Q \|W^{-1}(x)A_Q\|^{2u}\,dx < \infty$
and $ \widetilde{\mathcal{M}}^{(u)}_W$ is bounded from $L^{p(\cdot)}(\mathcal{K})$ to $L^{p(\cdot)}$.
Using these, Lemma \ref{norm matrix},
H\"older's inequality, and Lemma \ref{Holder},
we find that, for any cube $Q$ in $\mathbb{R}^n$,
\begin{align*}
\left[\int_{Q} \left| M_{N}^{\mathcal{K}} \left( \vec{f} \right)(x) \right|^{u}\,dx \right]^{\frac1u}
& \leq \left\| A_Q^{-1} \right\|
\left[\int_{Q}\left\| A_Q W^{-1}(x) \right\|^u \left| W(x) M_{N}^{\mathcal{K}} \left( \vec{f} \right)(x) \right|^{u}\,dx \right]^{\frac1u} \\
& \lesssim \left[\int_{Q}\left\| A_Q W^{-1}(x) \right\|^{2u}\,dx\right]^{\frac{1}{2u}}
\left[\int_{Q} \left| W(x) M_{N}^{\mathcal{K}} \left( \vec{f} \right)(x) \right|^{2u}\,dx \right]^{\frac{1}{2u}} \\
& \lesssim \left[\int_{Q} \left| W(x) M_{N}^{\mathcal{K}} \left( \vec{f} \right)(x) \right|^{2u}\,dx \right]^{\frac{1}{2u}}
\lesssim \left\|\,\left| W(\cdot) M_{N}^{\mathcal{K}} \left( \vec{f} \right) \right| \mathbf{1}_{Q} \right\|_{L^{p(\cdot)}} < \infty,
\end{align*}
which further implies that $ M_{N}^{\mathcal{K}}(\vec{f}) \in L_{\rm loc}^u(\mathcal{K}) $.
By Lemma \ref{f dec eq le}, we may choose $L\in {\mathbb{Z}}$ sufficiently large so that
\begin{align}\label{Hardy dep eq 32}
\left\|\frac{\|W(\cdot)\|}{(1+|\cdot|)^L}\right\|_{L^{p(\cdot)}}<\infty.
\end{align}
For any $x\in {\mathbb{R}^n}$ and $\epsilon\in (0,1)$, let
$$
\widetilde M_{N,\epsilon}^{{\mathcal K}}(\vec f)(x)
:=M_N^{{\mathcal K}}(\vec f)(x)+\frac{\epsilon}{(1+|x|)^L}\overline B
$$
be the nondegenerate grand maximal convex body,
where $\overline B$ denotes the closed unit ball of ${\mathbb{C}}^m$.
From \eqref{Hardy dep eq 32} and Theorem \ref{H equal}, we infer that
\begin{align}\label{Hardy dep eq 33}
\left\| \,\left|W(\cdot) \widetilde{M}^{\mathcal K}_{N,\epsilon} \left(\vec{f}\right)\right|\, \right\|_{L^{p(\cdot)}}
\lesssim \left\| \vec{f} \right\|_{H^{p(\cdot)}_W} + \left\|\frac{\|W(\cdot)\|}{(1+|\cdot|)^L}\right\|_{L^{p(\cdot)}}
< \infty.
\end{align}
Moreover, $\widetilde M_{N,\epsilon}^{{\mathcal K}}(\vec f)(x)\in {\mathcal K}_{\mathrm{abcs}}$ for every $x\in {\mathbb{R}^n}$,
and hence, by Lemma \ref{M u exist}, for each cube $Q\subset {\mathbb{R}^n}$
there exists a $u$-th convex body reducing operator $M_Q^{(u)}$ of $\widetilde M_{N,\epsilon}^{{\mathcal K}}(\vec f)$
over $Q$ such that, for every positive definite self-adjoint matrix $M\in M_m$,
\begin{align}\label{Hardy dep eq 19}
\left\|M_Q^{(u)}M\right\|
\sim
\left[\fint_Q \left|M\widetilde M_{N,\epsilon}^{{\mathcal K}}(\vec f)(x)\right|^u\,dx\right]^{1/u}.
\end{align}

Next, we define the local level set with the height of the reducing operator $M_Q^{(u)}$.
For each $Q\in \mathscr Q$, let
\begin{align}\label{eq}
E_Q:=\left\{x\in (3Q)^\circ:
\left|\left[M_{3Q}^{(u)}\right]^{-1}M_N^{{\mathcal K}}(\vec f)(x)\right|>\widetilde{C}
\right\},
\end{align}
where $(3Q)^\circ$ denotes the \emph{interior} of $3Q$ and
the constant $\widetilde C\in(0,\infty)$ will be chosen sufficiently large later.
Note that, by the geometry of the dyadic lattice,
there exists a positive constant $\widetilde C_{\cap}$
such that
\begin{align}\label{Ccap}
\sharp\left\{
Q'\in \mathscr Q:
7Q\cap 7Q'\neq \emptyset
\ \text{and}\
|s_Q-s_{Q'}|<8
\right\}
\le \widetilde C_{\cap}
\end{align}
for every $Q\in \mathscr Q$,
where $s_Q := -\log_2 (l(Q))$ and $s_{Q'} := -\log_2 (l(Q'))$.
We claim that $\widetilde C\in(0,\infty)$ can be chosen sufficiently large such that,
for any $Q\in\mathscr Q$,
\begin{align}\label{422}
|E_Q|<\widetilde C_{\cap}^{-1}2^{-12n}|Q|.
\end{align}
Indeed, by \eqref{Hardy dep eq 19} and the obvious fact that
$M_N^{{\mathcal K}}(\vec f)(x)\subset \widetilde M_{N,\epsilon}^{{\mathcal K}}(\vec f)(x)$,
we have
\begin{align*}
|E_Q|
&\le
\widetilde C^{-u}\int_{3Q}
\left|
\left[M_{3Q}^{(u)}\right]^{-1}M_N^{{\mathcal K}}(\vec f)(x)
\right|^u\,dx\lesssim
\widetilde C^{-u}|3Q|.
\end{align*}
Hence, \eqref{422} follows by choosing $\widetilde C$ sufficiently large.
Observe that, by Lemma \ref{Kconvex}, we have
$$ \left|\left[M^{(u)}_{3Q}\right]^{-1} M_N^{{\mathcal K}}\left(\vec{f}\right)(x)\right|
= \sup_{\phi\in{\mathcal S}_N}\sup_{t\in (0,\infty)}\left| \left[M^{(u)}_{3Q}\right]^{-1} {\mathcal K} \left(\phi_t\ast\vec{f}\right)(x) \right|. $$
Using this and the fact that $|[M^{(u)}_{3Q}]^{-1} {\mathcal K} (\phi_t\ast\vec{f})|$ is continuous
for any $\phi\in{\mathcal S}_N$ and $t\in(0,\infty)$,
we find that $|[M^{(u)}_{3Q}]^{-1} M_N^{{\mathcal K}}(\vec{f})|$ is lower semi-continuous and hence
$E_{Q}$ is open.
Let $\mathcal F_0:=\{Q_0\}$ and $E_1:=E_{Q_0}$.
Then, by Lemma \ref{stop coll} with $E := E_1$,
we find that there exists a sequence of cubes ${\mathcal F}_1$ satisfying
{\rm (i)} through {\rm (iv)} of Lemma \ref{stop coll}.
Furthermore, letting $E_2 := \cup_{Q\in {\mathcal F}_1} E_Q$ and
using Lemma \ref{stop coll} with $E := E_2$,
we find that there exists a sequence of cubes ${\mathcal F}_2$ satisfying
{\rm (i)} through {\rm (iv)} of Lemma \ref{stop coll}.

Now, iterate this step infinity times.
More precisely, for any $k\in{\mathbb{N}}$, if we have a sequence of cubes ${\mathcal F}_{k-1}$
satisfying {\rm (i)} through {\rm (iv)} of Lemma \ref{stop coll},
then, letting $E_k:=\cup_{Q\in\mathcal F_{k-1}} E_Q$
and using Lemma \ref{stop coll} with $E := E_k$,
we conclude that there exists a sequence of cubes ${\mathcal F}_{k}$
also satisfying {\rm (i)} through {\rm (iv)} of Lemma \ref{stop coll}.

Thus, we conclude that, for any $k\in{\mathbb{N}}$, there exists a sequence of cubes ${\mathcal F}_{k}$
having the following four properties:
\begin{itemize}
\item[{\rm (i)}] For any $L, L'\in {\mathcal F}_k$, if $L\neq L'$, then $L\cap L' = \emptyset$.
\item[{\rm (ii)}] $E_k = \bigcup_{L\in {\mathcal F}_k} L = \bigcup_{L\in {\mathcal F}_k} 9L$.
\item[{\rm (iii)}] For any $L\in {\mathcal F}_k$, $32L \cap E_k^{\complement} \neq \emptyset$.
\item[{\rm (iv)}] For any $L,L'\in {\mathcal F}_k$, if $7L\cap 7L' \neq \emptyset$,
then $|s_L - s_{L'}| < 8$.
\end{itemize}
Moreover, note that, for any $k\in{\mathbb{N}}$, $Q\in {\mathcal F}_{k-1}$,
$$3Q\cap E_k \subset \bigcup_{\genfrac{}{}{0pt}{}{Q'\in 
{\mathcal F}_{k-1}}{3Q'\cap 3Q\neq \emptyset}} E_{Q'} $$
and hence, by the property (iv) of ${\mathcal F}_{k-1}$, \eqref{Ccap}, and \eqref{422},
we obtain
\begin{align}\label{Hardy dep eq 2}
\left| 3Q\cap E_k\right| \leq \sum_{
\genfrac{}{}{0pt}{}{Q'\in {\mathcal F}_{k-1}}{3Q'\cap 3Q\neq \emptyset}} \left|E_{Q'}\right|
< C_{\cap} 2^{8n} C_{\cap}^{-1} 2^{-12n} |Q| = 2^{-4n}|Q|.
\end{align}
Applying this and Lemma \ref{stop coll}, we conclude that, for any $k\in{\mathbb{N}}$,
${\mathcal F}_{k}$ also has the following fifth property:
\begin{itemize}
\item[{\rm (v)}] For any $Q\in {\mathcal F}_{k-1}$ and $L\in {\mathcal F}_k$,
if $L^\ast \cap Q^\ast\neq \emptyset$,
then $32L \subset 3Q$ and $32L\cap (3Q\setminus E_k) \neq \emptyset$.
\end{itemize}
Now, for any $k\in\mathbb Z_+$
and $Q\in\mathcal F_k$, let
$$
\mathcal G(Q) := \left\{L\in {\mathcal F}_{k+1}:\
L^*\cap Q^*\neq\emptyset\right\}.
$$

For any $k\in\mathbb N$, let $\{\eta_Q:\ {\mathbb{R}^n}\to [0,1]\}_{Q\in \mathcal F_k}$
be the same as in Lemma \ref{stop eta} with $E:=E_{k}$ and
$S := \mathcal F_k$.
Moreover, for any fixed $L\in {\mathcal F}_k$ with $k\in\mathbb N$, from Lemma \ref{stop e},
we infer that there exist $\{e_i^{(L)}\}_{i = 1}^M \subset {\mathscr P}_s$ as in Lemma \ref{stop e}.
Using this, we also define $P_L(\vec{f})$ as in Lemma \ref{stop P}
and
\begin{align*}
\vec b_L := \left[ \vec{f} - P_L\left( \vec{f} \right) \right]\eta_L.
\end{align*}
Then, by Lemma \ref{stop P},
we obtain $ \langle \vec{b}_L,\cdot^\gamma \rangle = 0 $
for any $\gamma\in{\mathbb{Z}}_+^n$ with $|\gamma|\leq s$.
For any $k\in{\mathbb{N}}$, let $\vec{b}_k := \sum_{L\in {\mathcal F}_k} \vec{b}_L$.

We claim that
\begin{align}\label{Hardy dep eq 23}
\left\|M_W\left( \vec{b}_k, \psi \right)\right\|_{L^{p(\cdot)}}
\to 0 \ \ \text{as}\ \ k\to\infty.
\end{align}
Note that, by the definition of $\vec{b}_L$, we conclude that, for any $x\in{\mathbb{R}^n}$,
\begin{align}\label{Hardy dep eq 18}
M_W\left( \vec{b}_L, \psi \right)(x)
&\leq M_W \left( \eta_L P_L\vec{f}, \psi \right)(x) {\mathbf{1}}_{L^\ast}(x)
+ M_W\left( \vec{f}\eta_L ,\psi \right)(x) {\mathbf{1}}_{L^\ast}(x)
 + M_W\left( \vec{b}_L, \psi \right)(x) {\mathbf{1}}_{(L^\ast)^\complement}(x){\nonumber} \\
& =: I(x) + II(x) + III(x).
\end{align}
We first give the estimate of $I(x)$.
From the definition of $P_L$,
we deduce that, for any $t\in (0,\infty)$ and $x\in L^\ast$,
\begin{align}\label{Hardy dep eq 11}
\psi_t\ast \left(\eta_L P_L\vec{f} \right)(x)
&= \sum_{i = 1}^M \left\langle \vec{f}, e_i^{(L)} \widetilde{\eta}_L \right\rangle \left[\psi_t \ast \left( e_i^{(L)} \eta_L \right)(x)\right] {\nonumber}\\
&\in \sum_{i = 1}^M  \left|\psi_t \ast \left( e_i^{(L)} \eta_L \right)(x)\right| {\mathcal K}\left(\left\langle \vec{f}, e_i^{(L)} \widetilde{\eta}_L \right\rangle\right).
\end{align}
By Lemmas \ref{stop eta} and \ref{stop e},
we find that, for any $t\in (0,\infty)$ and $x\in L^\ast$,
\begin{align*}
\left| \psi_t \ast \left( e_i^{(L)} \eta_L \right)(x) \right|
\leq \int_{{\mathbb{R}^n}} \left|\psi_t(y)\right| \left| e_i^{(L)}(x-y) \right| \left| \eta_L (x-y)\right|\,dy
\lesssim \|\psi\|_{L^1},
\end{align*}
which, combined with \eqref{Hardy dep eq 11} and Lemma \ref{stop P} with $y := x$, further implies that
$$\left|W(x) \psi_t\ast \left(\eta_L P_L\vec{f} \right)(x)\right|
\lesssim \left|W(x) M_N^{{\mathcal K}}\left( \vec{f}\right)(x)\right|, $$
and hence
\begin{align}\label{Hardy dep eq 20}
I(x) \lesssim \left(M_N\right)_W \left( \vec{f}\right)(x) {\mathbf{1}}_{L^\ast}(x).
\end{align}

Next, using Lemma \ref{stop f},
we find that there exists a positive constant $C$ such that, for any $x\in L^\ast$,
$M^{{\mathcal K}} ( \vec{f}\eta_L,\psi )(x) \subset C M^{{\mathcal K}}_N ( \vec{f} )(x),$
and hence
\begin{align}\label{Hardy dep eq 21}
II(x) \lesssim \left| W(x) M_N^{{\mathcal K}} \left( \vec{f} \right)(x)\right| {\mathbf{1}}_{L^\ast}(x)
 = (M_N)_W \left( \vec{f} \right)(x) {\mathbf{1}}_{L^\ast}(x).
\end{align}

Now, we estimate $III(x)$. From Lemma \ref{stop f}, we infer that
there exists a positive constant $C$ such that,
for any $x\in (L^\ast)^\complement$ and $y\in{\mathbb{R}^n}$,
\begin{align*}
M^{{\mathcal K}} \left( \vec{b}_L,\psi \right)(x) \subset
C \left( \frac{l(L)}{l(L) + |x - c_L|} \right)^{n+s+1} \left( 2 + \frac{|y - c_L|}{l(L)} \right)^{N+n+1}
 M^{{\mathcal K}}_N \left( \vec{f}\right)(y)
\end{align*}
and hence, for any $y\in 3L$,
\begin{align*}
III(x)
& \lesssim \left( \frac{l(L)}{l(L) + |x - c_L|} \right)^{n+s+1} \left\|W(x) A_L^{-1}\right\|
\left\| A_L M^{(u)}_{3L} \right\| \left| \left[ M^{(u)}_{3L} \right]^{-1} M^{{\mathcal K}}_N \left( \vec{f}\right)(y)\right|.
\end{align*}
Using this, integrating with $y$ over $3L$,
and using the obvious fact that $M^{\mathcal{K}}_{N}(\vec{f})(x) \subset \widetilde{M}^{\mathcal K}_{N,\epsilon} (\vec{f})(x)$
and \eqref{Hardy dep eq 19}, we find that,
for any $x\in(L^\ast)^\complement$,
\begin{align*}
III(x)
&\lesssim \left( \frac{l(L)}{l(L) + |x - c_L|} \right)^{n+s+1} \left\|W(x) A_L^{-1}\right\|
\left\| A_L M^{(u)}_{3L} \right\| \left\{\fint_{3L} \left| \left[ M^{(u)}_{3L} \right]^{-1} M^{{\mathcal K}}_N \left( \vec{f}\right)(y)\right|^u\,dy\right\}^\frac1u\\
&\lesssim   \left\|W(x) A_L^{-1}\right\| \left\| A_L M^{(u)}_{3L} \right\|
 \left[\frac{l(L)}{l(L) + |x - c_L|} \right]^{n+s+1},
\end{align*}
which, combined with \eqref{Hardy dep eq 18}, \eqref{Hardy dep eq 20}, and \eqref{Hardy dep eq 21},
further implies that, for any $x\in{\mathbb{R}^n}$,
\begin{align*}
M_W\left( \vec{b}_L, \psi \right)(x)
\lesssim (M_N)_W \left( \vec{f} \right)(x) {\mathbf{1}}_{L^\ast}(x)
 + \left\|W(x) A_L^{-1}\right\| \left\| A_L M^{(u)}_{3L} \right\|
 \left[ \frac{l(L)}{l(L) + |x - c_L|} \right]^{n+s+1}.
\end{align*}
By this, Lemmas \ref{f dec eq le},
$s\in {\mathbb{Z}}_+ \cap [\lfloor d^{\rm upper}_{p(\cdot),\infty}(W) + n(\frac{1}{r} - 1)\rfloor,\infty)$,
and both {\rm (ii)} and {\rm (iv)}  of ${\mathcal F}_k$,
we conclude that
\begin{align}\label{Hardy dep eq 22}
&\left\|M_W\left( \vec{b}_k, \psi \right)\right\|_{L^{p(\cdot)}} {\nonumber}\\
&\quad \lesssim \left\|\sum_{L\in {\mathcal F}_k} (M_N)_W \left( \vec{f} \right) {\mathbf{1}}_{L^\ast} \right\|_{L^{p(\cdot)}}
+ \left\|\sum_{L\in {\mathcal F}_k} \left\| A_L M^{(u)}_{3L} \right\| \left\|W(\cdot) A_{L}^{-1} \right\| \left[\frac{l(L)}{l(L) + |\cdot-c_L|}\right]^{n + s +1} \right\|_{L^{p(\cdot)}} {\nonumber}\\
&\quad \lesssim \left\|(M_N)_W \left( \vec{f} \right) {\mathbf{1}}_{E_k} \right\|_{L^{p(\cdot)}}
+ \left\|\sum_{L\in {\mathcal F}_k} \left\| A_L M^{(u)}_{3L} \right\|^r {\mathbf{1}}_{L} \right\|^{\frac1r}_{L^{\frac{p(\cdot)}{r}}},
\end{align}
where $r := \min\{1,p_-\}$.
Observe that, using \eqref{Hardy dep eq 19} and Lemma \ref{QP5},
we obtain, for any $x\in L$,
\begin{align*}
\left\| A_L M^{(u)}_{3L} \right\|
&\lesssim \left\|A_L A_{3L}^{-1}\right\| \left\| A_{3L} M^{(u)}_{3L} \right\|
\lesssim \left[\fint_{3L} \left| A_{3L} \widetilde{M}^{{\mathcal K}}_{N,\epsilon}(\vec{f})(y) \right|^u \,dy\right]^\frac1u
\lesssim \widetilde{{\mathcal M}}^{(u)}_W \left(W(\cdot) \widetilde{M}^{{\mathcal K}}_{N,\epsilon}(\vec{f}) {\mathbf{1}}_{E_k} \right)(x).
\end{align*}
Applying this, the properties {\rm (ii)} and {\rm (iv)} of ${\mathcal F}_k$,
and Lemmas \ref{con f} and \ref{bound max wid},
we conclude that
\begin{align*}
\left\|\sum_{L\in {\mathcal F}_k} \left\| A_L M^{(u)}_{3L} \right\|^r {\mathbf{1}}_{L} \right\|^{\frac1r}_{L^{\frac{p(\cdot)}{r}}}
&\lesssim \left\|\sum_{L\in {\mathcal F}_k} \left[\widetilde{{\mathcal M}}^{(u)}_W \left( W(\cdot) \widetilde{M}^{{\mathcal K}}_{N,\epsilon}(\vec{f}) {\mathbf{1}}_{E_k} \right)\right]^r {\mathbf{1}}_{L} \right\|^{\frac1r}_{L^{\frac{p(\cdot)}{r}}} \\
&\lesssim \left\|\widetilde{{\mathcal M}}^{(u)}_W \left(W(\cdot) \widetilde{M}^{{\mathcal K}}_{N,\epsilon}(\vec{f}) {\mathbf{1}}_{E_k} \right) {\mathbf{1}}_{E_k} \right\|_{L^{p(\cdot)}}
\lesssim \left\|\,\left| W(\cdot)\widetilde{M}^{{\mathcal K}}_{N,\epsilon}(\vec{f})\right| {\mathbf{1}}_{E_k}  \right\|_{L^{p(\cdot)}},
\end{align*}
which, together with \eqref{Hardy dep eq 22}, \eqref{Hardy dep eq 33},
and the Lebesgue dominated convergence theorem with $|E_k| \to 0$ as $k\to \infty$,
further implies that
\begin{align*}
\left\|M_W\left( \vec{b}_k, \psi \right)\right\|_{L^{p(\cdot)}}
\lesssim \left\|\, \left|W(\cdot) \widetilde{M}^{{\mathcal K}}_{N,\epsilon}(\vec{f})\right| {\mathbf{1}}_{E_k}  \right\|_{L^{p(\cdot)}} \to 0
\end{align*}
as $k\to\infty$.
This further shows the claim \eqref{Hardy dep eq 23}.
Combining this and Proposition \ref{H embed},
we find that $\vec{b}_k \to \vec{0}$ in $({\mathcal S}')^m$ as $k\to\infty$.

Next, let $\vec{g} := \vec{f} - \vec{b}_1$.
Then, by the definitions of both $\vec{b}_1$ and $\mathcal F_1$ and
by ${\mathop\mathrm{\,supp\,}} \vec f\subset Q_0$, we obtain
\begin{align}\label{Hardy dep eq 12}
\vec{g} = \vec{f} - \sum_{L\in {\mathcal F}_1} \left[ \vec{f} - P_L\left(\vec{f}\right) \right]\eta_L
= \vec{f} {\mathbf{1}}_{E_1^\complement} + \sum_{L\in {\mathcal F}_1} P_L\left(\vec{f}\right)\eta_L.
\end{align}
Using this and \eqref{Hardy dep eq 23},
we conclude that
\begin{align*}
\vec{f} = \vec{g} + \sum_{k = 1}^{\infty} \left( \vec{b}_k - \vec{b}_{k+1} \right)
\end{align*}
in both $H^{p(\cdot)}_W$ and $({\mathcal S}')^m$.
From the assumption that ${\mathop\mathrm{\,supp\,}} \vec{f}\subset Q_0$,
it follows that, for any $L\in {\mathcal F}_1\setminus \mathcal G(Q_0)$,
$\vec{f}\eta_L = \vec{0}$ and hence $P_L(\vec{f})\eta_L = \vec{0}$.
Applying this and \eqref{Hardy dep eq 12} yields
\begin{align}\label{Hardy dep eq 16}
\vec{g} =\vec{f} {\mathbf{1}}_{E_1^\complement} + \sum_{L\in \mathcal G(Q_0)} P_L\left(\vec{f}\right)\eta_L.
\end{align}

For any $k\in\mathbb N$, $Q\in {\mathcal F}_k$, and $L\in {\mathcal F}_{k+1}$,
let
$$ \vec{c}_{Q,L} := P_{L} \left( \left[ \vec{f} - P_L\left(\vec{f}\right) \right] \eta_{Q} \right). $$
Note that, for any $k\in{\mathbb{N}}$, $Q\in{\mathcal F}_k$, and $L\in{\mathcal F}_{k+1}$,
if $Q^\ast \cap L^\ast = \emptyset$, then $\vec{c}_{Q,L} = \vec{0}$.
By the definition of $\eta_{Q}$ and the fact that,
for any $L\in{\mathcal F}_{k+1}$, $L\subset L^*\subset E_{k}\subset E_{k-1}$, we conclude that, for any $x\in L^*$,
$\sum_{Q\in {\mathcal F}_k,Q^*\cap L^*\neq\emptyset}\eta_Q(x)=1$,
which further implies that
$$\sum_{Q\in {\mathcal F}_k,Q^*\cap L^*\neq\emptyset} \vec{c}_{Q,L}\eta_L = P_{L} \left[\vec{f} - P_L\left(\vec{f}\right)\right]
 \eta_L = \vec{0}
\ \ \text{in}\ \ \left( {\mathcal S}' \right)^m. $$
Moreover, from the construction of $P_L$, we deduce that,
for any $Q\in {\mathcal F}_k$, any $L\in {\mathcal F}_{k+1}$, and any polynomial $q\in {\mathscr P}_s$,
\begin{align}\label{Hardy dep eq 26}
\int_{{\mathbb{R}^n}} q(x)\left[ \left\{ \vec{f}(x) - P_L\left( \vec{f} \right)(x) \right\}\eta_Q(x) - \vec{c}_{Q,L}(x) \right]\eta_L(x) = 0.
\end{align}
Observe that, for any $k\in{\mathbb{N}}$,
\begin{align*}
\vec{b}_k - \vec{b}_{k+1}
&= \sum_{Q\in{\mathcal F}_k} \left[ \vec{f} - P_Q\left( \vec{f} \right) \right]\eta_Q - \sum_{L\in {\mathcal F}_{k+1}} \left[ \vec{f} - P_L\left( \vec{f} \right) \right]\eta_L\nonumber\\
&= \sum_{Q\in{\mathcal F}_k} \left\{ \vec{f} \eta_Q - P_Q\left( \vec{f} \right) \eta_Q - \sum_{L\in \mathcal G(Q)} \left[ \vec{f} - P_L\left( \vec{f} \right) \right]\eta_L \eta_Q  \right\}\nonumber\\
&= \sum_{Q\in{\mathcal F}_k} \left\{ \vec{f} \eta_Q - P_Q\left( \vec{f} \right) \eta_Q - \sum_{L\in \mathcal G(Q)}
\left[ \vec{f} - P_L\left( \vec{f} \right) \right]\eta_L \eta_Q + \sum_{L\in \mathcal G(Q)}\vec c_{Q,L} \eta_L  \right\}\nonumber\\
&=: \sum_{Q\in{\mathcal F}_k} \vec{A}_{k,Q}\nonumber
\end{align*}
and let $\vec{A}_{0,Q_0} := \vec{g}$.

Now, we claim that
\begin{align}\label{Hardy dep eq 17}
\vec{f} = \sum_{k\in{\mathbb{Z}_+}} \sum_{Q\in {\mathcal F}_k} \vec{A}_{k,Q}
\end{align}
in $L^q$ for any $q\in (0,\infty)$ and hence in $(\mathcal S')^m$.
For any $k\in{\mathbb{N}}$, we first show that $\sum_{Q\in{\mathcal F}_k} \vec{A}_{k,Q}$ converges in $L^q$.
From Lemma \ref{stop e} and the assumption $\vec{f} \in \mathcal{O}_s$,
we infer that, for any $Q\in {\mathcal F}_k$ with $k\in\mathbb Z_+$ and for any $x\in Q^\ast$,
\begin{align}\label{Hardy dep eq 15}
P_Q\left( \vec{f} \right)(x) = \sum_{i = 1}^M \left\langle \vec{f}, e^{(Q)}_i\widetilde{\eta}_Q \right\rangle e^{(Q)}_i(x)
\lesssim \left\| \vec{f} \right\|_{L^\infty} \lesssim 1,
\end{align}
which, combined with the definition of $\vec{A}_{k,Q}$, further implies that,
for any $k\in{\mathbb{N}}$, $Q\in {\mathcal F}_k$, and $x\in{\mathbb{R}^n}$,
\begin{align}\label{AkQ}
|\vec{A}_{k,Q}|\lesssim {\mathbf{1}}_{3Q},
\end{align}
where the implicit positive constant depends on $\vec f$.
Using this and and the properties {\rm (ii)} and {\rm (iv)} of ${\mathcal F}_k$, we find that, for any $x\in{\mathbb{R}^n}$,
$ \sum_{Q\in{\mathcal F}_k}|\vec{A}_{k,Q}| \lesssim {\mathbf{1}}_{E_k} $.
Applying this, the fact that ${\mathbf{1}}_{E_k} \in L^q$,
and the Lebesgue dominated convergence theorem,
we conclude that $\sum_{Q\in{\mathcal F}_k} \vec{A}_{k,Q}$ converges in $L^q$.
Note that $\vec f=\vec b_{N}+\sum_{k=0}^{N-1}\sum_{Q\in{\mathcal F}_k} \vec{A}_{k,Q}$
and then, to show \eqref{Hardy dep eq 17},
we only need to prove that
$$\vec{b}_N = \sum_{L\in {\mathcal F}_N}\left[ \vec{f} - P_L\left( \vec{f} \right) \right]\eta_L \to 0$$
in $L^q$ as $N\to\infty$.
Using the properties {\rm (ii)} and {\rm (iv)} of ${\mathcal F}_k$ and
\eqref{Hardy dep eq 15}, we find that
\begin{align}\label{2036}
\left|\vec{b}_N(x)\right| \lesssim \left\|\vec{f}\right\|_{L^\infty} {\mathbf{1}}_{ E_{N-1} }(x).
\end{align}
Applying this and Fatou's lemma yields
$\vec{b}_N \to \vec{0} $ in $L^q$ as $N\to \infty$.
This finishes the proof of the claim \eqref{Hardy dep eq 17}.

Finally, we claim that, for any $k\in{\mathbb{Z}}_+$ and $Q\in {\mathcal F}_k$,
\begin{align}\label{Hardy dep eq 25}
\sup_{x\in 3Q} \left|\left[M^{(u)}_{3Q}\right]^{-1} \vec{A}_{k,Q}(x)\right| \lesssim 1.
\end{align}

We first consider the case $k\in{\mathbb{N}}$ of \eqref{Hardy dep eq 25}.
Observe that, for any $k\in{\mathbb{N}}$, $Q\in {\mathcal F}_k$, and $x\in{\mathbb{R}^n}$,
\begin{align}\label{A1234}
&\left|\left[M^{(u)}_{3Q}\right]^{-1}\vec{A}_{k,Q}(x)\right| {\nonumber}\\
&\quad = \Bigg|\left[M^{(u)}_{3Q}\right]^{-1}\Bigg\{\vec{f}(x) \eta_Q(x) - P_Q\left( \vec{f} \right)(x) \eta_Q(x) {\nonumber}\\
&\quad \quad - \sum_{L\in \mathcal G(Q)} \left[ \vec{f}(x) - P_L\left( \vec{f} \right)(x) \right]\eta_L(x) \eta_Q(x) + \sum_{L\in \mathcal G(Q)}\vec c_{Q,L}(x) \eta_L(x)\Bigg\}\Bigg| \nonumber\\
&\quad  \leq \left|\left[M^{(u)}_{3Q}\right]^{-1}\vec{f}(x) {\mathbf{1}}_{3Q\setminus E_Q}(x)\right| + \left|\left[M^{(u)}_{3Q}\right]^{-1} P_Q\left( \vec{f} \right)(x) \eta_Q(x) \right| \nonumber\\
&\quad \quad+ \sum_{L\in \mathcal G(Q)} \left|\left[M^{(u)}_{3Q}\right]^{-1} P_L\left( \vec{f} \right)(x)\eta_L(x) \eta_Q(x)\right|
+ \sum_{L\in \mathcal G(Q)} \left|\left[M^{(u)}_{3Q}\right]^{-1}\vec c_{Q,L}(x) \eta_L(x)\right|\nonumber\\
&\quad  =: I(x) + II(x) + III(x) + IV(x).
\end{align}

To proceed, we first show that,
for any $k\in\mathbb N$, $Q\in {\mathcal F}_{k}$, $L\in\mathcal G(Q)$,
and $i \in \{1,\dots,M\}$,
\begin{align}\label{Hardy dep eq 13}
\left|\left[M^{(u)}_{3L}\right]^{-1} \left\langle\vec{f}, e_i^{(L)}\widetilde{\eta}_{L}\right\rangle\right|\lesssim 1\ \ \text{and}\ \
\left|\left[M^{(u)}_{3Q}\right]^{-1} \left\langle\vec{f}, e_i^{(L)}\widetilde{\eta}_{L}\right\rangle\right|\lesssim 1,
\end{align}
where the implicit positive constants are independent of $\vec{f}$, $i$, $L$, and $Q$.
From Lemma \ref{stop P},
we deduce that, for any $y\in 3L$,
\begin{align*}
\left| \left[M^{(u)}_{3L}\right]^{-1} \left\langle \vec{f}, e_i^{(L)} \widetilde{\eta}_L \right\rangle\right|
\lesssim \left| \left[M^{(u)}_{3L}\right]^{-1} M_N^{{\mathcal K}} \left( \vec{f} \right)(y)\right|,
\end{align*}
which, together with \eqref{Hardy dep eq 19}, further implies that
\begin{align*}
\left| \left[M^{(u)}_{3L}\right]^{-1} \left\langle \vec{f}, e_i^{(L)} \widetilde{\eta}_L \right\rangle\right|
\lesssim \left\{ \fint_{3L} \left| \left[M^{(u)}_{3L}\right]^{-1} M_N^{{\mathcal K}} \left( \vec{f} \right)(y)\right|^u\,dy\right\}^\frac1u
\lesssim 1.
\end{align*}
Moreover, using the property {\rm (v)} of ${\mathcal F}_k$,
we obtain $32L\cap [3Q\setminus E_Q] \neq \emptyset$,
which further implies that there exists $y_L \in 32L\cap [3Q\setminus E_Q]$.
Applying this and Lemma \ref{stop P} with $y:=y_L$,
we conclude that
$\langle\vec{f}, e_i^{(L)}\widetilde{\eta}_{L}\rangle
\in C M^{{\mathcal K}}_N(\vec{f})(y_L)$,
which, combined with the definition of $E_Q$,
further implies that
\begin{align*}
\left|\left[M^{(u)}_{3Q}\right]^{-1} \left\langle\vec{f}, e_i^{(L)}\widetilde{\eta}_{L}\right\rangle\right|
\lesssim \left| \left[M^{(u)}_{3Q}\right]^{-1} M^{{\mathcal K}}_N\left(\vec{f}\right)(y_L)\right| \lesssim 1.
\end{align*}
This finishes the proof of \eqref{Hardy dep eq 13}.

Now, we estimate \eqref{A1234}.
Using the definition of $E_Q$ and Lemma \ref{f to}, we obtain, for any $x\in 3Q$
\begin{align}\label{Hardy dep eq 30}
I(x) \leq \left|\left[M^{(u)}_{3Q}\right]^{-1}\vec{f}(x) {\mathbf{1}}_{3Q\setminus E_Q}(x)\right| \lesssim 1.
\end{align}
By the definition of $P_L$, Lemmas \ref{stop eta} and \ref{stop e},
and \eqref{Hardy dep eq 13},
we find that, for any $x\in 3Q$,
\begin{align}\label{Hardy dep eq 29}
II(x)
\lesssim \sum_{i = 1}^M \left|\left[M^{(u)}_{3Q}\right]^{-1} \left\langle \vec{f}, e_i^{(Q)}\widetilde{\eta}_Q \right\rangle\right|\left| e_i^{(Q)}(x) \right|
\lesssim 1
\end{align}
and, moreover, for any $x\in 3L$ with $L\in \mathcal G(Q)$,
\begin{align}\label{Hardy dep eq 31}
\left|\left[M^{(u)}_{3Q}\right]^{-1} P_L\left( \vec{f} \right)\eta_L \eta_Q\right|
\lesssim \sum_{i = 1}^M \left|\left[M^{(u)}_{3Q}\right]^{-1} \left\langle \vec{f}, e_i^{(L)}\widetilde{\eta}_L \right\rangle\right| \left| e_i^{(L)}(x) \right|
\lesssim 1,
\end{align}
which, together with the property {\rm (iv)} of ${\mathcal F}_k$ and \eqref{Ccap},
further implies that, for any $x\in 3Q$,
\begin{align*}
III(x) \lesssim \sum_{L\in \mathcal G(Q)} {\mathbf{1}}_{L^\ast}(x) \lesssim 1.
\end{align*}
From Lemmas \ref{stop eta} and \ref{stop e} and the fact that $l(L) < \frac{1}{16}l(Q)$,
we infer that, for any $|\beta|\le N+1$,
\begin{align*}
\sup_{x\in{\mathbb{R}^n}} \left|\partial^\beta \left( e^{(L)}_i \widetilde{\eta}_{L}\eta_{Q} \right)(x)\right|
\sim \sup_{x\in{\mathbb{R}^n}} \left|\sum_{\alpha + \gamma \leq \beta}
\partial^\alpha e^{(L)}_i(x) \partial^{\gamma} \widetilde{\eta}_{L}(x) \partial^{\beta-\alpha-\gamma} \eta_{Q}(x) \right|
\lesssim  l(L)^{-n-|\beta|}
\end{align*}
and
\begin{align*}
\sup_{x\in{\mathbb{R}^n}} \left|\partial^\beta \left( e^{(L)}_i \widetilde{\eta}_{L}\right)(x)\right|
\sim \sup_{x\in{\mathbb{R}^n}} \left|\sum_{\alpha\leq \beta}
\partial^\alpha e^{(L)}_i(x) \partial^{\beta-\alpha} \widetilde{\eta}_{L}(x) \right|
\lesssim  l(L)^{-n-|\beta|}.
\end{align*}
Applying this, Lemma \ref{W f var},
and the same argument as that used in the proof of \eqref{Hardy dep eq 13},
we conclude that
\begin{align*}
\left| \left[M^{(u)}_{3Q}\right]^{-1} \left\langle \vec{f}, e^{(L)}_i \widetilde{\eta}_L \eta_Q \right\rangle \right| \lesssim 1
\quad\text{and}\quad\left| \left[M^{(u)}_{3Q}\right]^{-1} \left\langle \vec{f}, e^{(L)}_i \widetilde{\eta}_L \right\rangle \right| \lesssim 1,
\end{align*}
which further implies that
\begin{align*}
\left|\left[M^{(u)}_{3Q}\right]^{-1}\vec{c}_{Q,L} \eta_L\right|
&\leq \left| \left[M^{(u)}_{3Q}\right]^{-1}P_L \left( \vec{f}\eta_Q \right) \right|{\mathbf{1}}_{L^\ast}
+ \left| \left[M^{(u)}_{3Q}\right]^{-1}P_L \left( P_L\left( \vec{f} \right)\eta_Q \right) \right|{\mathbf{1}}_{L^\ast}
\lesssim {\mathbf{1}}_{L^\ast}.
\end{align*}
By this, the property {\rm (iv)} of ${\mathcal F}_k$, and \eqref{Ccap},
we conclude that, for any $x\in 3Q$,
\begin{align*}
IV(x) \lesssim \sum_{L\in\mathcal G(Q)} {\mathbf{1}}_{L^\ast}(x) \lesssim 1,
\end{align*}
which, combined with \eqref{A1234}, \eqref{Hardy dep eq 30}, \eqref{Hardy dep eq 29},
and \eqref{Hardy dep eq 31}, further implies that \eqref{Hardy dep eq 25} holds for any $k\in{\mathbb{N}}$.

Now, we consider the case $k = 0$ of \eqref{Hardy dep eq 25}.
Note that, by \eqref{Hardy dep eq 16}, ${\mathop\mathrm{\,supp\,}} \vec{f}\subset Q_0$,
and $L^\ast \subset 3Q_0$ for any $L\in \mathcal{G}(Q_0)$,
we have ${\mathop\mathrm{\,supp\,}} \vec{g}\subset Q_0^\ast$.
Using \eqref{Hardy dep eq 16}, \eqref{Hardy dep eq 13},
and \eqref{Hardy dep eq 30}, we obtain,
for any $x\in 3Q_0$,
\begin{align*}
\left| \left[M^{(u)}_{3Q_0}\right]^{-1} \vec{A}_{0,Q_0}(x) \right|
& \leq \left| \left[M^{(u)}_{3Q_0}\right]^{-1} \vec{f}(x) {\mathbf{1}}_{E_1^\complement}(x)\right|
+ \left| \left[M^{(u)}_{3Q_0}\right]^{-1} \sum_{L\in \mathcal G(Q_0)} P_L\left(\vec{f}\right)(x)\eta_L(x) \right|
\lesssim 1,
\end{align*}
which further implies that \eqref{Hardy dep eq 25} holds for $k = 0$
and hence completes the proof of the claim \eqref{Hardy dep eq 25}.

To finish the proof of (ii), for any $k\in {\mathbb{Z}_+}$ and $Q\in {\mathcal F}_k$, let
$$\lambda_{k,Q} := \left\|{\mathbf{1}}_{3Q}\right\|_{L^{p(\cdot)}} \left\|A_{3Q} M^{(u)}_{3Q}\right\|\ \ \text{and}\ \
\vec{a}_{k,Q} := \lambda_{k,Q}^{-1} \vec{A}_{k,Q},$$
and hence ${\mathop\mathrm{\,supp\,}} \vec{a}_{k,Q} \subset 3Q$ for any $k\in{\mathbb{Z}_+}$
and $Q\in {\mathcal F}_k$.
Then, by \eqref{Hardy dep eq 25}, we find that, for any $k\in{\mathbb{Z}}_+$ and $Q\in {\mathcal F}_k$,
\begin{align}\label{Hardy dep eq 28}
\sup_{x\in 3Q} \left| A_{3Q} \vec{a}_{k,Q} \right|
  &=\left\|{\mathbf{1}}_{3Q}\right\|_{L^{p(\cdot)}}^{-1} \left\|A_{3Q} M^{(u)}_{3Q}\right\|^{-1}
  \sup_{x\in 3Q} \left| A_{3Q} \vec{A}_{k,Q} \right| \nonumber\\
  &\leq \left\|{\mathbf{1}}_{3Q}\right\|_{L^{p(\cdot)}}^{-1}
  \sup_{x\in 3Q} \left| \left[ M^{(u)}_{3Q} \right]^{-1} \vec{A}_{k,Q} \right|
  \lesssim \left\|{\mathbf{1}}_{3Q}\right\|_{L^{p(\cdot)}}^{-1}.
\end{align}
In addition, observe that, by the vanishing moments of $\vec{f}$ and $\vec{b}_1$,
we have $\langle \vec{g},\cdot^\gamma \rangle = 0$ for any $\gamma\in {\mathbb{Z}}_+^n$ with $|\gamma|\leq s$.
Using this, the definition of $\vec{A}_{k,Q}$, Lemma \ref{stop e}, and \eqref{Hardy dep eq 26},
we conclude that, for any $k\in{\mathbb{Z}}_+$, any $Q\in {\mathcal F}_k$, and any polynomial $q\in {\mathscr P}_s$,
\begin{align*}
\left\langle \vec{A}_{k,Q}, q \right\rangle = 0.
\end{align*}
This, together with \eqref{Hardy dep eq 28}, further implies that
$\vec{a}_{k,Q}$ is a $(p(\cdot),\infty,s)_W$-atom supported in $3Q$.
Moreover, from \eqref{Hardy dep eq 17}, it follows that
$\vec{f} = \sum_{k\in{\mathbb{Z}_+}} \sum_{Q\in {\mathcal F}_k} \lambda_{k,Q}
\vec{a}_{k,Q}$
in $(\mathcal S')^m$.

Using Lemma \ref{fg Lp}, we obtain
\begin{align}\label{Hardy dep eq 24}
\left\| \sum_{k = 0}^\infty \sum_{Q\in {\mathcal F}_k} \left[\frac{|\lambda_{k,Q}|}{\|{\mathbf{1}}_{3Q}\|_{L^{p(\cdot)}}}\right]^{r} {\mathbf{1}}_{3Q} \right\|_{L^{\frac{p(\cdot)}{r}}}
&= \left\| \sum_{k = 0}^\infty \sum_{Q\in {\mathcal F}_k} \left\|A_{3Q} M^{(u)}_{3Q}\right\|^{r} {\mathbf{1}}_{3Q} \right\|_{L^{\frac{p(\cdot)}{r}}} {\nonumber} \\
& = \sup_{\|g\|_{L^{(\frac{p(\cdot)}{r})'}}\leq 1} \int_{{\mathbb{R}^n}} \sum_{k = 0}^\infty \sum_{Q\in {\mathcal F}_k} \left\|A_{3Q} M^{(u)}_{3Q}\right\|^{r} {\mathbf{1}}_{3Q}(x) g(x)\,dx.
\end{align}
Now, fix $g\in L^{(\frac{p(\cdot)}{r})'}$.
For any $k\in{\mathbb{Z}}_+$ and $Q\in {\mathcal F}_k$, let $F_Q := 3Q\setminus E_{k+1}$ and hence,
by the definition of $E_{k+1}$,
we find that, for any $L\in {\mathcal F}_{k+1}$, $F_Q \cap F_L = \emptyset$.
Observe that, from \eqref{Hardy dep eq 2},
we deduce that
$$|F_Q| = |3Q| - |3Q\cap E_{k+1}| > (1 - 3^{-n}2^{-4n})|3Q|.$$
Applying these, \eqref{Hardy dep eq 19}, and
Lemma \ref{norm matrix}, we conclude that
\begin{align*}
&\int_{{\mathbb{R}^n}} \sum_{k = 0}^\infty \sum_{Q\in {\mathcal F}_k} \left\|A_{3Q} M^{(u)}_{3Q}\right\|^{r} {\mathbf{1}}_{3Q}(x) g(x)\,dx\\
&\quad \le \sum_{k = 0}^\infty \sum_{Q\in {\mathcal F}_k} |3Q| \left\|A_{3Q} M^{(u)}_{3Q}\right\|^{r} \fint_{3Q}  g(x)\,dx
\lesssim \sum_{k = 0}^\infty \sum_{Q\in {\mathcal F}_k} |F_Q| \left\|A_{3Q} M^{(u)}_{3Q}\right\|^{r} \fint_{3Q}  g(x)\,dx\\
&\quad \lesssim \sum_{k = 0}^\infty \sum_{Q\in {\mathcal F}_k} \int_{F_Q} \left[\fint_{3Q} \left|A_{3Q} \widetilde{M}^{{\mathcal K}}_{N,\epsilon} \left(\vec{f}\right)(y) \right|^u\,dy\right]^{\frac ru}  {\mathcal M}\left(g\right)(x)  \,dx\\
&\quad \lesssim \sum_{k = 0}^\infty \sum_{Q\in {\mathcal F}_k} \int_{F_Q} \left[\widetilde{{\mathcal M}}^{(u)}_W\left(W(\cdot)\widetilde{M}^{{\mathcal K}}_{N,\epsilon} \left(\vec{f}\right)\right)(x)\right]^r  {\mathcal M}\left(g\right)(x)  \,dx\\
&\quad \lesssim \int_{{\mathbb{R}^n}} \left[\widetilde{{\mathcal M}}^{(u)}_W\left(W(\cdot)\widetilde{M}^{{\mathcal K}}_{N,\epsilon} \left(\vec{f}\right)\right)(x)\right]^r  {\mathcal M}\left(g\right)(x)  \,dx,
\end{align*}
which, combined with Lemmas \ref{Holder}, \ref{con f}, and \ref{bound max wid}, further implies that
\begin{align*}
&\int_{{\mathbb{R}^n}} \sum_{k = 0}^\infty \sum_{Q\in {\mathcal F}_k} \left\|A_{3Q} M^{(u)}_{3Q}\right\|^{r} {\mathbf{1}}_{3Q}(x) g(x)\,dx\\
&\quad \lesssim \left\|\left[\widetilde{{\mathcal M}}^{(u)}_W\left(W(\cdot)\widetilde{M}^{{\mathcal K}}_{N,\epsilon} \left(\vec{f}\right)\right)\right]^r\right\|_{L^{\frac{p(\cdot)}{r}}} \left\| {\mathcal M}\left(g\right)\right\|_{L^{(\frac{p(\cdot)}{r})'}}
\lesssim \left\|W(\cdot) \widetilde{M}^{{\mathcal K}}_{N,\epsilon} \left(\vec{f}\right)^r\right\|_{L^{\frac{p(\cdot)}{r}}} \left\|g\right\|_{L^{(\frac{p(\cdot)}{r})'}}.
\end{align*}
Using this, \eqref{Hardy dep eq 24}, and Lemma \ref{con f}, we find that
\begin{align*}
&\left\| \left\{\sum_{k = 0}^\infty \sum_{Q\in {\mathcal F}_k} \left[\frac{|\lambda_{k,Q}|}{\|{\mathbf{1}}_{3Q}\|_{L^{p(\cdot)}}}\right]^{r} {\mathbf{1}}_{3Q}\right\}^\frac1r \right\|_{L^{p(\cdot)}} \\
&\quad \lesssim \left\|W(\cdot) \widetilde{M}^{{\mathcal K}}_{N,\epsilon} \left(\vec{f}\right)\right\|_{L^{p(\cdot)}}
\lesssim \left\|W(\cdot) M^{{\mathcal K}}_N \left(\vec{f}\right)\right\|_{L^{p(\cdot)}} +  \left\| \frac{\epsilon\|W(\cdot)\|}{(1 + |\cdot|)^L} \right\|_{L^{p(\cdot)}},
\end{align*}
which, together with choosing $\epsilon$ small enough such that
$$\left\| \frac{\epsilon\|W(\cdot)\|}{(1 + |\cdot|)^L} \right\|_{L^{p(\cdot)}}
\le \left\|W(\cdot) M^{{\mathcal K}}_N \left(\vec{f}\right)\right\|_{L^{p(\cdot)}},
$$
further implies that
\begin{align*}
\left\| \left\{\sum_{k = 0}^\infty \sum_{Q\in {\mathcal F}_k} \left[\frac{|\lambda_{k,Q}|}{\|{\mathbf{1}}_{3Q}\|_{L^{p(\cdot)}}}\right]^{r} {\mathbf{1}}_{3Q}\right\}^\frac1r \right\|_{L^{p(\cdot)}}
\lesssim \left\|W(\cdot) M^{{\mathcal K}}_N \left(\vec{f}\right)\right\|_{L^{p(\cdot)}} = \left\|\vec{f}\right\|_{H^{p(\cdot)}_W}.
\end{align*}
From this, the fact that $\{\vec{a}_{k,Q}\}_{k\in\mathbb Z_+,
Q\in\mathcal F_k}$ are $(p(\cdot),\infty,s)_W$-atoms,
statement (i) of this theorem, and the uniqueness of the limit,
we infer that
$\vec{f} = \sum_{k\in{\mathbb{Z}_+}} \sum_{Q\in {\mathcal F}_k} \lambda_{k,Q}
\vec{a}_{k,Q}$
in $H^{p(\cdot)}_W$.
This finishes the proof of {\rm (ii)} and hence Theorem \ref{atom con}.
\end{proof}

\section{Dual of $H^{p(\cdot)}_W$}\label{sec dual}
In this section, we consider the dual space of $H^{p(\cdot)}_W$.
We first introduce the following reducing matrix-weighted variable Campanato space
$\mathcal{L}_{p(\cdot),q,s,{\mathbb{A}}}$.
\begin{definition}\label{def Com}
Let $p(\cdot) \in {\mathcal P}_0$, $q\in [1,\infty)$, $s\in{\mathbb{Z}}_+$,
and ${\mathbb{A}} := \{A_Q\}_{{\rm cube}\ Q}$ be a family of positive definite matrices.
The \emph{${\mathbb{A}}$-matrix-weighted variable Campanato space $\mathcal{L}_{p(\cdot),q,s,{\mathbb{A}}}$}
is defined to be the set of all $\vec{g} \in (L^{1}_{\rm loc})^m$ such that
\begin{align*}
\left\|\vec{g}\right\|_{ \mathcal{L}_{p(\cdot),q,s,{\mathbb{A}}} }
:= \sup_{Q} \inf_{\vec{P} \in ({\mathcal P}_s)^m} \frac{|Q|}{\| {\mathbf{1}}_Q \|_{L^{p(\cdot)}}}
\left\{\frac1{|Q|} \int_Q \left| A_Q^{-1} \left[ \vec{g}(x) - \vec{P}(x) \right] \right|^q\,dx \right\}^\frac1q < \infty.
\end{align*}
\end{definition}
Next, we establish an equivalent characterization of $ \mathcal{L}_{p(\cdot),q,s,{\mathbb{A}}} $.
Here, and thereafter, for any $s\in{\mathbb{Z}}_+$ and any compact set $E\subset {\mathbb{R}^n}$,
let
$ \Pi^s_E : L^1(E) \rightarrow {\mathcal P}_s $
be the \emph{natural projection} satisfying, for any $f\in L^1(E)$ and $q\in {\mathcal P}_s$,
\begin{align*}
\int_{E} \Pi^s_E(f)(x) q(x)\,dx = \int_E f(x)q(x)\,dx.
\end{align*}
In what follows, for any $\vec{f} := (f_1,\dots,f_m)$, we let
$ \Pi^s_E(\vec{f}) := (\Pi^s_E(f_1),\dots,\Pi^s_E(f_m)) $.
\begin{lemma}\label{Com eqiv}
Let all the notation be the same as in Definition \ref{def Com}.
Then, for any $\vec{g} \in \mathcal{L}_{p(\cdot),q,s,{\mathbb{A}}}$,
\begin{align*}
\left\|\vec{g}\right\|_{ \mathcal{L}_{p(\cdot),q,s,{\mathbb{A}}} }
\sim \sup_{Q} \frac{|Q|}{\| {\mathbf{1}}_Q \|_{L^{p(\cdot)}}}
\left\{ \frac1{|Q|}\int_Q \left| A_Q^{-1} \left[ \vec{g}(x) - \Pi^s_Q\left(\vec{g}\right) (x) \right] \right|^q\,dx \right\}^\frac1q,
\end{align*}
where the positive equivalence constants are independent of $\vec{g}$.
\end{lemma}
\begin{proof}
Using \cite[(2.12)]{jtyyz22}, we find that, for any cube $Q$ in ${\mathbb{R}^n}$,
\begin{align*}
\inf_{\vec{P} \in ({\mathcal P}_s)^m}  \left\{ \int_Q \left| A_Q^{-1} \left[ \vec{g}(x) - \vec{P}(x) \right] \right|^q\,dx \right\}^\frac1q
\sim \left\{ \int_Q \left| A_Q^{-1} \left[ \vec{g}(x) - \Pi^s_Q\left(\vec{g}\right) (x) \right] \right|^q\,dx \right\}^\frac1q,
\end{align*}
which completes the proof of Lemma \ref{Com eqiv}.
\end{proof}
The following dual theorem of $H^{p(\cdot)}_W$ is the main result of this section.
\begin{theorem}\label{dual Com}
Let $p(\cdot) \in {\mathcal P}_0\cap LH$ with $p_+ \leq 1$ and let $q\in [1,\infty)$, $r := \min\{1,p_-\},$
$s\in {\mathbb{Z}}_+ \cap [\lfloor d^{\rm upper}_{p(\cdot),\infty}(W) + n(\frac{1}{r} - 1)\rfloor,\infty)$,
$W\in {\mathscr A}_{p(\cdot),\infty}$, and $\{A_Q\}_{{\rm cube}\ Q}$ be a family of reducing operators of order $p(\cdot)$ for $W$.
Then the dual space of $H^{p(\cdot)}_W$, denoted by $(H^{p(\cdot)}_W)^\ast$,
is $ \mathcal{L}_{p(\cdot),q,s,{\mathbb{A}}} $ in the following sense:
\begin{itemize}
\item[{\rm (i)}] Suppose $\vec{g} \in \mathcal{L}_{p(\cdot),q,s,{\mathbb{A}}}$.
Then the linear functional
\begin{align}\label{def lg eq}
L_{\vec g} :\ \vec{f} \longmapsto L_{\vec g}\left( \vec{f} \right) := \int_{{\mathbb{R}^n}} \vec{f}(x)\cdot\vec{g}(x)\,dx,
\end{align}
initially defined for any $\vec{f} \in \mathcal{O}_s\cap H^{p(\cdot)}_W$, has a bounded linear extension to $H^{p(\cdot)}_W$.
\item[{\rm (ii)}] Conversely, any continuous linear functional on $H^{p(\cdot)}_W$
arises as in \eqref{def lg eq} with a unique $\vec{g} \in \mathcal{L}_{p(\cdot),q,s,{\mathbb{A}}}$.
\end{itemize}
\end{theorem}
\begin{remark}
If $p(\cdot)\equiv p$ with $p\in(0,1]$ is a constant exponent, then
Theorem \ref{dual Com} extends the duality result
\cite[Theorem 2.14]{cyy25} from matrix $A_p$ weights to matrix $A_{p,\infty}$
weights. On the other hand, for more general
$p(\cdot) \in {\mathcal P}_0\cap LH$ with $p_+ \leq 1$ and matrix
$\mathscr A_{p(\cdot),\infty}$ weights $W$, Theorem \ref{dual Com} is new.
\end{remark}
To prove this theorem, we need the following technical lemma.

\begin{lemma}\label{La}
Let $p(\cdot) \in {\mathcal P}_0\cap LH$ with $p_+ \leq 1$ and let $q\in (\max\{1, \frac{r_W p_+}{r_W - 1}\},\infty)$,
$s\in {\mathbb{Z}}_+ \cap [\lfloor d^{\rm upper}_{p(\cdot),\infty}(W) + n(\frac{1}{r} - 1)\rfloor,\infty)$,
and $W\in {\mathscr A}_{p(\cdot),\infty}$.
Then, for any continuous linear functional $L \in (H^{p(\cdot)}_W)^\ast$,
$$ \sup \left\{  \left| L\left(\vec{f}\right) \right|:\ \left\| \vec{f} \right\|_{H^{p(\cdot)}_W} \leq 1 \right\}
\sim \sup\left\{ \left| L\left( \vec{a} \right)\right|:\ \vec{a}\ \text{is a}\ (p(\cdot),q,s)_W\text{-atom}  \right\}, $$
where the positive equivalence constants are independent of $L$.
\end{lemma}
\begin{proof}
By Theorem \ref{atom con}(i), we find that, for any $(p(\cdot),q,s)_W$-atom $\vec{a}$,
$\|\vec{a}\|_{H^{p(\cdot)}_W} \lesssim 1$
and hence
$$ \sup \left\{  \left| L\left(\vec{f}\right) \right|:\ \left\| \vec{f} \right\|_{H^{p(\cdot)}_W} \leq 1 \right\}
\gtrsim \sup\left\{ \left| L\left( \vec{a} \right)\right|:\ \vec{a}\ \text{is a}\ (p(\cdot),q,s)_W\text{-atom} \right\}. $$
Next, we consider the converse inequality.
Let $\vec{f} \in \mathcal{O}_s \cap H^{p(\cdot)}_W$ with $\|\vec{f}\|_{H^{p(\cdot)}_W} \leq 1$.

Then, from Theorem \ref{atom con}{\rm (ii)} and \cite[(4.7) and (4.10)]{ns12}, we deduce that there exist a sequence $\{\lambda_k\}_{k\in{\mathbb{Z}}}$ in $\mathbb{C}$
and a sequence of $(p(\cdot),\infty,s)_W$-atoms $\{\vec{a}_k\}_{k\in{\mathbb{Z}}}$ such that
$\vec{f} = \sum_{k\in{\mathbb{Z}}} \lambda_k\vec{a}_k$ in both $H^{p(\cdot)}_W$ and $({\mathcal S}')^m$
and
$$ \sum_{k\in{\mathbb{Z}}} |\lambda_k| \lesssim \left\|\left\{\sum_{k\in{\mathbb{Z}}} \left[\frac{|\lambda_{k}|}{\|  {\mathbf{1}}_{Q_k} \|_{L^{p(\cdot)}} } \right]^r {\mathbf{1}}_{Q_k} \right\}^\frac1r \right\|_{L^{p(\cdot)}} \lesssim 1, $$
where $r = \min\{1,p_-\}$.
Applying this and the continuity of $L$,
we conclude that
\begin{align*}
\left| L\left( \vec{f} \right) \right| &\leq \left( \sum_{k\in{\mathbb{Z}}} |\lambda_k|  \left| L \left( \vec{a}_k \right)\right| \right)
\lesssim \sup\left\{ \left| L\left( \vec{a} \right)\right|:\ \vec{a}\ \text{is a}\ (p(\cdot),q,s)_W\text{-atom} \right\},
\end{align*}
which, combined with Proposition \ref{H dense O},
further implies that
\begin{align*}
\sup \left\{  \left| L\left(\vec{f}\right) \right|:\ \left\| \vec{f} \right\|_{H^{p(\cdot)}_W} \leq 1 \right\}
\lesssim \sup\left\{ \left| L\left( \vec{a} \right)\right|:\ \vec{a}\ \text{is a}\ (p(\cdot),q,s)_W\text{-atom} \right\}.
\end{align*}
This finishes the proof of Lemma \ref{La}.
\end{proof}
Now, we give the proof of Theorem \ref{dual Com}.
\begin{proof}[Proof of Theorem \ref{dual Com}]
By an argument similar to that used in the proof of
\cite[Theorem 2.14(ii)]{cyy25}
with
\cite[Lemma 2.21]{cyy25} replaced by Lemma \ref{La},
we prove (ii).

Now, we prove {\rm (i)}.
Let $\vec{f} \in \mathcal{O}_s \cap H^{p(\cdot)}_W$.
Then, from the proof of Theorem \ref{atom con}{\rm (ii)},
we infer that there exist a sequence of collections of cubes
${\mathcal F} = \bigcup_{k\in{\mathbb{Z}}_+} {\mathcal F}_k$
and the corresponding sequences $\{\lambda_Q\}_{Q\in{\mathcal F}}$ in $\mathbb C$
and $\{\vec{a}_Q\}_{Q\in{\mathcal F}}$ of $(p(\cdot),\infty,s)_W$-atoms
such that, for any $N\in{\mathbb{N}}$,
$$\vec{f} = \sum_{Q\in{\mathcal F}} \lambda_Q \vec{a}_Q
= \sum_{k = 0}^\infty \sum_{Q\in{\mathcal F}_k} \lambda_Q \vec{a}_Q
= \sum_{k = 0}^N \sum_{Q\in{\mathcal F}_k} \lambda_Q\vec{a}_Q + \vec{b}_{N+1}, $$
where $\vec{b}_{N+1} := \sum_{L\in{\mathcal F}_{N+1}} \vec{b}_L$.
Note that in the proof of Theorem \ref{atom con}{\rm (ii)} we obtain
$|E_{N}| \leq 2^{-4n N} |Q_0|,$
which, together with \eqref{2036}
and $\vec{g} \in\mathcal{L}_{p(\cdot),q,s,{\mathbb{A}}}\subset (L^{1}_{\rm loc})^m$,
further implies that
\begin{align*}
\int_{{\mathbb{R}^n}} \vec{b}_{N+1}(x)\cdot \vec{g}(x) \,dx \to 0
\end{align*}
as $N\to \infty$.
Using this, \eqref{AkQ}, and also $\vec{g} \in(L^{1}_{\rm loc})^m$, we find that
\begin{align}\label{dual com eq 1}
\int_{{\mathbb{R}^n}}  \vec{f}(x)\cdot \vec{g}(x) \,dx
&= \lim\limits_{N\to \infty} \int_{{\mathbb{R}^n}} \sum_{k = 0}^N \sum_{Q\in {\mathcal F}_k} \lambda_Q \vec{a}_Q(x)\cdot\vec g(x) \,dx{\nonumber} \\
&= \lim\limits_{N\to \infty} \sum_{k = 0}^N \sum_{Q\in {\mathcal F}_k} \lambda_Q \int_{{\mathbb{R}^n}}  \vec{a}_Q(x)\cdot \vec{g}(x)\,dx.
\end{align}
Observe that, by the definitions of $ \|\cdot\|_{\mathcal{L}_{p(\cdot),q,s,{\mathbb{A}}}} $
and $(p(\cdot),q',s)_W$-atoms,
for any $Q\in {\mathcal F}$,
\begin{align*}
\left| \int_{{\mathbb{R}^n}}  \vec{a}_Q(x)\cdot \vec{g}(x)\,dx \right|
& = \inf_{\vec{P} \in ({\mathcal P}_s)^m} \left| \int_{Q}  \vec{a}_Q\cdot \left[\vec{g}(x) - \vec{P}(x)\right] \,dx\right| \\
& \leq \left\| A_Q \vec{a}_Q \right\|_{L^{q'}} \inf_{\vec{P} \in ({\mathcal P}_s)^m}
\left\{\int_{Q} \left|A_Q^{-1} \left[ \vec{g}(x) - \vec{P}(x) \right]\right|^q \,dx \right\}^\frac1q\\
& \leq |Q|^{\frac{1}{q'}} \left\| {\mathbf{1}}_{Q} \right\|_{L^{p(\cdot)}}^{-1} \inf_{\vec{P} \in ({\mathcal P}_s)^m}
\left\{\int_{Q} \left|A_Q^{-1} \left[ \vec{g}(x) - \vec{P}(x) \right]\right|^q \,dx \right\}^\frac1q
 \leq \left\|\vec{g}\right\|_{ \mathcal{L}_{p(\cdot),q,s,{\mathbb{A}}} }.
\end{align*}
Applying this and \eqref{dual com eq 1}, we conclude that
\begin{align*}
\left| \int_{{\mathbb{R}^n}} \vec{f}(x)\cdot\vec{g}(x)\ \,dx \right|
&\leq \lim\limits_{N\to \infty} \sum_{k = 0}^N \sum_{Q\in {\mathcal F}_k} \left|\lambda_Q\right| \left|\int_{{\mathbb{R}^n}} \vec{a}_Q(x)\cdot \vec{g}(x)\,dx\right|\\
&\leq \left\|\vec{g}\right\|_{ \mathcal{L}_{p(\cdot),q,s,{\mathbb{A}}} }\lim\limits_{N\to \infty} \sum_{k = 0}^N \sum_{Q\in {\mathcal F}_k} \left|\lambda_Q\right|
 = \left\|\vec{g}\right\|_{ \mathcal{L}_{p(\cdot),q,s,{\mathbb{A}}} } \sum_{k\in{\mathbb{Z}}_+}\sum_{Q\in {\mathcal F}_k} \left|\lambda_Q\right|,
\end{align*}
which, combined with \cite[(4.7) and (4.10)]{ns12},
further implies that
\begin{align*}
\left| \int_{{\mathbb{R}^n}} \vec{f}(x)\cdot \vec{g}(x) \,dx \right|
\lesssim \left\|\vec{g}\right\|_{ \mathcal{L}_{p(\cdot),q,s,{\mathbb{A}}} } \left\|\sum_{k\in{\mathbb{Z}}} \left[\frac{|\lambda_{k}|}{\|  {\mathbf{1}}_{Q_k} \|_{L^{p(\cdot)}} } \right]^r {\mathbf{1}}_{Q_k} \right\|_{L^{\frac{p(\cdot)}{r}}}^{\frac1r}
\lesssim \left\|\vec{g}\right\|_{ \mathcal{L}_{p(\cdot),q,s,{\mathbb{A}}} } \left\| \vec{f} \right\|_{H^{p(\cdot)}_W},
\end{align*}
where $r = \min\{1,p_-\}$.
From this and Proposition \ref{H dense O}, we deduce that {\rm (i)} holds.
This finishes the proof of Theorem \ref{dual Com}.
\end{proof}

\section{Calder\'on--Zygmund Operators}\label{sec CZ}
In this section, we establish the boundedness of Calder\'on--Zygmund operators on $H^{p(\cdot)}_W$.
We first present the concept of the $s$-order standard kernel
(see, for instance, \cite[Chapter III]{s93}).
In what follows, for any $\gamma := (\gamma_1,\dots,\gamma_n) \in {\mathbb{Z}}_+^n$ and
any $\gamma$-order differentiable function $K(\cdot,\cdot)$ on $(x,y)\in{\mathbb{R}^n}\times {\mathbb{R}^n}$, let
\begin{align*}
\partial^\gamma_{(1)} K(x,y) := \frac{\partial^{|\gamma|}}{\partial x_1^{\gamma_1}\cdots \partial x_n^{\gamma_n}}K(x,y)
\ \ \text{and}\ \
\partial^\gamma_{(2)} K(x,y) := \frac{\partial^{|\gamma|}}{\partial y_1^{\gamma_1}\cdots \partial y_n^{\gamma_n}}K(x,y).
\end{align*}
\begin{definition}
Let $s\in{\mathbb{Z}}_+$ and $\delta \in (0,1]$. A measurable function $K$ on ${\mathbb{R}^n}\times{\mathbb{R}^n}\setminus \{ (x,x):\ x\in{\mathbb{R}^n} \}$
is called an \emph{$(s,\delta)$-type standard kernel} if there exists a positive constant $C$ such that,
for any $\gamma\in {\mathbb{Z}}_+^n$ with $|\gamma|\leq s$, the followings hold.
\begin{itemize}
\item[{\rm (i)}] For any $x,y\in{\mathbb{R}^n}$ with $x\neq y$,
\begin{align*}
\left|\partial^\gamma_{(1)} K(x,y)\right| \leq \frac{C}{|x-y|^{n + |\gamma|}}
\ \ \text{and}\ \
\left|\partial^\gamma_{(2)} K(x,y)\right| \leq \frac{C}{|x-y|^{n + |\gamma|}}.
\end{align*}
\item[{\rm (ii)}] For any $x,y,z\in{\mathbb{R}^n}$ with $x\neq y$ and $|x-y|\geq 2|y-z|$,
\begin{align*}
\left|\partial^\gamma_{(1)} K(y,x) - \partial^\gamma_{(1)} K(z,x)\right| \leq C\frac{|y - z|^\delta}{|x-y|^{n + |\gamma|+ \delta}}
\end{align*}
and
\begin{align*}
\left|\partial^\gamma_{(2)} K(x,y) - \partial^\gamma_{(2)} K(x,z)\right| \leq C\frac{|y - z|^\delta}{|x-y|^{n + |\gamma|+ \delta}}.
\end{align*}
\end{itemize}
\end{definition}
Next, we present the definition of Calder\'on--Zygmund operators.
\begin{definition}\label{def CZ}
Let $s\in {\mathbb{Z}}_+$ and $\delta\in (0,1]$. A linear operator $T$ is called an
\emph{$(s,\delta)$-type Calder\'on--Zygmund operator} if $T$ is bounded on $L^2$ and there exists
an $(s,\delta)$-type standard kernel $K$ such that, for any given $f\in L^2$ and for almost every $x\in{\mathbb{R}^n}$,
$$T(f)(x) := \lim_{\eta \to 0^+} T_\eta(f)(x),$$ where, for any $\eta\in (0,\infty)$,
\begin{align*}
T_\eta(f)(x) := \int_{{\mathbb{R}^n}\setminus B(x,\eta)} K(x,y)f(y)\,dy.
\end{align*}
\end{definition}
\begin{remark}
Let $T$ be the same as in Definition \ref{def CZ}. By \cite[p.\,102]{d01},
we find that, for any $q\in (1,\infty)$, $T$ is bounded on $L^q$ and, for any $f\in L^q$,
$T(f) = \lim_{\eta\to 0^+} T_\eta(f)$ both almost everywhere on ${\mathbb{R}^n}$ and in $L^q$.
\end{remark}
Next, we recall the concept of the well-known vanishing moments on $T$ (see, for instance, \cite[p.\,23]{mc97}).
\begin{definition}
Let $s\in{\mathbb{Z}}_+$ and $\delta\in (0,1]$.
An $(s,\delta)$-type Calder\'on--Zygmund operator $T$ is said to have the \emph{vanishing moments up to order $s$}
if, for any function $a\in L^2$ having compact support and satisfying that, for any $\gamma\in {\mathbb{Z}}_+^n$
with $|\gamma| \leq s$, $\int_{{\mathbb{R}^n}} x^\gamma a(x)\,dx = 0$, it holds that $\int_{{\mathbb{R}^n}}x^\gamma T(a)(x)\,dx = 0$.
\end{definition}
\begin{theorem}\label{bound CZ}
Let $p(\cdot) \in {\mathcal P}_0\cap LH$, $W\in {\mathscr A}_{p(\cdot),\infty}$,
$\delta\in (0,1]$, $r:=\min\{1,p_-\}$, and $s\in {\mathbb{Z}}_+ \cap (d^{\rm upper}_{p(\cdot),\infty}(W) + n(\frac{1}{r} - 1)-\delta,\infty)$.
Let $T$ be an $(s,\delta)$-type Calder\'on--Zygmund operator.
Then there exists a linear operator $\widetilde{T}_1$ bounded from $H^{p(\cdot)}_W$ to $L^{p(\cdot)}_W$
that agrees with $T$ on $ \mathcal{O}_s\cap H^{p(\cdot)}_W $.
If, in addition, $T$ has the vanishing moments up to order $s$, then there exists a linear operator $\widetilde{T}_2$
bounded from $H^{p(\cdot)}_W$ to $H^{p(\cdot)}_W$ that agrees with $T$ on $ \mathcal{O}_s \cap H^{p(\cdot)}_W $.
\end{theorem}

\begin{remark}
If $p(\cdot)\equiv p$ with $p\in(0,1]$ is a constant exponent, then
Theorem \ref{bound CZ} extends the boundedness result
\cite[Theorem 5.5]{bcyy24} from matrix $A_p$ weights to matrix $A_{p,\infty}$
weights. On the other hand, for more general $p(\cdot) \in {\mathcal P}_0\cap LH$ and matrix
$\mathscr A_{p(\cdot),\infty}$ weights, Theorem \ref{bound CZ} is new.
\end{remark}
\begin{proof}[Proof of Theorem \ref{bound CZ}]
We first consider the boundedness of $T$ from $H^{p(\cdot)}_W$ to $L^{p(\cdot)}_W$.
Let $\{A_Q\}_{{\rm cube}\ Q}$ be a family of reducing operators of order $p(\cdot)$ for $W$.
Assume that $\vec{a}_Q$ is a $(p(\cdot),\infty,s)_W$-atom
supported in cube $Q$ and let $r_Q := \frac{\sqrt{n}}{2}l(Q)$.
By the vanishing moments of $\vec{a}_Q$,
we have, for any $x\in B(c_Q,4r_Q)^\complement$,
\begin{align*}
\left|W(x) T \vec{a}_Q (x)\right|
& = \left|\int_{{\mathbb{R}^n}} W(x) K(x,y) \vec{a}_Q(y) \,dy\right| {\nonumber}\\
& = \left| \int_{Q} W(x) \left[ K(x,y) - \sum_{|\beta|\leq s} \frac{\partial^\beta_y K(x, c_Q)}{\beta!} \left( y - c_Q \right)^\beta \right] \vec{a}_Q(y) \,dy\right| {\nonumber}\\
& \leq \left\|W(x)A_Q^{-1}\right\| \int_{Q} \left| K(x,y) - \sum_{|\beta|\leq s} \frac{\partial^\beta_y K(x, c_Q)}{\beta!} \left( y - c_Q \right)^\beta \right| \left|A_Q \vec{a}_Q(y)\right| \,dy {\nonumber}\\
& = \left\|W(x)A_Q^{-1}\right\| \int_{Q} \left| \sum_{|\beta| = s} \frac{\partial^\beta_y K(x, c_Q) - \partial^\beta_yK(x,\xi_y)}{\beta!} \left( y - c_Q \right)^\beta \right| \left|A_Q \vec{a}_Q(y)\right| \,dy,
\end{align*}
which, together with the definition of $(s,\delta)$-kernels, further implies that
\begin{align*}
\left|W(x) T \vec{a}_Q (x)\right|
&\lesssim \left\|W(x)A_Q^{-1}\right\| \int_{Q} \left|A_Q \vec{a}_Q(y)\right| \frac{|y - c_Q|^{s+\delta}}{|x - c_Q|^{n + s+\delta}} \,dy\\
&\lesssim \left\|W(x)A_Q^{-1}\right\| \left[\frac{l(Q)}{|x - c_Q|}\right]^{n + s+\delta} \fint_{Q} \left|A_Q \vec{a}_Q(y)\right|\,dy.
\end{align*}
Using this and the linearity of $T$,
we obtain, for any $x\in{\mathbb{R}^n}$,
\begin{align*}
\left|W(x) T\vec{a}_Q(x)\right|
&\lesssim \left\|W(x) A_Q^{-1}\right\|  \left| T\left( A_Q \vec{a}_Q \right)(x) \right| {\mathbf{1}}_{B(c_Q,4r_Q)}(x){\nonumber}\\
&\quad + \left\|W(x)A_Q^{-1}\right\|\left[\frac{l(Q)}{l(Q) + |x - c_Q|}
\right]^{n+s+\delta} \fint_{Q} \left|A_Q \vec{a}_Q(y)\right|\,dy,
\end{align*}
which, combined with the size condition of $(p(\cdot),\infty,s)_W$-atoms,
further implies that
\begin{align}\label{0923}
\left|W(x) T\vec{a}_Q(x)\right|
&\lesssim \left\|W(x) A_Q^{-1}\right\|  \left| T\left( A_Q \vec{a}_Q \right)(x) \right| {\mathbf{1}}_{B(c_Q,4r_Q)}(x){\nonumber}\\
&\quad + \left\|W(x)A_Q^{-1}\right\|\left[\frac{l(Q)}{l(Q) + |x - c_Q|}\right]^{n+s+\delta} \left\|{\mathbf{1}}_Q\right\|_{L^{p(\cdot)}}^{-1}.
\end{align}
Fix $\vec{f}\in \mathcal{O}_s \cap H^{p(\cdot)}_W$ supported in $Q_0:=Q(c_0,l_0)$.
By Theorem \ref{atom con} and \eqref{Hardy dep eq 17}, we find that there exist
a sequence $\{\lambda_k\}_{k\in{\mathbb{Z}}}$ in $\mathbb C$ and a sequence of $(p(\cdot),\infty,s)_W$-atoms $\{\vec{a}_k\}_{k\in{\mathbb{Z}}}$
supported, respectively, in $\{Q_k\}_{k\in{\mathbb{Z}}}$ with $Q_k := Q(c_k,l_k)$ such that
\begin{align}\label{Hp CZ eq 5}
\vec{f} = \sum_{k\in{\mathbb{Z}}} \lambda_k \vec{a}_k
\end{align}
in $L^q$ for any $q\in(0,\infty)$ and
\begin{align}\label{Hp CZ eq 7}
\left\| \left\{ \sum_{k\in{\mathbb{Z}}}  \left[\frac{|\lambda_k|}{\| {\mathbf{1}}_{Q_k} \|_{L^{p(\cdot)}}}\right]^r {\mathbf{1}}_{Q_k} \right\}^\frac1r \right\|_{L^{p(\cdot)}} \lesssim \left\| \vec{f} \right\|_{H^{p(\cdot)}_W}.
\end{align}
For any $k\in\mathbb Z$, let $r_k:=\frac{\sqrt n}{2}l_k$. From \eqref{Hp CZ eq 5}, the boundedness of $T$ on $L^q$ for any $q\in(1,\infty)$, and \eqref{0923},
we infer that
\begin{align}\label{Hp CZ eq 12}
\left\| T\left(\vec{f}\right) \right\|_{L^{p(\cdot)}_W}
&\le\left\|\sum_{k\in{\mathbb{Z}}}|\lambda_k|T\vec a_k\right\|_{L^{p(\cdot)}_W}{\nonumber}\\
&\lesssim \left\|\sum_{k\in{\mathbb{Z}}} |\lambda_k|\left\|W(\cdot) A_{Q_k}^{-1}\right\|  \left| T\left( A_{Q_k} \vec{a}_{k} \right) \right| {\mathbf{1}}_{B(c_k,4r_k)}\right\|_{L^{p(\cdot)}}{\nonumber}\\
&\quad + \left\|\sum_{k\in{\mathbb{Z}}} |\lambda_k|\left\|W(\cdot)A_{Q_k}^{-1}\right\|\left[\frac{l_k}{l_k + |\cdot - c_k|}\right]^{n+s+\delta} \left\|{\mathbf{1}}_{Q_k}\right\|_{L^{p(\cdot)}}^{-1}\right\|_{L^{p(\cdot)}}.
\end{align}
By Lemma \ref{a atom eq le}, the boundedness of $T$ on $L^q$ for any $q\in(1,\infty)$,
and the definition of $(p(\cdot),\infty,s)_W$-atoms,
we find that
\begin{align*}
&\left\| \sum_{k\in{\mathbb{Z}}}|\lambda_k| \left\|W(\cdot) A_{Q_k}^{-1}\right\| \left| T\left( A_{Q_k} \vec{a}_k \right) \right| {\mathbf{1}}_{B(c_k,4r_k)}\right\|_{L^{p(\cdot)}} \\
&\quad \lesssim \left\| \left\{ \sum_{k\in{\mathbb{Z}}}  \left(|\lambda_k| |Q_k|^{-\frac1q} \| A_{Q_k}\vec{a}_k \|_{L^{q}}\right)^r {\mathbf{1}}_{Q_k} \right\}^\frac1r \right\|_{L^{p(\cdot)}}
\leq \left\| \left\{ \sum_{k\in{\mathbb{Z}}}  \left[\frac{|\lambda_k|}{\| {\mathbf{1}}_{Q_k} \|_{L^{p(\cdot)}}}\right]^r {\mathbf{1}}_{Q_k} \right\}^\frac1r \right\|_{L^{p(\cdot)}}.
\end{align*}
Applying this, \eqref{Hp CZ eq 12}, and Lemma \ref{f dec eq le},
we conclude that
\begin{align*}
\left\| T\left(\vec{f}\right) \right\|_{L^{p(\cdot)}_W} \lesssim \left\| \left\{ \sum_{k\in{\mathbb{Z}}}  \left[\frac{|\lambda_k|}{\| {\mathbf{1}}_{Q_k} \|_{L^{p(\cdot)}}}\right]^r {\mathbf{1}}_{Q_k} \right\}^\frac1r \right\|_{L^{p(\cdot)}}
\lesssim \left\| \vec{f} \right\|_{H^{p(\cdot)}_W}.
\end{align*}
Using this and Proposition \ref{H dense O}, we find that
$T$ can be extended to a bounded linear operator $\widetilde{T}_1$ from $H^{p(\cdot)}_W$ to $L^{p(\cdot)}_W$.

Next, we prove the boundedness of $T$ from $H^{p(\cdot)}_W$ to $H^{p(\cdot)}_W$.
Let $\psi\in {\mathcal S}$ satisfy ${\mathop\mathrm{\,supp\,}} \psi\subset B(\mathbf{0},1)$ and $\int_{{\mathbb{R}^n}} \psi(x)\,dx \neq 0$.
Let $\vec a_Q$ be a $(p(\cdot),\infty,s)_W$-atom supported in $Q:= Q(c_Q, l(Q))$.
Then we claim that, for any $x\in{\mathbb{R}^n}$,
\begin{align}\label{Hp CZ eq 4}
M_W\left( T\vec{a}_Q,\psi \right)(x)
&\lesssim \left\|W(x) A_Q^{-1}\right\| {\mathcal M}\left( \left| T\left( A_Q \vec{a}_Q \right) \right| \right)(x) {\mathbf{1}}_{B(c_Q,4r_Q)}(x){\nonumber}\\
&\quad + \left\|W(x)A_Q^{-1}\right\|\left[\frac{l(Q)}{l(Q) + |x - c_Q|}\right]^{n+s+\delta} \left\|{\mathbf{1}}_Q\right\|^{-1}_{L^{p(\cdot)}},
\end{align}
where $r_Q := \frac{\sqrt{n}}{2} l(Q)$.
By the assumption that $T$ is bounded on $L^2$ and the size condition of $(p(\cdot),\infty,s)_W$-atoms,
we conclude that
\begin{align}\label{Hp CZ eq 1}
\left\| T\left(A_Q \vec{a}_Q \right) \right\|_{L^2} \lesssim \left\| A_Q \vec{a}_Q \right\|_{L^2}
\leq |Q|^{\frac12} \left\|{\mathbf{1}}_Q\right\|^{-1}_{L^{p(\cdot)}}.
\end{align}
Using the vanishing moments of $T$ and $\vec a_Q$, we obtain, for any $t\in (0,\infty)$ and $x\in B(c_Q, 4r_Q)^\complement$,
\begin{align}\label{Hp CZ eq 2}
&\left|W(x)\psi_t\ast\left[ T\left( \vec{a}_Q \right) \right](x)\right| {\nonumber}\\
&\quad \leq t^{-n} \int_{{\mathbb{R}^n}} \left| \psi\left( \frac{x-y}{t} \right) W(x) T\left( \vec{a}_Q \right)(y) \right|\,dy {\nonumber}\\
&\quad \leq \left\|W(x)A_Q^{-1}\right\|t^{-n}  \int_{{\mathbb{R}^n}} \left| \psi\left( \frac{x-y}{t} \right) T\left(A_Q \vec{a}_Q \right)(y) \right|\,dy {\nonumber}\\
&\quad \leq \left\|W(x)A_Q^{-1}\right\|t^{-n}  \int_{{\mathbb{R}^n}} \left| \psi\left( \frac{x-y}{t} \right) - \sum_{|\beta|\leq s} \frac{\partial^\beta \psi(\frac{x- c_Q}{t})}{\beta!} \left( \frac{y - c_Q}{t} \right)^\beta \right|
\left|T\left(A_Q \vec{a}_Q \right)(y) \right|\,dy {\nonumber}\\
&\quad \leq \left\|W(x)A_Q^{-1}\right\|t^{-n}  \left(\int_{|y - c_Q| < 2r_Q} + \int_{2r_Q \leq |y - c_Q| < \frac{|x-c_Q|}{2}} + \int_{|y - c_Q| \geq \frac{|x-c_Q|}{2}}\right) {\nonumber}\\
&\quad \quad \times \left| \psi\left( \frac{x-y}{t} \right) - \sum_{|\beta|\leq s} \frac{\partial^\beta \psi(\frac{x- c_Q}{t})}{\beta!} \left( \frac{y - c_Q}{t} \right)^\beta \right|
\left|T\left(A_Q \vec{a}_Q \right)(y) \right|\,dy {\nonumber}\\
&\quad =: \left\|W(x)A_Q^{-1}\right\|\left( {\rm I}_1 + {\rm I}_2 + {\rm I}_3 \right).
\end{align}
Applying this and \eqref{Hp CZ eq 1} and repeating the estimation of \cite[(6.14), (6.15), and (6.16)]{zyyw21}
with $a_j$ and $\gamma$ replaced, respectively, by $A_Q \vec{a}_Q$ and $s+\delta$,
we find that
\begin{align*}
{\rm I}_1 \lesssim \frac{[l(Q)]^{s+1}}{|x - c_Q|^{n+s+1}} \left\| T\left(A_Q \vec{a}_Q \right) \right\|_{L^2} |Q|^\frac12
\lesssim \left[\frac{l(Q)}{|x - c_Q|}\right]^{n+s+1} \left\|{\mathbf{1}}_Q\right\|^{-1}_{L^{p(\cdot)}},
\end{align*}
\begin{align*}
{\rm I}_2 &\lesssim \frac{[l(Q)]^{s+\delta}}{|x - c_Q|^{n+s+1}}
\int_{2r_Q \leq |y - c_Q| < \frac{|x-c_Q|}{2}} \frac{1}{|y - c_Q|^{n+\delta}} \,dy
\left\| T\left(A_Q \vec{a}_Q \right) \right\|_{L^2} |Q|^\frac12 \\
&\lesssim \left[\frac{l(Q)}{|x - c_Q|}\right]^{n+s+\delta} \left\|{\mathbf{1}}_Q\right\|^{-1}_{L^{p(\cdot)}},
\end{align*}
and
\begin{align*}
{\rm I}_3 & \lesssim \left\| T\left(A_Q \vec{a}_Q \right) \right\|_{L^2} |Q|^\frac12
\left\{  \frac{[l(Q)]^{s+\delta}}{|x - c_Q|^{n+s+\delta}} \int_{|y - c_Q| \geq \frac{|x-c_Q|}{2}} \left| \psi_t(x-y) \right|\,dy\right. \\
&\quad + \left.\sum_{|\beta|\leq s} \left[ l(Q) \right]^{s + \delta} \int_{|y-c_Q|\geq \frac{|x-c_Q|}{2}}
\frac{1}{|x - c_Q|^{n+|\beta|}} \frac{1}{|y-c_Q|^{n+s+\delta-|\beta|}}\,dy \right\}\\
&\lesssim \left[\frac{l(Q)}{|x - c_Q|}\right]^{n+s+\delta} \left\|{\mathbf{1}}_Q\right\|^{-1}_{L^{p(\cdot)}}.
\end{align*}
These, together with \eqref{Hp CZ eq 2}, further implies that, for any $x\in B(c_Q,4r_Q)^{\complement}$,
\begin{align}\label{Hp CZ eq 3}
M_W\left( T\vec{a}_Q,\psi \right)(x)
&\lesssim \left\|W(x)A_Q^{-1}\right\|\left[\frac{l(Q)}{|x - c_Q|}\right]^{n+s+\delta} \left\|{\mathbf{1}}_Q\right\|^{-1}_{L^{p(\cdot)}}{\nonumber} \\
&\lesssim \left\|W(x)A_Q^{-1}\right\|\left[\frac{l(Q)}{l(Q) + |x - c_Q|}\right]^{n+s+\delta} \left\|{\mathbf{1}}_Q\right\|^{-1}_{L^{p(\cdot)}} .
\end{align}
From \cite[Corollary 2.1.12]{g14} and $\psi\in\mathcal S$, it follows that,
for any $x\in B(c_Q,4r_Q)$,
\begin{align*}
M_W\left( T\vec{a}_Q,\psi \right)(x)
\le\left\|W(x) A_Q^{-1}\right\|\sup_{t\in(0,\infty)}
\left|\psi_t*T(A_Q\vec a_Q)\right|
\lesssim \left\|W(x) A_Q^{-1}\right\| {\mathcal M}\left( \left| T\left( A_Q \vec{a}_Q \right) \right| \right),
\end{align*}
which, combined with \eqref{Hp CZ eq 3}, further implies that \eqref{Hp CZ eq 4} holds.

Now, fix $\vec{f}\in \mathcal{O}_s \cap H^{p(\cdot)}_W$.
Similarly to the proof of the boundedness of $T$ from $H^{p(\cdot)}_W$ to $L^{p(\cdot)}_W$,
there exist a sequence $\{\lambda_k\}_{k\in{\mathbb{Z}}}$ in $\mathbb C$
and a sequence of $(p(\cdot),\infty,s)_W$-atoms $\{\vec{a}_k\}_{k\in{\mathbb{Z}}}$
satisfying \eqref{Hp CZ eq 5} and \eqref{Hp CZ eq 7}
and, moreover, $\vec{f} := \sum_{k\in{\mathbb{Z}}} \lambda_k \vec{a}_k$
in $L^q$ for any $q\in (0,\infty)$.
Using \eqref{Hp CZ eq 4} and the linearity of $T$,
we find that
\begin{align}\label{Hp CZ eq 6}
\left\| T\vec f \right\|_{H^{p(\cdot)}_W}
&\lesssim \left\| \sum_{k\in\mathbb Z}|\lambda_k| M_W \left(T\vec{a}_k,\psi \right) \right\|_{L^{p(\cdot)}}{\nonumber} \\
&\lesssim \left\| \sum_{k\in\mathbb Z}|\lambda_k| \left\|W(\cdot) A_{Q_k}^{-1}\right\| {\mathcal M}\left( \left| T\left( A_{Q_k} \vec{a}_k \right) \right| \right) {\mathbf{1}}_{B(c_k,4r_k)}\right\|_{L^{p(\cdot)}}{\nonumber} \\
&\quad +\left\| \sum_{k\in\mathbb Z}\frac{|\lambda_k|}{\|{\mathbf{1}}_{Q_k}\|_{L^{p(\cdot)}}} \left\|W(\cdot)A_{Q_k}^{-1}\right\|\left[\frac{l_k}{l_k + |\cdot - c_k|}\right]^{n+s+\delta} \right\|_{L^{p(\cdot)}}.
\end{align}
Then, by Lemma \ref{a atom eq le} with $T$ replaced by ${\mathcal M} \circ T$ and the boundedness of ${\mathcal M}$ and $T$ in $L^q$
for any $q\in (1,\infty)$ and by the definition of $(p(\cdot),q,s)_W$-atoms, we obtain
\begin{align*}
\left\| \sum_{k\in\mathbb Z}|\lambda_k| \left\|W(\cdot) A_{Q_k}^{-1}\right\| {\mathcal M}\left( \left| T\left( A_{Q_k} \vec{a}_k \right) \right| \right) {\mathbf{1}}_{B(c_k,4r_k)}\right\|_{L^{p(\cdot)}}
\lesssim \left\| \left\{ \sum_{k\in\mathbb Z} \left[\frac{|\lambda_k|}{\| {\mathbf{1}}_{Q_k} \|_{L^{p(\cdot)}}}\right]^r {\mathbf{1}}_{Q_k} \right\}^\frac1r \right\|_{L^{p(\cdot)}}.
\end{align*}
This, together with \eqref{Hp CZ eq 6} and Lemma \ref{f dec eq le},
further implies that
\begin{align*}
\left\| T\vec f \right\|_{H^{p(\cdot)}_W}
&\lesssim \left\| \left\{ \sum_{k\in\mathbb Z}  \left[\frac{|\lambda_k|}{\| {\mathbf{1}}_{Q_k} \|_{L^{p(\cdot)}}}\right]^r {\mathbf{1}}_{Q_k} \right\}^\frac1r \right\|_{L^{p(\cdot)}}.
\end{align*}
Using this and \eqref{Hp CZ eq 7}, we conclude that
\begin{align*}
\left\| T\left(\vec{f}\right) \right\|_{H^{p(\cdot)}_W}
\lesssim \left\| \vec{f} \right\|_{H^{p(\cdot)}_W},
\end{align*}
which, combined with Proposition \ref{H dense O}, further implies that
$T$ can be extended to a bounded linear operator $\widetilde{T}_2$ from $H^{p(\cdot)}_W$ to $H^{p(\cdot)}_W$.
This finishes the proof of Theorem \ref{bound CZ}.
\end{proof}

\noindent\textbf{Funding}\quad This project is partially supported by the National
Natural Science Foundation of China (Grant Nos. 12431006 and 12371093),
the Beijing Natural Science Foundation
(Grant No. 1262011),
and the Fundamental Research Funds
for the Central Universities (Grant No. 2253200028).

\bigskip

\noindent\textbf{Data availability}\quad There is no data.

\bigskip

\noindent\textbf{Code availability}\quad Not applicable.

\section*{Declarations}

\noindent\textbf{Conflict of interest}\quad There is no conflict of interest.

\bigskip

\noindent Yiqun Chen, Dachun Yang(Corresponding author), Wen Yuan and
Zongze Zeng

\medskip

\noindent Laboratory of Mathematics and Complex Systems (Ministry of Education of China),
School of Mathematical Sciences, Institute for Advanced Study,
Beijing Normal University, Beijing 100875,
The People's Republic of China

\smallskip

\noindent{\it E-mails:} \texttt{yiqunchen@mail.bnu.edu.cn} (Y. Chen)

\noindent\phantom{{\it E-mails:} }\texttt{dcyang@bnu.edu.cn} (D. Yang)

\noindent\phantom{{\it E-mails:} }\texttt{wenyuan@bnu.edu.cn} (W. Yuan)

\noindent\phantom{{\it E-mails:} }\texttt{zzzeng@mail.bnu.edu.cn} (Z. Zeng)

\end{document}